\documentclass[11pt,
]{amsart}
\usepackage{
amssymb, graphicx
}
\usepackage{mathrsfs} 
\usepackage{amsthm}
\usepackage{bm}
\usepackage{graphicx,color}
\usepackage[dvipsnames]{xcolor}
\usepackage{amsmath}
\usepackage{enumitem}
\allowdisplaybreaks[4]
\newtheorem{theorem}{Theorem}
\newtheorem{lemma}[theorem]{Lemma}
\newtheorem{proposition}[theorem]{Proposition}
\newtheorem{definition}{Definition}

\newtheorem{remark}{Remark}
\theoremstyle{remark}

\date{\today}

\title[Determination of hyperelastic parameters]{Determination of hyperelastic parameters with non-trapping wavespeeds}

      \author[G. Uhlmann]{Gunther Uhlmann}
\address{G. Uhlmann: Department of Mathematics, University of Washington, Seattle, WA 98195, USA (\tt{gunther@math.washington.edu})
}
  \author[J. Zhai]{Jian Zhai}
\address{School of Mathematical Sciences,
  Fudan University, Shanghai 200433, China; 
Center for Applied Mathematics,
  Fudan University, Shanghai 200433, China}
 \email{jianzhai@fudan.edu.cn}
   \thanks{J. Zhai is supported by National Key Research and Development Programs of China (No. 2023YFA1009103), NSFC(No. 12471396), Science and Technology Commission of Shanghai Municipality (23JC1400501)}

\begin{document}
\begin{abstract}
We consider the elastic wave equation in three-dimensional isotropic medium with up to quadratic nonlinear terms. We show that five out of the six elastic parameters can be uniquely determined by boundary measurements under the assumption that the wavespeeds are non-trapping, while the last one can also be determined in a generic setting.
\end{abstract}
\maketitle

\section{Introduction}
Assume $u=(u_1,u_2,u_3)^T$ is the displacement vector of a three-dimensional isotropic elastic body, then the (right) Cauchy-Green strain tensor is 
\[
\varepsilon(u)=\frac{1}{2}(\nabla u+\nabla^Tu+\nabla^Tu\nabla u).
\]
Componentwisely
\[
\varepsilon_{ij}=\frac{1}{2}\left(\frac{\partial u_i}{\partial x_j}+\frac{\partial u_j}{\partial x_i}+\sum_{m=1}^3\frac{\partial u_m}{\partial x_i}\frac{\partial u_m}{\partial x_j}\right).
\]
For hyperelastic material, there exists a stored energy function $H(\nabla u)$, representing the potential energy in the Hamiltonian. The Piola-Kirchhoff stress tensor is then given by 
\[
\sigma=\frac{\partial H}{\partial (\nabla u)},
\]
or componentwisely
\[
\sigma_{ij}=\frac{\partial H}{\partial (\partial_j u_i)}.
\]
 Then the motion of displacement is characterized by a nonlinear system
\begin{equation}\label{elasticwaveequation}
\rho\partial^2_tu-\nabla\cdot\sigma=0,
\end{equation}
that is
\[
\rho\partial_t^2u_i-\sum_{j=1}^3\partial_j\sigma_{ij}=0,
\]
where $\rho>0$ is the density of mass. See \cite{ciarlet2021mathematical,gurtin1981topics} for more details. 

For isotropic material, it is shown \cite{ogden1997non} that $H$ depends only on the principal invariants of $\varepsilon$, here in terms of
\[
\begin{split}
&\kappa_1=2\mathrm{trace}(\varepsilon),\\
&\kappa_2=2(\mathrm{trace}(\varepsilon)^2-\mathrm{trace}(\varepsilon^2)),\\
&\kappa_3=\frac{4}{3}\left(\mathrm{trace}(\varepsilon)^3-3\mathrm{trace}(\varepsilon)\mathrm{trace}(\varepsilon^2)+2\mathrm{trace}(\varepsilon^3)\right).
\end{split}
\]
 We will use the Taylor expansions with respect to small displacement $u$, and truncate the equation \eqref{elasticwaveequation} up to the second order terms in $u$. Therefore, relevant terms in $H$ are
\[
H=H_0+\sigma_1\kappa_1+\sigma_2\kappa_2+\sigma_3\kappa_3+\frac{1}{2}\sigma_{11}\kappa_1^2+\sigma_{12}\kappa_1\kappa_2+\frac{1}{6}\sigma_{111}\kappa_1^3+\mathrm{l.o.t.}.
\]
Denote $G:=\nabla u$, i.e., $G_{ij}=\partial_j u_i$, and $\nabla^Tu=G^T$. Then we can write
\[
\begin{split}
\kappa_1=&2\mathrm{trace}(G)+\mathrm{trace}(G^TG),\\
\kappa_2=&2\mathrm{trace}(G)^2+2\mathrm{trace}(G)\mathrm{trace}(G^TG)-\mathrm{trace}(G^2)\\
&-\mathrm{trace}(G^TG)-2\mathrm{trace}(GG^TG)+\mathrm{l.o.t.},\\
\kappa_3=&\frac{4}{3}\mathrm{trace}(G)^3-2\mathrm{trace}(G)\mathrm{trace}(G^2)-2\mathrm{trace}(G)\mathrm{trace}(G^TG)\\
&+\frac{2}{3}\mathrm{trace}(G^3)+2\mathrm{trace}(GG^TG)+\mathrm{l.o.t.}.
\end{split}
\]
Here we have used the fact that $\mathrm{trace}(AB)=\mathrm{trace}(BA)$ and $\mathrm{trace}(A)=\mathrm{trace}(A^T)$.
Note that
\[
\begin{split}
&\frac{\partial\mathrm{trace}(G)}{\partial G}=I,\\
&\frac{\partial\mathrm{trace}(G^TG)}{\partial G}=2G,\\
&\frac{\partial\mathrm{trace}(G^2)}{\partial G}=2G^T,\\
&\frac{\partial\mathrm{trace}(G^3)}{\partial G}=3(G^T)^2,\\
&\frac{\partial\mathrm{trace}(GG^TG)}{\partial G}=G^2+GG^T+G^TG.
\end{split}
\]
We impose the condition $\sigma_1=0$, which means that the reference configuration is stress-free. Then it follows by tedious calculation that
\[
\begin{split}
\frac{\partial H}{\partial(\nabla u)}=&4(\sigma_{11}+\sigma_2)(\nabla\cdot u)I-2\sigma_2(\nabla u+\nabla^Tu)+4(\sigma_{111}+3\sigma_{12}+\sigma_3)(\nabla\cdot u)^2 I\\
&+2(\sigma_{11}-\sigma_{12}+\sigma_2-\sigma_3)\left(2(\nabla\cdot u)\nabla u+\mathrm{trace}(\nabla^T u\nabla u)I\right)\\
&-2(\sigma_{12}+\sigma_3)(2(\nabla\cdot u)\nabla^T u+\mathrm{trace}((\nabla u)^2)I)\\
&-2(\sigma_2-\sigma_3)((\nabla u)^2+\nabla u\nabla^Tu+\nabla^Tu\nabla u)+2\sigma_3(\nabla^T u)^2+\mathrm{l.o.t.}.
\end{split}
\]
See also \cite{sideris1996null,agemi2000global} for the calculation. We remark here that \cite{sideris1996null} started with the left strain tensor, in contrast with the right strain tensor adopted here. But one will end up with exactly the same stress tensor in either way.

Denote
\[
\lambda=4(\sigma_{11}+\sigma_2),\quad \mu=-2\sigma_2,
\]
which are known as the Lam\'e parameters, and
\[
\begin{split}
&\mathscr{A}=8\sigma_3,\\
&\mathscr{B}=-4(\sigma_{12}+\sigma_3),\\
&\mathscr{C}=4(\sigma_{111}+3\sigma_{12}+\sigma_3).
\end{split}
\]
We impose the condition (which is called the ``\textit{very strong ellipticity condition}" in various literature)
\begin{equation}\label{ellipticitycondition}
\mu>0,\quad 3\lambda+2\mu>0.
\end{equation}
The stress tensor is then of the form
\[
\begin{split}
\sigma(u)=&\lambda\mathrm{trace}(\varepsilon)I+2\mu\varepsilon+\lambda(\nabla\cdot u)\nabla u+2\mu\nabla u\widetilde{\varepsilon}\\
&+\mathscr{A}\tilde{\varepsilon}^2+\mathscr{B}(2(\nabla\cdot u)\widetilde{\varepsilon}+\mathrm{trace}(\tilde{\varepsilon}^2)I)+\mathscr{C}(\nabla\cdot u)^2I+\mathrm{l.o.t.},
\end{split}
\]
where 
\[
\widetilde{\varepsilon}=\frac{1}{2}(\nabla u+\nabla^Tu)
\]
is the linearized strain tensor. For the rest of the paper we ignore the lower order terms in $\sigma$. The components of $\sigma$ are
\begin{equation}\label{stressform}
\begin{split}
\sigma_{ij}=&\lambda \varepsilon_{mm}\delta_{ij}+\lambda\widetilde{\varepsilon}_{mm}\frac{\partial u_i}{\partial x_j}+2\mu\left(\varepsilon_{ij}+\widetilde{\varepsilon}_{jn}\frac{\partial u_i}{\partial x_n}\right)\\
&+ \mathscr{A}\widetilde{\varepsilon}_{in}\widetilde{\varepsilon}_{jn}+\mathscr{B}(2\widetilde{\varepsilon}_{nn}\widetilde{\varepsilon}_{ij}+\widetilde{\varepsilon}_{mn}\widetilde{\varepsilon}_{mn}\delta_{ij})+\mathscr{C}\widetilde{\varepsilon}_{mm}\widetilde{\varepsilon}_{nn}\delta_{ij},
\end{split}
\end{equation}
where we have used Einstein's convention that all repeated indices are summed. This model is widely used in physical literature, cf., for example, \cite{landau1960theory,gol1961interaction,de2003finite}.
For mathematical study of this equation, we refer to \cite{sideris1996null,sideris2000nonresonance,agemi2000global}. \\

Consider the initial boundary value problem
\begin{equation}\label{elastic_eq}
\begin{cases}
\rho\partial^2_tu-\nabla\cdot \sigma(u)=0,\quad &(t,x)\in (0,T)\times\Omega,\\
u(t,x)=f(t,x),\quad &(t,x)\in (0,T)\times\partial \Omega,\\
u(0,x)=\frac{\partial}{\partial t}u(0,x)=0,\quad &x\in \Omega.
\end{cases}
\end{equation}
We introduce the displacement-to-traction map
\[
\Lambda:f\mapsto\nu\cdot\sigma(u)\vert_{(0,T)\times\partial\Omega},
\]
where $u$ is the solution to \eqref{elastic_eq}. It is proved in \cite{de2018nonlinear} that, for an arbitrary $m\geq 5$, $\Lambda$ is well defined as a map
\[
C_c^m((0,T)\times\partial\Omega)\rightarrow \bigcap_{k=0}^mC^k([0,T];W^{m-k,2}(\Omega)),
\]
where $\|f\|_{C^m}\leq \varepsilon_0$ with some small $\varepsilon_0>0$ (depending on $m,T$).
In this paper we will investigate the inverse problem of determining $\lambda,\mu,\rho,\mathscr{A},\mathscr{B},\mathscr{C}\in C^\infty(\overline{\Omega})$ from $\Lambda$.\\

Replacing $\sigma$ with the linearized stress tensor
\[
\widetilde{\sigma}(u)=\lambda\mathrm{trace}(\widetilde{\varepsilon})I+2\mu\widetilde{\varepsilon}=\lambda(\nabla\cdot u)I+\mu(\nabla u+\nabla^Tu),
\]
that is, discarding the second order terms in $u$,
one ends up with the classical linear elastic wave equation in isotropic medium
\begin{equation}\label{elastic_lineareq}
\begin{cases}
\rho\partial_t^2u-\nabla\cdot \widetilde{\sigma}(u)=0,\quad &(t,x)\in (0,T)\times\Omega,\\
u(t,x)=f(t,x),\quad &(t,x)\in (0,T)\times\partial \Omega,\\
u(0,x)=\frac{\partial}{\partial t}u(0,x)=0,\quad &x\in \Omega.
\end{cases}
\end{equation}

The linearized displacement-to-traction (Dirichlet-to-Neumann) map is defined as
\[
\Lambda^\mathrm{lin}:f\mapsto \nu\cdot\widetilde{\sigma}(u)\vert_{(0,T)\times\partial\Omega},
\]
where $u$ is the solution to \eqref{elastic_lineareq}.

For the linear elastic wave equation, there are two wavespeeds, namely, $P$-wavespeed $c_P$ and $S$-wavespeed $c_S$, related to the material parameters in the following way
\[
c_P=\sqrt{\frac{\lambda+2\mu}{\rho}},\quad c_S=\sqrt{\frac{\mu}{\rho}}.
\]
By \eqref{ellipticitycondition}, we have $c_P>c_S$.
Each wavespeed induces a Riemannian metric on $\Omega$
\[
g_\bullet=c_\bullet^{-2}\mathrm{d}s^2,
\]
with $\bullet=P/S$ and $\mathrm{d}s^2$ is the Euclidean metric. 

Let us review some basic notations in Riemannian geometry here. Let $(M,g)$ be a compact, connected and oriented Riemannian manifold with smooth boundary $\partial M$. Denote 
\[
SM=\{(x,v)\vert x\in M, v\in T_xM,|v|_g=1\}
\]
 to be the unit sphere bundle of $M$. 
Given $(x,v)\in SM$, the geodesic $\gamma_{x,v}$ is the unique geodesic determined by $(x,v)$ such that $\gamma_{x,v}(0)=x$ and $\dot{\gamma}_{x,v}(0)=v$. Denote
\[
\tau(x,v)=\inf\{t\vert \gamma_{x,v}(t)\in\partial M\}\in [0,\infty]
\]
to be the exit time when the geodesic $\gamma_{x,v}$ exits $M$. A Riemannian manifold $(M,g)$  is called \textit{non-trapping} if $\tau(x,v)<+\infty$ for all $(x,v)\in SM$. Denote then
\[
\mathrm{diam}(M,g)=\sup\{\tau(x,v)\vert (x,v)\in SM\}.
\]
One can also use the same notations for $(x,v)\in T^*M$ by taking a musical diffeomorphism of $v$. \\

The main result of this paper is as follows.
\begin{theorem}\label{maintheorem}
Assume $(\Omega,g_\bullet)$ is non-trapping and $\partial\Omega$ is convex with respect to $g_\bullet$ for either $\bullet=P,S$ and $T>2\max\{\mathrm{diam}(\Omega,g_P),\mathrm{diam}(\Omega,g_S)\}$. Suppose in addition $\lambda+3\mu+\mathscr{A}+2\mathscr{B}\neq 0$ in $\overline{\Omega}$. Then $\Lambda$ uniquely determines $\lambda,\mu,\rho,\mathscr{A},\mathscr{B}$ in $\overline{\Omega}$. Moreover, $\Lambda$ uniquely determines $\mathscr{C}$ in the set $\overline{\{\nabla(\mu^2/\rho)\neq 0\}}$.
\end{theorem}

Note that the parameters $\lambda,\mu,\rho$ already appear in the linear elastic wave equation. The recovery of these three parameters from the linearized Dirichlet-to-Neumann $\Lambda^\mathrm{lin}$ has been considered in \cite{rachele2000inverse,rachele2003uniqueness,stefanov2017local,bhattacharyya2018local,zhai2026determination}. However those results require stronger geometrical conditions. Generally speaking, one can reduce this inverse boundary value problem to certain geometrical inverse problems such as \textit{lens rigidity} and \textit{tensor tomography} problems, for which \textit{simplicity} or \textit{foliation condition} is usually assumed \cite{michel1981rigidite,paternain2015invariant,stefanov2016boundary,stefanov2018inverting,uhlmann2024invertibility2}. It is well known that either of these two conditions implies non-trapping. To be more precise, it is proved \cite{stefanov2017local,zhai2026determination} that all the three parameters can be determined by $\Lambda^\mathrm{lin}$ under the foliation condition. Under the simplicity condition, the uniqueness of $\frac{\lambda}{\rho}$ and $\frac{\mu}{\rho}$ is proved \cite{rachele2000inverse}, while the uniqueness of $\rho$ requires some further curvature condition \cite{rachele2003uniqueness}. In a recent work \cite{yi2026leakage}, the author proves the unique determination of one wavespeed from a so-called \textit{solution map} if the other one is a priori known, under the non-trapping condition and the additional condition $\nabla\rho\neq 0$.

On the other hand, the inverse problem for the nonlinear elastic wave equation has been considered in \cite{de2018nonlinear,uhlmann2021inverse,uhlmann2024determination}. However, in these works the wavespeeds are still recovered from the linearized DtN map, thus requiring the more restrictive geometrical conditions. When the simplicity or foliation condition is satisfied, all the six parameters $\lambda,\mu,\rho,\mathscr{A},\mathscr{B},\mathscr{C}$ are proved to be uniquely determined  by $\Lambda$ \cite{uhlmann2024determination}. One of the major endeavors of this paper is to use the nonlinearity to determine the wavespeeds under the less restrictive non-trapping condition. The uniqueness of other parameters can also be established. The determination of $\rho,\mathscr{A},\mathscr{B}$ is not so different than the approach in \cite{uhlmann2024determination}. The determination of parameter $\mathscr{C}$ is however a little bit trickier. Under stronger geometrical conditions, the uniqueness of $\mathscr{C}$ is established in  \cite{uhlmann2021inverse}. Under the mere non-trapping condition, we have one more restriction as stated in Theorem \ref{maintheorem}. But physically this restriction is not so troubling since for homogeneous medium all parameters should be simultaneously constant.

Since the work \cite{kurylev2018inverse}, rapid progress has been witnessed on inverse problems for nonlinear hyperbolic equations. Among the massive recent literature, we refer to \cite{lassas2018inverse,feizmohammadi2021inverse,hintz2022dirichlet,chen2021detection,kurylev2022inverse,uhlmann2020determination,balehowsky2022inverse,hintz2024inverse,sa2024recovery,alexakis2024inverse,nursultanov2025determining,chen2025inverse,chen2025stable,uhlmann2026inverse} and the references therein. In most of these works, nonlinearity is used in an essential way to tackle the inverse problems.  Waves have rich behaviors due to their nonlinear interactions, and thus carry plentiful information that can be beneficial for inverse problems. The most notable property for elastic waves, in contrast with other models considered extensively, is that there are two types of waves, namely $P$- and $S$- waves, traveling in different speeds. This property will be heavily utilized in our argument.\\

The organization of the paper is as follows. In Section \ref{secondorderlinearization}, we carry out the second order linearization and obtain an integral identity needed for the inverse problem. In Section \ref{gaussian}, we detail the construction of Gaussian beam solutions for the linear elastic wave equation. In Section \ref{gaussianwithoutreflection}, we carefully analyze the reflection of elastic Gaussian beams on the boundary subject to the Dirichlet boundary condition, and outline a scattering control scheme to remove the reflection. Then, as some preparation, we use the Gaussian beams (without reflection) to prove that the linear DtN map gives us the lens relations for both $P$- and $S$- waves in Section \ref{linearequationresult}. Before we prove the uniqueness of the wavespeeds, we first deal with an associated geometrical inverse problem in Section \ref{brokenscattering}. The geometrical inverse problem is a variant of the broken scattering rigidity problem. The recovery of the $P$-wavespeed can be reduced to this geometrical inverse problem, as will be shown in Section \ref{determinationPwavespeed}. Once the reduction is done, one can recover the $P$-wavespeed by applying the result for the related geometrical inverse problem. With the already determined $P$-wavespeed, we then determine the $S$-wavespeed in Section \ref{determinationSwavespeed}, via recovering part of an $S$-wave trajectory in spacetime. In Section \ref{uniquenessnonlinearparameters}, we establish the uniqueness of the density $\rho$ and two nonlinear parameters $\mathscr{A},\mathscr{B}$, the proof of which is not so different than that in \cite{uhlmann2024determination}. Finally in Section \ref{recoveryC}, we consider the recovery of $\mathscr{C}$.

\section{Second-order linearization}\label{secondorderlinearization}
Linearization with multiple small parameters has been proved to be quite useful for solving inverse problems for nonlinear partial differential equations, see, for example, \cite{kurylev2018inverse,lassas2018inverse,hintz2022inverse}. We aim to recover coefficients of up to second order nonlinear terms, so we will only use first and second order linearizations of the displacement-to-traction map to be physically more meaningful. Higher order linearizations would involve higher order nonlinear terms in the equation itself, which have been ignored here.
In this section we summarize the second order linearization of $\Lambda$, which has already been carried out in \cite{uhlmann2021inverse,uhlmann2024determination}. \\

Take $\epsilon_1,\epsilon_2\in\mathbb{R}$ small enough, and let $u_\epsilon$ be the solution to the initial boundary value problem \eqref{elastic_eq} with boundary value $f=\epsilon_1f^{(1)}+\epsilon_2f^{(2)}$. Here $f^{(1)},f^{(2)}$ vanish near $\{t=0\}$. Then $u_\epsilon$ has the asymptotic expansion
\[
u_\epsilon=\epsilon_1u^{(1)}+\epsilon_2u^{(2)}+\frac{1}{2}\epsilon_1^2u^{(11)}+\frac{1}{2}\epsilon_2^2u^{(22)}+\epsilon_1\epsilon_2u^{(12)}+\text{higher order terms in }\epsilon_1,\epsilon_2.
\]
Here $u^{(j)},j=1,2$ is the solution to the linearized equation
\begin{equation}\label{linear_eqj}
\begin{cases}
\rho\frac{\partial^2u^{(j)}}{\partial t^2}-\nabla\cdot \widetilde{\sigma}(u^{(j)})=0,\quad &(t,x)\in (0,T)\times\Omega,\\
u^{(j)}(t,x)=f^{(j)}(t,x),\quad &(t,x)\in (0,T)\times\partial \Omega,\\
u^{(j)}(0,x)=\frac{\partial}{\partial t}u^{(j)}(0,x)=0,\quad &x\in \Omega,
\end{cases}
\end{equation}
and $u^{(12)}$ is the solution to the equation
\begin{equation}\label{linear_eq12}
\begin{cases}
\rho\frac{\partial^2u^{(12)}}{\partial t^2}-\nabla\cdot \widetilde{\sigma}(u^{(12)})=\nabla\cdot G(u^{(1)},u^{(2)}),\quad &(t,x)\in (0,T)\times\Omega,\\
u^{(12)}(t,x)=0,\quad &(t,x)\in (0,T)\times\partial \Omega,\\
u^{(12)}(0,x)=\frac{\partial}{\partial t}u^{(12)}(0,x)=0,\quad &x\in \Omega,
\end{cases}
\end{equation}
where $G_{ij}$ comes from the second order term of $u$ in $\sigma(u)$, having the form
\begin{align}\label{G_form}
G_{ij}(u^{(1)},u^{(2)})=&(\lambda+\mathscr{B})\frac{\partial u^{(1)}_m}{\partial x_n}\frac{\partial u^{(2)}_m}{\partial x_n}\delta_{ij}+2\mathscr{C}\frac{\partial u^{(1)}_m}{\partial x_m}\frac{\partial u^{(2)}_n}{\partial x_n}\delta_{ij}+\mathscr{B}\frac{\partial u_m^{(1)}}{\partial x_n}\frac{\partial u_n^{(2)}}{\partial x_m}\delta_{ij}\nonumber\\
&+\mathscr{B}\left(\frac{\partial u^{(1)}_m}{\partial x_m}\frac{\partial u^{(2)}_j}{\partial x_i}+\frac{\partial u^{(2)}_m}{\partial x_m}\frac{\partial u^{(1)}_j}{\partial x_i}\right)+\frac{\mathscr{A}}{4}\left(\frac{\partial u^{(1)}_j}{\partial x_m}\frac{\partial u^{(2)}_m}{\partial x_i}+\frac{\partial u^{(2)}_j}{\partial x_m}\frac{\partial u^{(1)}_m}{\partial x_i}\right)\nonumber\\
&+(\lambda+\mathscr{B})\left(\frac{\partial u^{(1)}_m}{\partial x_m}\frac{\partial u^{(2)}_i}{\partial x_j}+\frac{\partial u^{(2)}_m}{\partial x_m}\frac{\partial u^{(1)}_i}{\partial x_j}\right)\\
&+(\mu+\frac{\mathscr{A}}{4})\left(\frac{\partial u^{(1)}_m}{\partial x_i}\frac{\partial u^{(2)}_m}{\partial x_j}+\frac{\partial u^{(2)}_m}{\partial x_i}\frac{\partial u^{(1)}_m}{\partial x_j}+\frac{\partial u^{(1)}_i}{\partial x_m}\frac{\partial u^{(2)}_j}{\partial x_m}+\frac{\partial u^{(2)}_i}{\partial x_m}\frac{\partial u^{(1)}_j}{\partial x_m}\right.\nonumber\\
&\quad\quad\left.+\frac{\partial u^{(1)}_i}{\partial x_m}\frac{\partial u^{(2)}_m}{\partial x_j}+\frac{\partial u^{(2)}_i}{\partial x_m}\frac{\partial u^{(1)}_m}{\partial x_j}\right).\nonumber
\end{align}

We define
 the bilinear map
\[
\begin{split}
\Lambda^{''}:(f^{(1)},f^{(2)})&\mapsto \left(\nu\cdot \widetilde{\sigma}(u^{(12)})+\nu\cdot G(u^{(1)},u^{(2)})\right)\Big\vert_{(0,T)\times\partial\Omega}.\\
\end{split}
\]
One can verify that formally 
\[
\begin{split}
\Lambda^\mathrm{lin}(f^{(j)})&=\frac{\partial}{\partial\epsilon_j}\Lambda(\epsilon_jf^{(j)})\vert_{\epsilon_j=0},\\
\Lambda^{''}(f^{(1)},f^{(2)})&=\frac{\partial^2}{\partial\epsilon_1\partial\epsilon_2}\Lambda(\epsilon_1f^{(1)}+\epsilon_2f^{(2)})\vert_{\epsilon_1=\epsilon_2=0}.
\end{split}
\]
The above linearizations can be rigorously justified using standard theory of hyperbolic equations.
Therefore, we can recover $\Lambda^\mathrm{lin},\Lambda''$ from $\Lambda$. Recall that $\Lambda^\mathrm{lin}$ is the DtN map for the linear elastic wave equation. One has the mapping properties
\[
\Lambda^\mathrm{lin}:C_c^\infty((0,T)\times\partial\Omega)\rightarrow C^\infty((0,T)\times\partial\Omega),
\]
\[
\Lambda'':C_c^\infty((0,T)\times\partial\Omega)\times C_c^\infty((0,T)\times\partial\Omega)\rightarrow C^\infty((0,T)\times\partial\Omega),
\]
also by standard hyperbolic theory.

Assume that $u^{(0)}$ solves the initial boundary value problem for the backward elastic wave equation
\begin{equation}\label{backward_eq}
\begin{cases}
\rho\frac{\partial^2}{\partial t^2}u^{(0)}-\nabla\cdot \widetilde{\sigma}(x,u^{(0)})=0,\quad &(t,x)\in (0,T)\times\Omega,\\
u(t,x)=f^{(0)}(t,x),\quad &(t,x)\in (0,T)\times\partial \Omega,\\
u^{(0)}(T,x)=\frac{\partial}{\partial t}u^{(0)}(T,x)=0,\quad &x\in \Omega.
\end{cases}
\end{equation}
Using integration by parts, we have (cf. \cite{uhlmann2021inverse})
\begin{equation}\label{integralform}
\int_0^T\int_{\partial \Omega}\Lambda^{''}(f^{(1)},f^{(2)})\cdot f^{(0)}\mathrm{d}S\mathrm{d}t\\
=\int_0^T\int_{\Omega}\mathcal{G}(\nabla u^{(1)},\nabla u^{(2)},\nabla u^{(0)})\mathrm{d}x\mathrm{d}t,
\end{equation}
where
\begin{align}\label{integrand_G}
&\mathcal{G}(\nabla u^{(1)},\nabla u^{(2)},\nabla u^{(0)})\nonumber\\
=&(\lambda+\mathscr{B})(\nabla u^{(1)}:\nabla u^{(2)})(\nabla\cdot u^{(0)})+2\mathscr{C} (\nabla\cdot u^{(1)})(\nabla\cdot u^{(2)})(\nabla\cdot u^{(0)})\nonumber\\
&+\mathscr{B}(\nabla u^{(1)}:\nabla^T u^{(2)})(\nabla\cdot u^{(0)})\nonumber\\
&+\mathscr{B}\left( (\nabla\cdot u^{(1)})(\nabla u^{(2)}:\nabla^T u^{(0)})+(\nabla\cdot u^{(2)})(\nabla u^{(1)}:\nabla^T u^{(0)})\right)\nonumber\\
&+\frac{\mathscr{A}}{4}\left(\frac{\partial u^{(1)}_j}{\partial x_m}\frac{\partial u^{(2)}_m}{\partial x_i}+\frac{\partial u^{(2)}_j}{\partial x_m}\frac{\partial u^{(1)}_m}{\partial x_i}\right)\frac{\partial u^{(0)}_i}{\partial x_j}\\
&+(\lambda+\mathscr{B})\left((\nabla\cdot u^{(1)})(\nabla u^{(2)}:\nabla u^{(0)})+(\nabla\cdot u^{(2)})(\nabla u^{(1)}:\nabla u^{(0)})\right)\nonumber\\
&+(\mu+\frac{\mathscr{A}}{4})\Bigg(\frac{\partial u^{(1)}_m}{\partial x_i}\frac{\partial u^{(2)}_m}{\partial x_j}+\frac{\partial u^{(2)}_m}{\partial x_i}\frac{\partial u^{(1)}_m}{\partial x_j}+\frac{\partial u^{(1)}_i}{\partial x_m}\frac{\partial u^{(2)}_j}{\partial x_m}+\frac{\partial u^{(2)}_i}{\partial x_m}\frac{\partial u^{(1)}_j}{\partial x_m}\nonumber\\
&\quad\quad\quad\quad\quad\quad\quad\quad\quad+\frac{\partial u^{(1)}_i}{\partial x_m}\frac{\partial u^{(2)}_m}{\partial x_j}+\frac{\partial u^{(2)}_i}{\partial x_m}\frac{\partial u^{(1)}_m}{\partial x_j}\Bigg)\frac{\partial u^{(0)}_i}{\partial x_j}.\nonumber
\end{align}
Here we have used the notation $A:B=\sum_{i,j}A_{ij}B_{ij}$.

For the study of inverse problems,
we will construct a family of special solutions, called Gaussian beams, substitute into the integral equation \eqref{backward_eq}, and then extract useful information about the parameters we want to recover.

\section{Elastic Gaussian beams}\label{gaussian}
Denote
\[
Pu:=\rho\partial_t^2u-\nabla\cdot\widetilde{\sigma}(u).
\]
In this section, we construct Gaussian beam solutions of the form
\[
u(t,x)\sim\sum_{k=0}^{N+1} \varrho^{-k}\mathbf{a}_k(t,x)e^{\mathrm{i}\varrho\varphi}
\]
to the linear elastic wave equation 
\[
Pu=\rho\partial_t^2u-\nabla (\lambda\nabla \cdot u)-\nabla\cdot(\mu (\nabla u+\nabla^Tu))=0.
\]
 Gaussian beam solutions are asymptotic solutions that concentrate near a null geodesic $\vartheta$ in the large $\varrho$ limit. Here $\varphi$ is a complex-valued phase function that vanishes on $\vartheta$ with $\Im\varphi(t,x)\geq c_0d((t,x),\vartheta)^2$ for $\varrho>0$, and $\Im\varphi(t,x)\leq -c_0d((t,x),\vartheta)^2$ for $\varrho<0$, where $c_0>0$. So the solutions exhibit a Gaussian decay away from the null geodesic. Such solutions have been widely used for the study of inverse problems \cite{katchalov1998multidimensional,bao2014sensitivity, belishev1996boundary, dos2016calderon,feizmohammadi2019recovery,feizmohammadi2019inverse,lassas2021inverse,feizmohammadi2021inverse}.

Gaussian beam solutions for elastic waves have been constructed in \cite{uhlmann2021inverse,uhlmann2024determination}.
We will carry out the construction in a more coordinate-free way. In other words, we will resort to the Fermi coordinates (introduced below) only when necessary.\\

To begin with, we calculate
\begin{align*}
&\partial_t^2(\mathbf{a}e^{\mathrm{i}\varrho\varphi})=e^{\mathrm{i}\varrho\varphi}\left((\mathrm{i}\varrho)^2(\partial_t\varphi)^2\mathbf{a}+2\mathrm{i}\varrho\partial_t\varphi\partial_t\mathbf{a}+\mathrm{i}\varrho\partial_t^2\varphi\mathbf{a}+\partial_t^2\mathbf{a}\right),\\
&\nabla\cdot(\mathbf{a}e^{\mathrm{i}\varrho\varphi})=e^{\mathrm{i}\varrho\varphi}\left(\mathrm{i}\varrho\nabla\varphi\cdot\mathbf{a}+\nabla\cdot\mathbf{a}\right),\\
&\nabla(\lambda \nabla\cdot(\mathbf{a}e^{\mathrm{i}\varrho\varphi}))=e^{\mathrm{i}\varrho\varphi}\Big((\mathrm{i}\varrho)^2\lambda\nabla\varphi(\nabla\varphi\cdot\mathbf{a})+(\mathrm{i}\varrho)\lambda\nabla\varphi\nabla\cdot\mathbf{a}+\mathrm{i}\varrho\nabla\lambda(\nabla\varphi\cdot\mathbf{a})+\mathrm{i}\varrho\lambda\nabla(\nabla\varphi\cdot\mathbf{a})\\
&\quad\quad\quad\quad\quad\quad\quad\quad\quad\quad\quad\quad+\nabla\lambda\nabla\cdot\mathbf{a}+\lambda\nabla\nabla\cdot\mathbf{a}\Big),\\
&\nabla (\mathbf{a}e^{\mathrm{i}\varrho\varphi})+\nabla^T(\mathbf{a}e^{\mathrm{i}\varrho\varphi})=e^{\mathrm{i}\varrho\varphi}(\mathrm{i}\varrho\nabla\varphi\otimes\mathbf{a}+\mathrm{i}\varrho\mathbf{a}\otimes\nabla\varphi+\nabla\mathbf{a}+\nabla^T\mathbf{a}),\\
&\nabla\cdot(\mu(\nabla (\mathbf{a}e^{\mathrm{i}\varrho\varphi})+\nabla^T(\mathbf{a}e^{\mathrm{i}\varrho\varphi})))=e^{\mathrm{i}\varrho\varphi}\Big((\mathrm{i}\varrho)^2\mu(\nabla\varphi\cdot\mathbf{a})\nabla\varphi+(\mathrm{i}\varrho)^2\mu|\nabla\varphi|^2\mathbf{a}+\mathrm{i}\varrho\mu\nabla\varphi\cdot(\nabla\mathbf{a}+\nabla^T\mathbf{a})\\
&\quad\quad\quad\quad\quad\quad\quad\quad\quad\quad\quad\quad\quad\quad\quad\quad\quad+\mathrm{i}\varrho\mu(\nabla^2\varphi\mathbf{a}+\nabla\varphi\cdot\nabla\mathbf{a}+\Delta\varphi\mathbf{a}+\nabla\varphi\nabla\cdot\mathbf{a})\\
&\quad\quad\quad\quad\quad\quad\quad\quad\quad\quad\quad\quad\quad\quad\quad\quad\quad+\mathrm{i}\varrho((\nabla\mu\cdot\nabla\varphi)\mathbf{a}+(\nabla\mu\cdot\mathbf{a})\nabla\varphi)+\nabla\cdot(\mu(\nabla\mathbf{a}+\nabla^T\mathbf{a}))\Big).
\end{align*}
Here $(\nabla\varphi\cdot\nabla\mathbf{a})_i=(\sum_{j}\partial_j\varphi\partial_ja_i)$.
Introduce the following notations
\[
E(x,\varphi)\mathbf{a}:=\rho\mathbf{a}(\partial_t\varphi)^2-(\lambda+\mu)(\nabla\varphi\cdot\mathbf{a})\nabla\varphi-\mu|\nabla\varphi|^2\mathbf{a},
\]
and
\[
\begin{split}
T(x,\varphi)\mathbf{a}:=&2\rho\partial_t\varphi\partial_t\mathbf{a}+\rho\partial_t^2\varphi\mathbf{a}-\lambda\nabla\varphi\nabla\cdot\mathbf{a}-\nabla\lambda(\nabla\varphi\cdot\mathbf{a})-\lambda\nabla(\nabla\varphi\cdot\mathbf{a})\\
&-\mu\nabla\varphi\cdot(\nabla\mathbf{a}+\nabla^T\mathbf{a})-\mu(\nabla^2\varphi\cdot\mathbf{a}+\nabla\varphi\cdot\nabla\mathbf{a}+\Delta\varphi\mathbf{a}+\nabla\varphi\nabla\cdot\mathbf{a})\\
&-(\nabla\mu\cdot\nabla\varphi)\mathbf{a}-(\nabla\mu\cdot\mathbf{a})\nabla\varphi.
\end{split}
\]
Then we can write
\[
\begin{split}
e^{-\mathrm{i}\varrho\varphi}P\left(\sum_{k=0}^{N+1} \mathbf{a}_k(t,x)e^{\mathrm{i}\varrho\varphi}\right)=&-\varrho^2 E(x,\varphi)\mathbf{a}_0+\varrho(\mathrm{i}T(x,\varphi)\mathbf{a}_0-E(x,\varphi)\mathbf{a}_1)\\
&+\sum_{k=1}^{N}\varrho^{1-k}\left(\mathrm{i}T(x,\varphi)\mathbf{a}_k-E(x,\varphi)\mathbf{a}_{k+1}+P\mathbf{a}_{k-1}\right)\\
&+\varrho^{-N}(\mathrm{i}T(x,\varphi)\mathbf{a}_{N+1}+P\mathbf{a}_{N})+\varrho^{-1-N}P\mathbf{a}_{N+1}.
\end{split}
\]
Denote for $j=2,1,0,-1,\cdots, -N+1$,
\[
\mathcal{I}_j=\mathrm{i}T(x,\varphi)\mathbf{a}_{1-j}-E(x,\varphi)\mathbf{a}_{2-j}+P\mathbf{a}_{-j},
\]
where $\mathbf{a}_{-2}=\mathbf{a}_{-1}=0$. Note that
\[
\mathcal{I}_2=-E(x,\varphi)\mathbf{a}_0,
\]
\[
\mathcal{I}_1=\mathrm{i}T(x,\varphi)\mathbf{a}_0-E(x,\varphi)\mathbf{a}_1.
\]
We first extend the parameters $\lambda,\mu,\rho$ smoothly to a slightly larger domain $\widetilde{\Omega}\supset\supset\Omega$ and construct asymptotic solutions to $Pu=0$ in $\mathbb{R}\times\widetilde{\Omega}$. For $\bullet=P/S$, consider $(\widetilde{\Omega},g_\bullet)$ as a Riemannian manifold, and then $(\mathbb{R}\times\widetilde{\Omega},-\mathrm{d}t^2+g_\bullet)$ is naturally a Lorentzian manifold. Assume $\gamma(t)$ is a geodesic on  $(\widetilde{\Omega},g_\bullet)$ parametrized by arclength, then $\vartheta(t,\gamma(t))$ is a null-geodesic on $(\mathbb{R}\times\widetilde{\Omega},-\mathrm{d}t^2+g_\bullet)$. Throughout the paper, we assume $(\Omega,g_\bullet)$ is non-trapping and $\partial \Omega$ is convex with respect to $g_\bullet$.\\

\subsection{Fermi coordinates} We introduce the so called local Fermi coordinates in a neighborhood of a null geodesic $\vartheta(t)=(t,\gamma(t))$, $t\in (t_-,t_+)$, in $(\widetilde{\Omega},g_\bullet)$. We refer to \cite{feizmohammadi2019timedependent,kurylev2022inverse} for more details.

Assume $\gamma(t_0)=x_0\in \widetilde{\Omega}$ where $t_0\in(t_-,t_+)$. Choose $\alpha_2,\alpha_3\in T_{x_0}\widetilde{\Omega}$ such that $\{\dot{\gamma}(t_0),\alpha_2,\alpha_3\}$ forms an orthonormal basis for $T_{x_0}\widetilde{\Omega}$. Let $s$ denote the arclength along $\gamma$ from $x_0$. We note here that $s$ can be positive or negative, and $(t_0+s,\gamma(t_0+s))=\vartheta(t_0+s)$. For $k=2,3$, let $e_k(s)\in T_{\gamma(t_0+s)}\widetilde{\Omega}$ be the parallel transport of the vector $\alpha_k$ along $\gamma$ to the point $\gamma(t_0+s)$.

Define the coordinate system $(y^0=t,y^1=s,y^2,y^3)$ through $\mathcal{F}_1:\mathbb{R}^{1+3}\rightarrow \mathbb{R}\times\widetilde{\Omega}$:
\[
\mathcal{F}_1(y^0=t,y^1=s,y^2,y^3)=(t,\exp_{\gamma(t_0+s)}\left(y^2e_2(s)+y^3e_3(s)\right)).
\]
Then the null-geodesic $\vartheta$ can be expressed as 
\[
\vartheta=\{t-s=t_0,y^2=y^3=0\}.
\]
We remark here that although $\gamma$ as Riemannian geodesic in $(\widetilde{\Omega},g_\bullet)$ can have self-intersections, $\vartheta$ as a null geodesic in $(\mathbb{R}\times\widetilde{\Omega},-\mathrm{d}t^2+g_\bullet)$ does not have self-intersections.
Notice that $(y^1=s,y^2,y^3)$ gives a coordinate system on $\widetilde{\Omega}$ in a neighborhood of $\gamma(t_0-\varepsilon,t_0+\varepsilon)$, for any $t_0\in(t_-,t_+)$ and $\varepsilon>0$ small enough, such that
\[
g_\bullet\vert_{\gamma}=\sum_{j=1}^3(\mathrm{d}y^j)^2,\quad\text{and}\quad\frac{\partial g_{\bullet,jk}}{\partial y^i}\Big\vert_\gamma=0,~1\leq i,j,k\leq 3.
\]
Under this coordinate system, the Euclidean metric $g_E:=\mathrm{d}s^2$ locally takes the form
\[
g_E=c_\bullet^2g_\bullet=\sum_{1\leq i,j\leq 3}c^2_\bullet g_{\bullet,ij}\mathrm{d}y^i\mathrm{d}y^j,
\]
and the Christoffel symbols for the Euclidean metric are given by
\begin{equation}\label{Christoffel}
\begin{split}
&\Gamma_{\alpha\beta}^1=-c_\bullet^{-1}\frac{\partial c_\bullet}{\partial s} g_{\bullet,\alpha\beta},\quad \Gamma_{1\alpha}^\beta=\delta^\alpha_\beta c_\bullet^{-1}\frac{\partial c_\bullet}{\partial s},\quad \Gamma_{1\alpha}^1=c_\bullet^{-1}\frac{\partial c_\bullet}{\partial y^\alpha},\\
&\Gamma_{11}^\alpha=-c_\bullet^{-1}g_\bullet^{\alpha\beta}\frac{\partial c_\bullet}{\partial y^\beta},\quad \Gamma_{11}^1=c_\bullet^{-1}\frac{\partial c_\bullet}{\partial s},
\end{split}
\end{equation}
where $\alpha,\beta\in \{2,3\}$.
Introduce the map $(z^0,z'):=(z^0,z^1,z^2,z^3)=\mathcal{F}_2(y^0=t,y^1=s,y^2,y^3)$, where
\[
z^0=\tau=\frac{1}{\sqrt{2}}(t-t_0+s),\quad z^1=r=\frac{1}{\sqrt{2}}(-t+t_0+s),\quad z^j=y^j,\, j=2,3.
\]
The Fermi coordinates $(z^0=\tau,z^1=r,z^2,z^3)$ near $\vartheta$ is given by $\mathcal{F}:\mathbb{R}^{1+3}\rightarrow \mathbb{R}\times\widetilde{\Omega}$, where $\mathcal{F}=\mathcal{F}_1\circ\mathcal{F}_2^{-1}$. Note that on $\vartheta$, $r=0$. Denote $\tau_\pm=\sqrt{2}(t_\pm-t_0)$.
Then on $\vartheta$ we have
\[
\overline{g}_\bullet\vert_\vartheta=2\mathrm{d}\tau\mathrm{d}r+\sum_{j=2}^3(\mathrm{d}z^j)^2\quad\text{and}\quad\frac{\partial \overline{g}_{\bullet,jk}}{\partial z^i}\Big\vert_\vartheta=0,~0\leq i,j,k\leq 3.
\]

\subsection{WKB approximation}
We construct Gaussian beam solutions $u_\varrho$ of the form
\[
u_\varrho=\sum_{k=0}^{N+1} \varrho^{-k}\mathbf{a}_k(t,x)e^{\mathrm{i}\varrho\varphi},
\]
where $\mathbf{a}_k$ is supported in a neighborhood of $\vartheta$, and
\[
Pu_\varrho\sim 0.
\]

 For this purpose, we first construct smooth functions $\varphi$ and $\mathbf{a}_k$, $k=0,1,\cdots, N+1$, such that 
\[
\varphi(\vartheta(t))=0,
\]
\[
\begin{split}
\Im\varphi(t,x)\geq c_0d^2((t,x),\vartheta)\text{ for }\varrho>0,\quad\Im\varphi(t,x)\leq -c_0d^2((t,x),\vartheta)\text{ for }\varrho<0,\\
\end{split}
\]
and
$\mathcal{I}_j$ vanishes up to order $2(N+j)-1$ on $\vartheta$, i.e.,
\begin{equation}\label{conditionIj}
\frac{\partial^\Theta}{\partial z^\Theta}\mathcal{I}_j=0,\quad\text{for any }|\Theta|\leq 2(N+j)-1\text{ on }\vartheta,
\end{equation}
where and also in the following $\Theta=(\Theta_0=0,\Theta_1,\Theta_2,\Theta_3)$, $j=2,1,\cdots, 1-N$. If the above conditions hold, then
\[
|\mathcal{I}_j|\leq Cr^{2(N+j)},
\]
and then
\[
|\varrho^{j}\mathcal{I}_je^{\mathrm{i}\varrho\varphi}|\leq C\varrho^{-N}.
\]
Overall we have
\[
|Pu_\varrho|\leq C\rho^{-N}.
\]
By taking derivatives, we also have
\[
\|Pu_\varrho\|_{C^k((0,T)\times\Omega)}\leq C\rho^{-N+k}.
\]
We refer to \cite[Lemma 2.49]{kachalov2001inverse} for more details.\\
%

The condition \eqref{conditionIj} for $j=2$ reads
\begin{equation}\label{I2identity}
\frac{\partial^\Theta}{\partial z^\Theta}E(x,\varphi)\mathbf{a}_0=0,\quad\text{for any }|\Theta|\leq 2(N+2)-1\text{ on }\vartheta.
\end{equation}
Recall that
\[
E(x,\varphi)\mathbf{a}_0=\rho\mathbf{a}_0(\partial_t\varphi)^2-(\lambda+\mu)\langle\mathbf{a}_0,\nabla\varphi\rangle\nabla\varphi-\mu|\nabla\varphi|^2\mathbf{a}_0.
\]
Notice that
\[
\langle E(x,\varphi)\mathbf{a}_0,\nabla\varphi\rangle=(\rho(\partial\varphi)^2-(\lambda+2\mu)|\nabla\varphi|^2)\langle \mathbf{a}_0,\nabla\varphi\rangle.
\]
Here and throughout the paper, the inner produce $\langle\cdot\,,\,\cdot\rangle$ and norm $|\cdot|$ are with respect to the Euclidean metric if there is no subscript.
So if $\langle \mathbf{a}_0,\nabla\varphi\rangle\neq 0$, we must have
\[
\frac{\partial^\Theta}{\partial z^\Theta}(\rho(\partial_t\varphi)^2-(\lambda+2\mu)|\nabla\varphi|^2)=0,\quad\text{for any }|\Theta|\leq 2(N+2)-1\text{ on }\vartheta,
\]
and then
\[
\frac{\partial^\Theta}{\partial z^\Theta}\left(\mathbf{a}_0-\frac{\langle\mathbf{a}_0,\nabla\varphi\rangle\nabla\varphi}{|\nabla\varphi|^2}\right)=0,\quad\text{for any }|\Theta|\leq 2(N+2)-1\text{ on }\vartheta.
\]
If  $\langle \mathbf{a}_0,\nabla\varphi\rangle= 0$, we use
\[
E(x,\varphi)\mathbf{a}_0-\frac{\langle E(x,\varphi)\mathbf{a}_0,\nabla\varphi\rangle\nabla\varphi}{|\nabla\varphi|^2}=(\rho(\partial_t\varphi)^2-\mu|\nabla\varphi|^2)\left(\mathbf{a}_0-\frac{\langle\mathbf{a}_0,\nabla\varphi\rangle\nabla\varphi}{|\nabla\varphi|^2}\right)
\]
to conclude that
\[
\frac{\partial^\Theta}{\partial z^\Theta}(\rho(\partial_t\varphi)^2-\mu|\nabla\varphi|^2),\quad\text{for any }|\Theta|\leq 2(N+2)-1\text{ on }\vartheta,
\]
and then, due to the fact $\lambda+\mu>0$,
\[
\frac{\partial^\Theta}{\partial z^\Theta}\langle\mathbf{a}_0,\nabla\varphi\rangle=0,\quad\text{for any }|\Theta|\leq 2(N+2)-1\text{ on }\vartheta.
\]


For $P$-waves, we take $\varphi=\varphi_P$ such that
\begin{equation}\label{Pwavephaseeikonal}
\rho(\partial_t\varphi)^2-(\lambda+2\mu)|\nabla\varphi|^2=0,\quad\text{to order $2(N+2)-1$ on $\vartheta$},
\end{equation}
and
\begin{equation}\label{Pwaveamplitudeparallel}
\mathbf{a}_0=A_P\nabla\varphi.
\end{equation}
Then
\[
\mathcal{I}_2=-E(x,\varphi)\mathbf{a}_0=-\left(\rho(\partial_t\varphi)^2-(\lambda+2\mu)|\nabla\varphi|^2\right)\nabla\varphi.
\]
satisfies the condition \eqref{I2identity}.

For $S$-waves, we take $\varphi=\varphi_S$ such that
\begin{equation}\label{Swavephaseeikonal}
\rho(\partial_t\varphi)^2-\mu|\nabla\varphi|^2=0,\quad\text{to order $2(N+2)-1$ on $\vartheta$},
\end{equation}
and $\mathbf{a}_0$ such that
\begin{equation}\label{Swaveamplitudeorthogonal}
\langle \mathbf{a}_0,\nabla\varphi\rangle=0,\quad\text{to order $2(N+2)-1$ on $\vartheta$}.
\end{equation}
Then the condition \eqref{I2identity} is also satisfied.\\

For the rest of this section we assume $\varrho>0$. For $\varrho<0$, the construction is similar.

\subsection{Phase functions}
We first outline the construction of the phase function $\varphi$. Notice that one can rewrite the Eikonal equations \eqref{Pwavephaseeikonal} and \eqref{Swavephaseeikonal} for $\varphi=\varphi_\bullet$ in the following unified form.
\begin{equation}\label{phaseeikonal}
\frac{\partial^\Theta}{\partial z^\Theta}((\partial_t\varphi)^2-c_\bullet^2|\nabla\varphi|^2)=0,\quad \text{on }\vartheta,\quad|\Theta|\leq 2N+3,
\end{equation}
where $\bullet=P,S$. In the following we drop the subscript $\bullet$.
We can take
\[
\varphi(\tau,z')=\sum_{k=0}^{2N+3}\varphi_k(\tau,z'),
\]
where for each $k=0,1,\cdots, 2N+3$, $\varphi_k$ is a complex-valued homogeneous polynomial of degree $k$ in the variables $z^1,z^2,z^3$.
Actually one can take (cf. \cite{feizmohammadi2019recovery})
\[
\varphi_0(\tau,z')=0,\quad \varphi_1(\tau,z')=z^1=r,\quad \varphi_2(\tau,z')=\sum_{i,j=1}^3H_{ij}(\tau)z^iz^j,
\]
where $H$ is a symmetric matrix solving the Riccati equation
\begin{equation}\label{Ricatti}
\frac{\mathrm{d}}{\mathrm{d}\tau}H+HCH+D=0,\tau\in \left(\tau_-,\tau_+\right),\quad H(\tau_0)=H_0,\text{ with }\Im H_0>0,
\end{equation}
where $C$, $D$ are matrices with $C_{11}=0$, $C_{ii}=2$, $i=2,3$, $C_{ij}=0$, $i\neq j$, $D_{ij}=\frac{1}{4}(\partial_{ij}^2\overline{g}^{11})$, $H_0$ is any given symmetric matrix with $\Im H_0>0$ (For $\varrho<0$, one needs to choose $\Im H_0<0$). Here $(\overline{g}^{ij})$ is the inverse of the Lorentzian metric $(\overline{g}_{ij})$ under the Fermi coordinates $(z^0,z^1,z^2,z^3)$, i.e., $\overline{g}_{ij}=\overline{g}(\frac{\partial}{\partial z^i},\frac{\partial}{\partial z^j})$. Then the equation \eqref{Ricatti} has a unique solution with $\Im(H(\tau))>0$ for all $\tau$ (cf. \cite[Lemma 2.56]{kachalov2001inverse}). \\

For solving the Ricatti equation \eqref{Ricatti}, we take
\[
H(\tau)=Z(\tau)Y(\tau)^{-1},
\]
where $Z(\tau)$ and $Y(\tau)$ are solutions to the first order linear ODEs
\[
\begin{split}
\frac{\mathrm{d}}{\mathrm{d}\tau}Y=CZ,\quad Y(\tau_0)=Y_0,\\
\frac{\mathrm{d}}{\mathrm{d}\tau}Z=-DY,\quad Z(\tau_0)=H_0Y_0,
\end{split}
\]
where $Y_0$ is any non-degenerate matrix.
Here $Y(\tau)$ is non-degnerate for all $\tau$. Moreover, the following identity holds
\begin{equation}\label{identityHY}
\det(\Im H(\tau))|\det(Y(\tau))|^2=c_0,
\end{equation}
where $c_0$ is a constant independent of $\tau$. For more discussions on the Ricatti equation, we refer to \cite[Section 2.4]{kachalov2001inverse}.  We remark here that for $\varrho<0$ we can take conjugates of all the matrices $Y, Z$ and $H_0$.\\

Since the Lorentzian metric $\overline{g}=-\mathrm{d}t^2+g$ is of the product form, we have
\[
\overline{g}_{00}=\overline{g}\left(\frac{\partial}{\partial z^0},\frac{\partial}{\partial z^0}\right)=\overline{g}\left(\frac{1}{\sqrt{2}}\frac{\partial}{\partial y^1}+\frac{1}{\sqrt{2}}\frac{\partial}{\partial y^0},\frac{1}{\sqrt{2}}\frac{\partial}{\partial y^1}+\frac{1}{\sqrt{2}}\frac{\partial}{\partial y^0}\right)=\frac{1}{2}g_{11}-\frac{1}{2}.
\]
Here $(g_{ij})$ is the Riemannian metric under local Fermi coordinates $(y^1,y^2,y^3)$. Using \cite[Lemma 3.4 and Corollary 3.5]{feizmohammadi2019timedependent}, we have
\[
\frac{\partial^2\overline{g}^{11}}{\partial z^1\partial z^i}\Bigg\vert_{\vartheta(t)}=0,
\]
for any $i=1,2,3$, and
\[
\frac{\partial^2\overline{g}^{11}}{\partial z^i\partial z^j}\Big\vert_{\vartheta(t)}=-\frac{\partial^2\overline{g}_{00}}{\partial z^i\partial z^j}\Big\vert_{\vartheta(t)}=-\frac{1}{2}\frac{\partial^2g_{11}}{\partial y^i\partial y^j}\Big\vert_{\gamma(t)}=-R_{1i1j}\vert_{\gamma(t)},
\]
for $i,j=1,2,3$, where $R$ is the Riemann curvature tensor on the Riemannian manifold $(\Omega,g)$. To see the last equality one can calculate
\[
R_{1i1j}=\langle\nabla_1\nabla_i\partial_1,\partial_j\rangle_g-\langle\nabla_i\nabla_1\partial_1,\partial_j\rangle_g.
\]
We have
\[
\nabla_i\partial_1=\Gamma_{i1}^m(g)\partial_m,\quad \nabla_1\partial_1=\Gamma_{11}^m(g)\partial_m,
\]
where $\Gamma(g)$ is the Christoffel symbol of $g$, and
\[
\begin{split}
\nabla_1\nabla_i\partial_1=(\partial_1\Gamma_{i1}^m(g))\partial_m+\Gamma^m_{1k}(g)\Gamma^k_{i1}(g)\partial_m,\\
\nabla_i\nabla_1\partial_1=(\partial_i\Gamma^m_{11}(g))\partial_m+\Gamma_{ik}^m(g)\Gamma^k_{11}(g)\partial_m.
\end{split}
\]
Recall that $\partial_ig_{ik}\vert_\gamma=0$,  then $\Gamma^i_{jk}(g)\vert_\gamma=0$ and $\partial_1\Gamma^i_{jk}(g)\vert_\gamma=0$. Therefore
\[
\begin{split}
R_{1i1j}\vert_\gamma=&-\langle \partial_i\Gamma^m_{11}(g)\partial_m,\partial_j\rangle_g\vert_\gamma\\
=&-g_{mj}\partial_i\Gamma^m_{11}(g)\vert_\gamma\\
=&\frac{1}{2}\partial_{ij}^2g_{11}\vert_\gamma.
\end{split}
\]

Therefore one can take $Y_{11}=1$, $Y_{1j}=Y_{i1}=0$ for all $i,j=1,2,3$.
Denote $\widetilde{Y}=(Y_{\alpha\beta})_{\alpha,\beta=2,3}$, $\widetilde{D}=(D_{\alpha\beta})_{\alpha,\beta=2,3}$. Then
\[
\det(Y)=\det(\widetilde{Y}).
\]
and  $\widetilde{Y}$ solves the equation
\[
\frac{\mathrm{d}^2}{\mathrm{d}t^2}\widetilde{Y}=-4\widetilde{D}\widetilde{Y},
\]
which is equivalent to
\begin{equation}\label{jacobimatrixlocal}
\frac{\mathrm{d}^2}{\mathrm{d}t^2}\widetilde{Y}_{ij}+\sum_{k=2}^3R_{i11k}\widetilde{Y}_{k j}=0.
\end{equation}

Note that $\widetilde{Y}$ can be viewed as a transversal $(1,1)$-tensor along a geodesic $\gamma\subset (\widetilde{\Omega},g)$, the definition of which we shall recall now. We first define the orthogonal complement $\dot{\gamma}(t)^\perp$ of $\dot{\gamma}(t)$ as
\[
\dot{\gamma}(t)^\perp:=\{v\in T_{\gamma(t)}\widetilde{\Omega}\vert \langle \dot{\gamma}(t),v\rangle_g=0\}.
\]
We also define the $(1,1)$-tensor $\Pi_\gamma(t)=\Pi_i^j(t)\frac{\partial}{\partial y^j}\otimes\mathrm{d}x^i$ to be the projection from $T_{\gamma(t)}\widetilde{\Omega}$ onto $\dot{\gamma}(t)^\perp$. To be more precise
\[
\Pi_\gamma(t)v=v-\frac{\langle v,\dot{\gamma}(t)\rangle_g\dot{\gamma}(t)}{|\dot{\gamma}(t)|_g^2}.
\]
We say that a $(1,1)$-tensor $L(t)$ along $\gamma$ is transversal if $\Pi_\gamma L\Pi_\gamma=L$. Note that one can view $L(t)$ as a map $T_{\gamma(t)}\widetilde{\Omega}\rightarrow T_{\gamma(t)}\widetilde{\Omega}$.

Then one can express the equation \eqref{jacobimatrixlocal} in the following invariant form(noticing that $\partial_1=\dot{\gamma}$)
\[
\frac{D^2}{\mathrm{d}t^2}\widetilde{Y}(t)+K(t)\widetilde{Y}(t)=0,
\]
where $K(t):T_{\gamma(t)}\widetilde{\Omega}\rightarrow T_{\gamma(t)}\widetilde{\Omega}$ such that for any $V\in T_{\gamma(t)}\widetilde{\Omega}$,
\[
K(t)V=R(V,\dot{\gamma}(t))\dot{\gamma}(t).
\]

The higher order terms $\varphi_k$, $k=3,\cdots, 2N+3$, can be constructed so that \eqref{phaseeikonal} is satisfied (cf. \cite{feizmohammadi2019timedependent,feizmohammadi2019recovery}). \\
\subsection{Amplitudes of $P$-waves}
In this section, we show how to construct the amplitudes for $P$-waves. We can write
\begin{equation}\label{amplitudej}
\mathbf{a}_j=\sum_{k=0}^{2(N-j)+3}\chi\left(\frac{z'}{\delta}\right)\mathbf{a}_{jk}(\tau,z'),
\end{equation}
where $\mathbf{a}_{jk}$ is a complex-valued polynomial homogeneous of degree $k$ in $z'$, and $\chi$ is a common cut-off function such that $\chi(t)=1$ for $|t|\leq \frac{1}{4}$ and $\chi(t)=0$ for $|t|\geq \frac{1}{2}$ with $\delta>0$ sufficiently small. Actually we only need to determine the derivatives of $\mathbf{a}_j$ on $\vartheta$ up to order $2(N-j)+3$.\\

First we construct $\mathbf{a}_0$. Recall that we took $\mathbf{a}_0=A_P\nabla\varphi$, and so $A_P=\frac{\langle\mathbf{a}_0,\nabla\varphi\rangle}{|\nabla\varphi|^2}$. We start with
\[
\mathcal{I}_1=\mathrm{i}T(x,\varphi)\mathbf{a}_0-E(x,\varphi)\mathbf{a}_1=0,\quad\text{to order $2(N+1)-1$ on $\vartheta$}.
\]
Taking inner product of $\mathcal{I}_1$ and $\nabla\varphi$ yields
\[
\langle\mathrm{i}T(x,\varphi)\mathbf{a}_0-E(x,\varphi)\mathbf{a}_1,\nabla\varphi\rangle=0,\quad\text{to order $2(N+1)-1$ on $\vartheta$.}
\]
Notice that since
\[
(\partial_t\varphi)^2-c_P^2|\nabla\varphi|^2=0\quad\text{on }\vartheta,
\]
we have
\[
\langle E(x,\varphi)\mathbf{a}_1,\nabla\varphi\rangle=(\rho(\partial_t\varphi)^2-(\lambda+2\mu)|\nabla\varphi|^2)\langle\mathbf{a}_1,\nabla\varphi\rangle=0,\quad\text{to order $2(N+2)-1$ on $\vartheta$},
\]
so
\[
\langle T(x,\varphi)\mathbf{a}_0,\nabla\varphi\rangle=0,\quad\text{to order $2(N+1)-1$ on $\vartheta$}.
\]
By tedious calculation, we can write the above equation as
\begin{equation}\label{Ptransportinvariant}
\begin{split}
2\rho\partial_t\varphi\partial_tA_P|\nabla\varphi|^2-2(\lambda+2\mu)\langle\nabla\varphi,\nabla A_P\rangle|\nabla\varphi|^2&\\
+\rho\partial_t\varphi\partial_t|\nabla\varphi|^2A_P-(\lambda+2\mu)\langle\nabla\varphi,\nabla|\nabla\varphi|^2\rangle A_P&\\
-\langle \nabla\varphi,\nabla(\lambda+2\mu)\rangle A_P|\nabla\varphi|^2+\rho\partial_t^2\varphi A_P|\nabla\varphi|^2-(\lambda+2\mu)\Delta\varphi A_P|\nabla\varphi|^2&=0.
\end{split}
\end{equation}

Under Fermi coordinates,
\[
\partial_t\varphi=-\frac{1}{\sqrt{2}},\quad \partial_{y_1}\varphi=\frac{1}{\sqrt{2}},\quad \partial_{y_\alpha}\varphi=0,\quad\text{on }\vartheta.
\]
Then
\[
\langle\nabla\varphi,\nabla\rangle\vert_\vartheta=c_P^{-2}\frac{1}{\sqrt{2}}\partial_s,
\]
and consequently
\[
\rho\partial_t\varphi\partial_t-(\lambda+2\mu)\langle\nabla\varphi,\nabla\rangle=-\frac{1}{\sqrt{2}}\rho(\partial_t+\partial_s)=-\rho\partial_\tau,\quad\text{on }\vartheta.
\]
Notice that
\[
|\nabla\varphi|^2=\frac{1}{2}c_P^{-2},\quad\text{on }\vartheta.
\]
We have the following transport equation on $\vartheta$,
\[
\rho c_P^{-2}\partial_\tau A_P-\frac{1}{2}c_P^{-3}\rho\partial_\tau c_P A_P+\frac{1}{2}\rho\partial_\tau(\lambda+2\mu)c_P^{-2}A_P+\frac{1}{2}(\lambda+2\mu)c_P^{-2}\frac{\partial^2\varphi}{\partial y^\alpha\partial y^\alpha}A_P=0,
\]
where we have used
\[
\Delta\varphi=c_P^{-2}\partial_i^2\varphi-c_P^{-2}\Gamma_{ij}^k\partial_k\varphi=c_P^{-2}\sum_{\alpha=2}^3\frac{\partial^2\varphi}{\partial y^\alpha\partial y^\alpha}+\frac{1}{\sqrt{2}}c_P^{-3}\frac{\partial c_P}{\partial s}=\sum_{\alpha=2}^3\frac{\partial^2\varphi}{\partial y^\alpha\partial y^\alpha}+c_P^{-3}\frac{\partial c_P}{\partial \tau}.
\]
Therefore $A_P$ satisfies the transport equation along $\vartheta$,
\begin{equation}\label{transportequationAp}
\mathcal{T}A_P:=\partial_\tau A_P+\frac{1}{2}(\lambda+2\mu)^{-1}\partial_\tau(\lambda+2\mu)A_P-\frac{1}{2}c_P^{-1}\partial_\tau c_PA_P+\frac{1}{2}\det(Y_P)^{-1}\partial_\tau(\det(Y_P))A_P=0.
\end{equation}
So one can take
\[
A_P=C_0(\lambda+2\mu)^{-1/2}c_P^{1/2}\det(Y_P)^{-1/2}=C_0c_P^{-1/2}\rho^{-1/2}\det(Y_P)^{-1/2},\quad\text{on }\vartheta,
\]
as a solution to the above transport equation.
Taking derivative $\frac{\partial}{\partial z^\alpha}$ of \eqref{Ptransportinvariant}, we can get transport equations for $\frac{\partial A_P}{\partial z^\alpha}$, $\alpha=1,2,3$, in the form
\[
\mathcal{T}\left(\frac{\partial A_P}{\partial z^\alpha}\right)=F(\varphi,\partial\varphi,\partial^2\varphi,\partial^3\varphi,A_P).
\]
Imposing arbitrary initial values and solving the above transport equations along $\vartheta$, one can get $\frac{\partial A_P}{\partial z^\alpha}$.
Successively, one can determine
\[
\frac{\partial^\Theta A_P}{\partial z^\Theta}\Big\vert_\vartheta,\quad |\Theta|\leq 2N+1.
\]
We can take arbitrary values for
\begin{equation}\label{AParbitrary}
\frac{\partial^\Theta A_P}{\partial z^\Theta}\Big\vert_\vartheta
\end{equation}
for $|\Theta|=2N+2,2N+3$.

Then we use the equation
\[
\mathcal{I}_1-\frac{\langle \mathcal{I}_1,\nabla\varphi\rangle\nabla\varphi}{|\nabla\varphi|^2}=0,\quad\text{to order $2(N+1)-1$ on $\vartheta$},
\]
and the identity
\[
E(x,\varphi)\mathbf{a}_1-\frac{\langle E(x,\varphi)\mathbf{a}_1,\nabla\varphi\rangle\nabla\varphi}{|\nabla\varphi|^2}=(\rho(\partial_t\varphi)^2-\mu|\nabla\varphi|^2)\left(\mathbf{a}_1-\frac{\langle \mathbf{a}_1,\nabla\varphi\rangle\nabla\varphi}{|\nabla\varphi|^2}\right)
\]
to determine $
\mathbf{b}_1:=\mathbf{a}_1-\frac{\langle \mathbf{a}_1,\nabla\varphi\rangle\nabla\varphi}{|\nabla\varphi|^2}
$
 on $\vartheta$ to order $2N+1$, since $\rho(\partial_t\varphi)^2-\mu|\nabla\varphi|^2\neq 0$. Using
\[
\langle\mathrm{i}T(x,\varphi)\mathbf{a}_1-E(x,\varphi)\mathbf{a}_{2}+P\mathbf{a}_{0},\nabla\varphi\rangle=0,\quad\text{to order $2N-1$ on $\vartheta$},
\]
one can determine $A_1:=\frac{\langle \mathbf{a}_1,\nabla\varphi\rangle}{|\nabla\varphi|^2}$ on $\vartheta$ to order $2N-1$ by solving linear transport equations $\vartheta$ in a similar way as above. Choosing arbitrary values for $\frac{\partial^\Theta A_1}{\partial z^\Theta}\Big\vert_\vartheta=0$ for $|\Theta|=2N,2N+1$, we have determined $\mathbf{a}_1=\mathbf{b}_1+A_1\nabla\varphi$ on $\vartheta$ to order $2N+1$.
Continuing with this process, one can determine $\mathbf{a}_{j}$ on $\vartheta$ up to order $2N-2j+3$, for $j=2,3,\cdots, N+1$. We remark here that at the last step one can only determine $\mathbf{b}_{N+1}:=\mathbf{a}_{N+1}-\frac{\langle \mathbf{a}_{N+1},\nabla\varphi\rangle\nabla\varphi}{|\nabla\varphi|^2}$ on $\vartheta$ to order $1$, and we can take arbitrary values for $A_{N+1}:=\frac{\langle \mathbf{a}_{N+1},\nabla\varphi\rangle}{|\nabla\varphi|^2}$ on $\vartheta$ to order $1$.

\subsection{Amplitudes of $S$-waves}
Now we construct the amplitudes $\mathbf{a}_j$, $j=0,1,\cdots, N+1$, still in the form \eqref{amplitudej} for $S$-waves.
Notice that
\[
\begin{split}
E(x,\varphi)\mathbf{a}_1-\frac{\langle E(x,\varphi)\mathbf{a}_1,\nabla\varphi\rangle\nabla\varphi}{|\nabla\varphi|^2}=(\rho(\partial_t\varphi)^2-\mu|\nabla\varphi|^2)&\left(\mathbf{a}_1-\frac{\langle \mathbf{a}_1,\nabla\varphi\rangle\nabla\varphi}{|\nabla\varphi|^2}\right)=0,\\
&\quad\text{to order $2(N+2)-1$ on $\vartheta$}.
\end{split}
\]
Then we have
\[
T(x,\varphi)\mathbf{a}_0-\frac{\langle T(x,\varphi)\mathbf{a}_0,\nabla\varphi\rangle\nabla\varphi}{|\nabla\varphi|^2}=0,\quad\text{to order $2(N+1)-1$ on $\vartheta$},
\]
which can be written as
\begin{equation}\label{Stransportinvariant}
\begin{split}
2\rho\partial_t\varphi\partial_t\mathbf{a}_0-2\mu\langle\nabla\varphi,\nabla\mathbf{a}_0\rangle&\\
+\rho|\nabla\varphi|^{-2}\langle\mathbf{a}_0,\nabla(\partial_t\varphi)^2\rangle\nabla\varphi-\mu|\nabla\varphi|^{-2}\langle\mathbf{a}_0,\nabla|\nabla\varphi|^2\rangle\nabla\varphi&\\
-\langle\nabla\varphi,\nabla\mu\rangle\mathbf{a}_0+\rho\partial_t^2\varphi \mathbf{a}_0-\mu\Delta\varphi\mathbf{a}_0&=0,
\end{split}
\end{equation}
where we have used the fact that
\[
\partial_t\langle\mathbf{a}_0,\nabla\varphi\rangle=\langle\partial_t\mathbf{a}_0,\nabla\varphi\rangle+\langle\mathbf{a}_0,\nabla\partial_t\varphi\rangle=0,\quad\text{to order $2(N+1)$ on $\vartheta$},
\]
\[
\nabla\langle\mathbf{a}_0,\nabla\varphi\rangle=\nabla\varphi\cdot\nabla^T\mathbf{a}_0+\nabla^2\varphi\cdot\mathbf{a}_0=0,\quad\text{to order $2(N+1)$ on $\vartheta$}.
\]
Write $\mathbf{a}_0=(a_{01},a_{02},a_{03})$ in local Fermi coordinates and notice that $a_{01}=0$.
Then we can write the $0$-th order of \eqref{Stransportinvariant} as
\[
-2\rho\partial_\tau a_{0\beta}-c_S^{-2}\partial_\tau\mu a_{0\beta}+\mu c_S^{-3}\partial_\tau c_Sa_{0\beta}-\mu\sum_{\alpha=2}^3\frac{\partial^2\varphi}{\partial y^\alpha\partial y^\alpha}a_{0\beta}=0,\quad\beta=2,3,
\]
where we have used the identity
\[
a_{0\alpha;1}=\partial_1 a_{0\alpha}-\Gamma_{\alpha1}^\beta a_{0\beta}=\frac{ \partial a_{0\alpha}}{\partial s}-c_S^{-1}\frac{\partial c_S}{\partial s}a_{0\alpha}.
\]
Therefore $a_{0\alpha}$ satisfies the transport equation along $\vartheta$,
\begin{equation}\label{transportequationAs}
\mathcal{T}a_{0\alpha}=\partial_\tau a_{0\alpha}+\frac{1}{2}\mu^{-1}\partial_\tau\mu a_{0\alpha}-\frac{1}{2}c_S^{-1}\partial_\tau c_Sa_{0\alpha}+\frac{1}{2}\det(Y_S)^{-1}\partial_\tau(\det(Y_S))a_{0\alpha}=0.
\end{equation}
Actually we can take $a_{0\alpha}=C_0c_S^{-1/2}\rho^{-1/2}\det(Y_S)^{-1/2}$ on $\vartheta$, for $\alpha=2,3$. To write it in a geometrically invariant form, one can choose a (unit) covector field $\eta$ orthogonal and parallel along $\gamma$, and then take 
\[
\mathbf{a}_0\vert_\vartheta=A_S\eta,
\]
where 
\[
A_S=C_0c_S^{-1/2}\rho^{-1/2}\det(Y_S)^{-1/2}.
\]
Similar as for $P$-waves, one can determine 
\[
\frac{\partial^\Theta \mathbf{a}_0}{\partial z^\Theta}\Big\vert_\vartheta,\quad |\Theta|\leq 2N+1.
\]
For $|\Theta|= 2N+2,2N+3$, $\frac{\partial^\Theta \mathbf{a}_0}{\partial z^\Theta}\big\vert_\vartheta$ can be determined by the requirement \eqref{Swaveamplitudeorthogonal} (with certain freedoms though similar as for \eqref{AParbitrary}). Then one can determine
\[
B_1:=\frac{\langle\mathbf{a}_1,\nabla\varphi\rangle}{|\nabla\varphi|^2}
\]
on $\vartheta$ up to order $2N+1$ using the fact
\[
\langle \mathcal{I}_1,\nabla\varphi\rangle=0,\quad\text{to order $2(N+1)-1$ on $\vartheta$}.
\]
And then one can determine
\[
\mathbf{b}_1=\mathbf{a}_1-\frac{\langle\mathbf{a}_1,\nabla\varphi\rangle\nabla\varphi}{|\nabla\varphi|^2}
\]
on $\vartheta$ up to order $2N-1$ using
\[
\langle\mathrm{i}T(x,\varphi)\mathbf{a}_1-E(x,\varphi)\mathbf{a}_{2}+P\mathbf{a}_{0},\nabla\varphi\rangle=0,\quad\text{to order $2N-1$ on $\vartheta$}.
\]
For $|\Theta|= 2N,2N+1$, $\frac{\partial^\Theta \mathbf{b}_1}{\partial z^\Theta}\big\vert_\vartheta$ would be constrained by the obvious property $\langle \mathbf{b}_1,\nabla\varphi\rangle=0$.
The lower order terms can be constructed successively as for $P$-waves. At the last step one can only determine $B_{N+1}:=\frac{\langle \mathbf{a}_{N+1},\nabla\varphi\rangle}{|\nabla\varphi|^2}$ on $\vartheta$ to order $1$, and we can take arbitrary values for  $\mathbf{b}_{N+1}:=\mathbf{a}_{N+1}-\frac{\langle \mathbf{a}_{N+1},\nabla\varphi\rangle\nabla\varphi}{|\nabla\varphi|^2}$ on $\vartheta$ to order $1$.\\


\section{Characterization of reflections and scattering control}\label{gaussianwithoutreflection}
Once one designs some boundary sources that generate $P$- or $S$- Gaussian beams, those waves will be reflected when they hit the boundary generating both $P$- and $S$- waves. And of course there would also be further multiple reflections. The aim of this section is to analyze those mode conversions and do a blind scattering control, that is, to remove boundary reflections without knowing the material parameters $\lambda,\mu,\rho$, but only using the boundary measurements $\Lambda^\mathrm{lin}$. Such programme has been implemented in \cite{hintz2022dirichlet} for geometric optics solutions. Here we carry out this scheme for Gaussian beam solutions.

\subsection{Reflection of Gaussian beams}\label{secreflectiongaussian}
Assume $\vartheta(t)=(t,\gamma(t)),t\in[t_0,t_1]$ is a null geodesic in $((0,T)\times\Omega,-\mathrm{d}t^2+g_\bullet)$ with $\gamma(t_0),\gamma(t_1)\in\partial\Omega$. By the convexity assumption on $\partial\Omega$, $\vartheta$ intersects with $(0,T)\times\partial\Omega$ transversely. Extend $\vartheta$ a little bit such that $\vartheta(t),t\in(t_0-\varepsilon,t_1+\varepsilon)$ is a null geodesic in the extended domain $((0,T)\times\widetilde{\Omega},-\mathrm{d}t^2+g_\bullet)$. We also extend $\lambda,\mu,\rho$ from $\Omega$ to $\widetilde{\Omega}$ smoothly.\\

First let us take a $P$-wave Gaussian beam as the incident wave,
\begin{equation}\label{Pincidentwave}
u_\varrho^+=\sum_{k=0}^{N+1}\varrho^{-k}\mathbf{a}_{P,k}^+e^{\mathrm{i}\varrho\varphi_P^+},
\end{equation}
concentrating near a $P$-wave null geodesic $\vartheta$ as in previous section. 
When this incident wave hits the boundary at point $p:=\vartheta(t_1)=(t_1,\gamma(t_1))$, it will be reflected according to the zero Dirichlet boundary condition. After reflection, both $P$- and $S$- waves would be generated.
Therefore in a neighborhood of $p$ in $(0,T)\times\Omega$, the solution would be of the form
\begin{equation}\label{Ptotalgaussianbeam}
\begin{split}
u_\varrho=u_\varrho^++u_\varrho^-:=\sum_{k=0}^{N+1}\varrho^{-k}\mathbf{a}_{P,k}^+e^{\mathrm{i}\varrho\varphi_P^+}+\sum_{k=0}^{N+1}\varrho^{-k}\mathbf{a}_{P,k}^-e^{\mathrm{i}\varrho\varphi_P^-}+\sum_{k=0}^{N+1}\varrho^{-k}\mathbf{a}_{S,k}^-e^{\mathrm{i}\varrho\varphi_S^-},
\end{split}
\end{equation}
where the later two terms represent the reflected $P$- and $S$- waves respectively, which are also of the forms as in previous section.
We impose the conditions
\[
\varphi_P^+\vert_{(0,T)\times\partial \Omega}=\varphi_P^-\vert_{(0,T)\times\partial \Omega}=\varphi_S^-\vert_{(0,T)\times\partial \Omega},\quad\text{to order }2(N+2)-1\text{ at }p.
\]
Assume the reflected $P$-wave travel along a null geodesic $\vartheta_P^-$ in $((0,T)\times\Omega,-\mathrm{d}t^2+g_P)$ and  the reflected $S$-wave travel along a null geodesic $\vartheta_S^-$ in $((0,T)\times\Omega,-\mathrm{d}t^2+g_S)$. By Snell's law, we have
\[
\dot{\vartheta}_P^-(t_1)^{\flat,P}\vert_{T((0,T)\times\partial\Omega)}=\dot{\vartheta}_S^-(t_1)^{\flat,S}\vert_{T((0,T)\times\partial\Omega)}=\dot{\vartheta}_P^+(t_1)^{\flat,P}\vert_{T((0,T)\times\partial\Omega)}\in T^*((0,T)\times\partial\Omega).
\]
Here the musical diffeomorphism $\flat$ is taken with respect to the indicated Lorentzian metric.
 Then one can solve the Eikonal equation \eqref{phaseeikonal} along $\vartheta_{P/S}^-$  to obtain $\varphi_{P/S}^-$. In particular the derivatives of $\varphi^-_{P/S}$ at $p$ up to order $2N+3$ can be determined.

To construct the amplitudes $\mathbf{a}_{P,k}^-$ and $\mathbf{a}_{S,k}^-$ for reflected waves, one imposes the conditions
\begin{equation}\label{equationamplitudes}
(\mathbf{a}_{P,k}^++\mathbf{a}_{P,k}^-+\mathbf{a}_{S,k}^-)\vert_{(0,T)\times\partial\Omega}=0,\quad\text{to order }2(N-k)+3\text{ at }p,
\end{equation}
which follows from the zero Dirichlet boundary condition.
The derivatives of $\mathbf{a}_{P/S,k}^-\vert_{(0,T)\times\partial\Omega}$ at $p$ up to order $2(N-k)+3$ can be determined by also incorporating the equations \eqref{conditionIj} for them.  We shall give more details below.\\

Let $(y^1,y^2)$ be a local coordinates of $\partial \Omega$ near $x_p$ such that the metric induced by the Euclidean metric under this coordinates is $\delta_{ij}$ at $x_p$, where $p=(t_p,x_p)$. Then for any $x\in \Omega$ in a neighborhood of $x_p$, denote $x=(y^1,y^2,y^3)$ where $y^3$ is the distance from $x$ to the boundary $\partial\Omega$, and $(y^1,y^2)$ is the local coordinate for the nearest point on $\partial\Omega$ to $x_p$. Then we have the boundary normal coordinates $(y^0=t,y^1,y^2,y^3)$ for $(0,T)\times\Omega$ near $p$. The Minkowskian metric is then $\mathrm{diag}(-1,1,1,1)$ at $p$ under this boundary normal coordinates. Without loss of generality we assume $p=(0,0,0,0)$ and then denote $z=(y^0,y^1,y^2)$. Then the boundary $(0,T)\times\partial\Omega$ is locally expressed as $\{y^3=0\}$.

Assume in boundary normal coordinates,
\[
\partial_t\varphi_P^+(p)=\xi_0^+,\quad \partial_t\varphi_P^-(p)=\xi_{P,0}^-,\quad \partial_t\varphi_S^-(p)=\xi_{S,0}^-,
\]
and
\[
\nabla\varphi_P^+(p)=\xi^+=(\xi^+_1,\xi^+_2,\xi^+_3),
\]
\[
\nabla\varphi_P^-(p)=\xi^-_P=(\xi^-_{P,1},\xi^-_{P,2},\xi^-_{P,3}),\quad\quad\nabla\varphi_S^-(p)=\xi^-_S=(\xi^-_{S,1},\xi^-_{S,2},\xi^-_{S,3}).
\]
We have
\[
\xi^+_0=\xi^-_{P,0}=\xi^-_{S,0},\quad\xi^+_1=\xi^-_{P,1}=\xi^-_{S,1},\quad \xi^+_2=\xi^-_{P,2}=\xi^-_{S,2}.
\]
Since $\varphi_P^+$, $\varphi_P^-$, $\varphi_S^-$ all satisfy the corresponding Eikonal equations at point $p$, we have
\[
\xi^+_3=\sqrt{c_P^{-2}(\xi_0^+)^2-(\xi^+_1)^2-(\xi^+_2)^2},\quad \xi^-_{P,3}=-\xi^+_3,\quad \xi^-_{S,3}=-\sqrt{c_S^{-2}(\xi_0^+)^2-(\xi^+_1)^2-(\xi^+_2)^2}.
\]
To simplify notations, we denote
\[
\xi^+_0=\xi^-_{P,0}=\xi^-_{S,0}=\xi_0,
\]
and
\[
\xi^+=(\xi_1,\xi_2,\xi_3),\quad \xi^-_P=(\xi_1,\xi_2,-\xi_3),\quad \xi^-_S=(\xi_1,\xi_2,-\xi_{S,3}).
\]

By \eqref{equationamplitudes} with $k=0$, we have
\[
\mathbf{a}_{P,0}^-+\mathbf{a}_{S,0}^-=-\mathbf{a}_{P,0}^+\quad \text{at }p.
\]
We note that $\mathbf{a}_{P,0}^-$ is of the form
\begin{equation}\label{ap0minus}
\mathbf{a}_{P,0}^-=A_P^-\nabla\varphi_P^-,\quad\text{on }\vartheta_P^-,
\end{equation}
by \eqref{Pwaveamplitudeparallel}, and $\mathbf{a}_{S,0}^-$ satisfies 
\begin{equation}\label{as0minus}
\langle\mathbf{a}_{S,0}^-,\nabla\varphi_P^-\rangle=0,\quad\text{to order $2(N+2)-1$ on $\vartheta_S^-$},
\end{equation}
by \eqref{Swaveamplitudeorthogonal}.
In particular,
\[
\mathbf{a}_{P,0}^+(p)=A_P^+(p)\xi_P^+,\quad \mathbf{a}_{P,0}^-(p)=A_P^-(p)\xi_P^-,\quad \xi_S^-\cdot \mathbf{a}_{S,0}^-(p)=0.
\]
Therefore we have
\begin{equation}\label{systemprincipalreflection}
\left(\begin{array}{cccc}
\xi_1 & 1 & &\\
\xi_2 & & 1&\\
-\xi_3& &&1\\
0&\xi_1&\xi_2&-\xi_{S,3}
\end{array}\right)
\left(\begin{array}{c}
A_P^-\\
a_{S,01}^-\\
a_{S,02}^-\\
a_{S,03}^-
\end{array}\right)(p)=-\left(\begin{array}{c}
A_P^+\xi_1\\
A_P^+\xi_2\\
A_P^+\xi_3\\
0
\end{array}\right)(p).
\end{equation}
The above linear system is non-singular, thus one can determine $A_P^-,\mathbf{a}^-_{S,0}$ at point $p$. Then one can solve the transport equations \eqref{transportequationAp} for $A_P^-$ on $\vartheta_P^-$, and \eqref{transportequationAs} for $\mathbf{a}^-_{S,0}$ on $\vartheta_S^-$. Taking tangential derivatives of \eqref{equationamplitudes} with $k=0$ and also using \eqref{ap0minus}\eqref{as0minus}, one can determine the derivatives of $A_P^-\vert_{(0,T)\times\partial\Omega},\mathbf{a}^-_{S,0}\vert_{(0,T)\times\partial\Omega}$ at $p$. Then solving the respective transport equations, one can determine the derivatives of $A_P^-$ on $\vartheta_P^-$ and $\mathbf{a}^-_{S,0}$ on $\vartheta_S^-$. Similarly, one can also determine the higher order derivatives of $A_P^-$ and  $\mathbf{a}^-_{S,0}$. Successively one can finish the construction of $\mathbf{a}^-_{P,k}$ and
$\mathbf{a}^-_{S,k}$, $k=1,2,\cdots, N+1$.\\

The phase functions restricted on the boundary can be expressed as
\[
\varphi_P^{+}(z)=\Phi(z)+\Xi_P^{+}(z),\quad \varphi_P^{-}(z)=\Phi(z)+\Xi_P^{-}(z),\quad \varphi_S^{-}(z)=\Phi(z)+\Xi_S^{-}(z),
\]
where $\Phi$ is a polynomial homogeneous of degree $2N+3$ in $z$, and $|\Xi_P^{+}|,|\Xi_P^{-}|,|\Xi_S^{-}|\leq C|z|^{2N+4}$. Then one can estimate
\[
\begin{split}
&\vert\mathbf{a}_{P,k}^+e^{\mathrm{i}\varrho\varphi_P^+}+\mathbf{a}_{P,k}^-e^{\mathrm{i}\varrho\varphi_P^-}+\mathbf{a}_{S,k}^-e^{\mathrm{i}\varrho\varphi_S^-}\vert\\
\leq &|e^{\mathrm{i}\varrho\Phi}||\mathbf{a}_{P,k}^++\mathbf{a}_{P,k}^-+\mathbf{a}_{S,k}^-|+|\mathbf{a}_{P,k}^+||e^{\mathrm{i}\varrho\varphi_P^+}-e^{\mathrm{i}\varrho\Phi}|+|\mathbf{a}_{P,k}^-||e^{\mathrm{i}\varrho\varphi_P^-}-e^{\mathrm{i}\varrho\Phi}|+|\mathbf{a}_{S,k}^-||e^{\mathrm{i}\varrho\varphi_S^-}-e^{\mathrm{i}\varrho\Phi}|.
\end{split}
\]
Notice that
\[
|e^{\mathrm{i}\varrho\Phi}||\mathbf{a}_{P,k}^++\mathbf{a}_{P,k}^-+\mathbf{a}_{S,k}^-|\leq e^{-C\varrho|z|^2}|z|^{2(N-k)+4}\leq C\varrho^{-N+k-2}.
\]
For the other terms, we have, for example,
\[
\begin{split}
|e^{\mathrm{i}\varrho\varphi_P^+}-e^{\mathrm{i}\varrho\Phi}|\leq &C\varrho|e^{\mathrm{i}\varrho\Phi}||\varphi_P^+-\Phi|\Big|\int_0^1e^{\mathrm{i}\varrho r(\varphi_P^+-\Phi)}\mathrm{d}r\Big|\\
\leq &C e^{-C\varrho|z|^2}\varrho|z|^{2N+4} e^{C\varrho|z|^{2N+4}}\\
\leq &C\varrho^{-N-2}.
\end{split}
\]
Therefore
\[
\|u_\varrho\|_{C(R_1)}\leq C\varrho^{-N-2},
\]
where $R_1\subset(0,T)\times\partial \Omega$ is a neighborhood of $p$. By taking derivatives, one can also prove
\[
\|u_\varrho\|_{C^k(R_1)}\leq C\varrho^{-N+k-2}.
\]

\begin{figure}[htbp]
\centering
\includegraphics[width=0.3\textwidth]{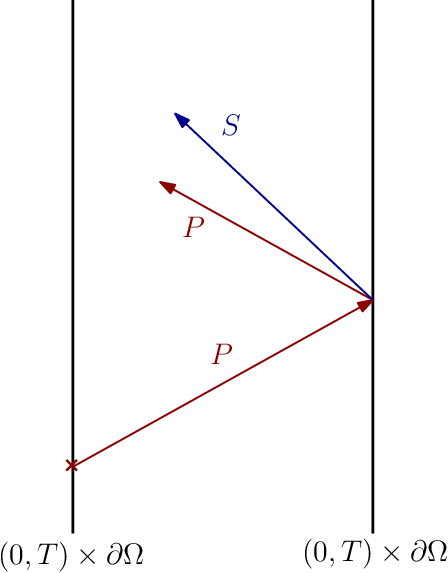}
\hspace{7em}
\includegraphics[width=0.3\textwidth]{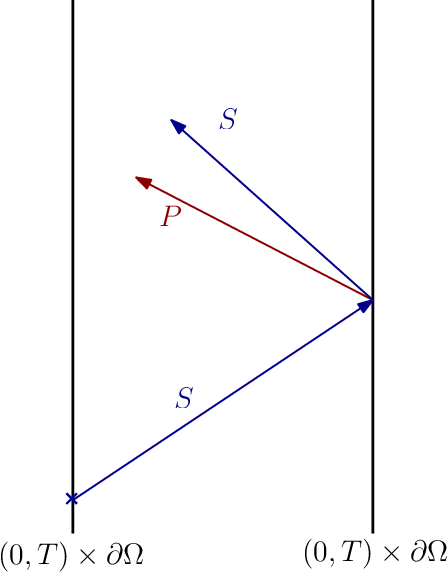}
\caption{Gaussian beams with reflections}
\label{GBwithrefction}
\end{figure}

If one takes an $S$-wave as the incident wave,
\[
u_\varrho^+=\sum_{k=0}^{N+1}\varrho^{-k}\mathbf{a}_{S,k}^+e^{\mathrm{i}\varrho\varphi_S^+},
\]
one can also characterize its reflection similarly. Near the reflection point, the incident plus the reflected waves are in the form
\[
\begin{split}
u_\varrho=u_\varrho^++u_\varrho^-:=&\sum_{k=0}^{N+1}\varrho^{-k}\mathbf{a}_{S,k}^+e^{\mathrm{i}\varrho\varphi_S^+}+\sum_{k=0}^{N+1}\varrho^{-k}\mathbf{a}_{P,k}^-e^{\mathrm{i}\varrho\varphi_P^-}+\sum_{k=0}^{N+1}\varrho^{-k}\mathbf{a}_{S,k}^-e^{\mathrm{i}\varrho\varphi_S^-}.
\end{split}
\]

Here we only briefly explain how to determine $\mathbf{a}_{P,0}^-$ and $\mathbf{a}_{S,0}^-$ at $p$. Notice that they still satisfy \eqref{ap0minus} and \eqref{as0minus}.
Assume in boundary normal coordinates
\[
\nabla \varphi_S^+(p)=\xi^+_S, \quad\nabla \varphi_P^-(p)=\xi^-_P, \quad\nabla \varphi_S^-(p)=\xi^-_S.
\] 
Denote
\[
\xi_0=\partial_t\varphi_S^+(p)=\partial_t\varphi_S^-(p)=\partial_t\varphi_P^-(p),
\]
and
\[
\xi^+_S:=\xi=(\xi_1,\xi_2,\xi_3), \quad \xi^-_S=(\xi_1,\xi_2,-\xi_3),\quad \xi^-_P=(\xi_1,\xi_2,-\xi_{P,3}),
\]
where 
\begin{equation}\label{xiP3}
\xi_{3}=\sqrt{c_S^{-2}\xi_0^2-\xi_1^2-\xi_2^2},\quad \xi_{P,3}=\sqrt{c_P^{-2}\xi_0^2-\xi_1^2-\xi_2^2}.
\end{equation}

In parallel with \eqref{systemprincipalreflection}, now we have
\begin{equation}\label{systemprincipalreflectionS}
\left(\begin{array}{cccc}
\xi_1 & 1 & &\\
\xi_2 & & 1&\\
-\xi_{P,3}& &&1\\
0&\xi_1&\xi_2&-\xi_{3}
\end{array}\right)
\left(\begin{array}{c}
A_P^-\\
a_{S,01}^-\\
a_{S,02}^-\\
a_{S,03}^-
\end{array}\right)(p)=-\left(\begin{array}{c}
a_{S,01}^+\\
a_{S,02}^+\\
a_{S,03}^+\\
0
\end{array}\right)(p).
\end{equation}
Therefore one can determine $\mathbf{a}_{P,0}^-$ and $\mathbf{a}_{S,0}^-$ at $p$. 

We remark here that $\xi_{P,3}$ could be imaginary, which corresponds to an evanescent wave. For this instance, the reflected $P$-wave $\sum_{k=0}^{N+1}\varrho^{-k}\mathbf{a}_{P,k}^-e^{\mathrm{i}\varrho\varphi_P^-}$ concentrate in a neighborhood of $p$.

\subsection{Remainder for Gaussian beams with reflections}
Now take an incident Gaussian beam solution $u_\varrho^+$ and denote $u_\varrho$ to be the total Gaussian beam solution with possibly multiple reflections. Assume $R_0\subset(0,T)\times\partial \Omega$ is a small neighborhood of $p_0=\vartheta(t_0)$, then take $u$ to be the solution to the initial boundary value problem 
\begin{equation}\label{IBVPwithreflection}
\begin{cases}
\rho\frac{\partial^2u}{\partial t^2}-\nabla\cdot \widetilde{\sigma}(u)=0,\quad &(t,x)\in (0,T)\times\Omega,\\
u(t,x)=u_\varrho^+(t,x)\vert_{R_0},\quad &(t,x)\in (0,T)\times\partial \Omega,\\
u(0,x)=\frac{\partial}{\partial t}u(0,x)=0,\quad &x\in \Omega,
\end{cases}
\end{equation}

Denote $R_\varrho=u-u_\varrho$. By previous discussions we have
\[
\|\rho\frac{\partial^2 R_\varrho}{\partial t^2}-\nabla\cdot \widetilde{\sigma}(R_\varrho)\|_{C^k((0,T)\times\Omega)}\leq C\varrho^{-N+k},
\]
and 
\[
\|R_\varrho\|_{C^{k+1}((0,T)\times\partial\Omega)}\leq C\varrho^{-N+k-1}.
\]
By standard theory for hyperbolic equations, we have
\[
\|R_\varrho\|_{H^{k+1}((0,T)\times\Omega)}\leq C\varrho^{-N+k}.
\]
By Sobolev embedding, for any $K>0$, one can take $N$ sufficiently large so that
\[
\|R_\varrho\|_{C^j((0,T)\times\Omega)}\leq C\varrho^{-2K+j-1}.
\]

\subsection{Scattering control for $P$-waves}
The reflection on the boundary would make the waves traveling inside the domain very difficult to keep track of.
In this subsection, we show how to remove the boundary reflection using $\Lambda^\mathrm{lin}$. 
First we need the boundary determination result in \cite{rachele2000boundary}.
\begin{proposition}\cite[Theorem 1]{rachele2000boundary}
The Dirichlet-to-Neumann map $\Lambda^\mathrm{lin}$ for the linear elastic wave equation uniquely determine $\lambda,\mu,\rho$ and all their derivatives on $\partial\Omega$.
\end{proposition}

This boundary determination allows us to extend $\lambda,\mu,\rho$ smoothly from $\Omega$ to $\widetilde{\Omega}$ without knowing their values in the interior of $\Omega$. First we take a $P$ incident wave $u^+_\varrho$ concentrating near a $P$-wave null geodesic $\vartheta$ as in the beginning part of Section \ref{secreflectiongaussian}. Assume $u_\varrho$ is the total Gaussian beam solution constructed in previous subsection, and $u$ is the associated solution to the initial boundary value problem \eqref{IBVPwithreflection}. We remark here that the values of $u_\varrho^+(t,x)\vert_{R_0}$ can be decided (up to certain orders at $p_0$) since the values of $\lambda,\mu,\rho$ in $\widetilde{\Omega}\setminus\Omega$ are already known. However the value of $u_\varrho^+(t,x)$ near $p$ cannot be directly determined since it depends also on the values of $\lambda,\mu,\rho$ in the interior of $\Omega$ (by solving the Eikonal and transport equations for the phase and the amplitudes), which are unknown for the inverse problem. In this section, we use the measurements of Neumann value of $u$ near $p$ to determine $u_\varrho^+(t,x)$ near $p$.\\

In a neighborhood of $p$, $u_\varrho$ is of the form \eqref{Ptotalgaussianbeam}.  Therefore we can choose a sufficiently large $K$ such that, near $p\in(0,T)\times\partial\Omega$, the corresponding Neumann value of $u=u_\varrho+R_\varrho$ is
\begin{equation}\label{asymptoticneumann}
\varrho\mathcal{N}_1+\sum_{k=0}^{2K-1}\varrho^{-k}\mathcal{N}_{-k}+\mathcal{O}_{C^j}(\varrho^{-2K+j}),
\end{equation}
where
\[
\begin{split}
\mathcal{N}_1=&\mathrm{i}\left(\lambda(\nabla\varphi_P^+\cdot\mathbf{a}_{P,0}^+)\nu+\mu(\nabla\varphi_P^+\cdot\nu)\mathbf{a}_{P,0}^++\mu(\mathbf{a}_{P,0}^+\cdot\nu)\nabla\varphi_P^+\right)e^{\mathrm{i}\varrho\varphi_P^+}\\
&+\mathrm{i}\left(\lambda(\nabla\varphi_P^-\cdot\mathbf{a}_{P,0}^-)\nu+\mu(\nabla\varphi_P^-\cdot\nu)\mathbf{a}_{P,0}^-+\mu(\mathbf{a}_{P,0}^-\cdot\nu)\nabla\varphi_P^-\right)e^{\mathrm{i}\varrho\varphi_P^-}\\
&+\mathrm{i}\left(\lambda(\nabla\varphi_S^-\cdot\mathbf{a}_{S,0}^-)\nu+\mu(\nabla\varphi_S^-\cdot\nu)\mathbf{a}_{S,0}^-+\mu(\mathbf{a}_{S,0}^-\cdot\nu)\nabla\varphi_S^-\right)e^{\mathrm{i}\varrho\varphi_S^-},
\end{split}
\]
and
\[
\begin{split}
\mathcal{N}_{-k}=&\Big(\mathrm{i}\lambda(\nabla\varphi_P^+\cdot\mathbf{a}_{P,k+1}^+)\nu+\mathrm{i}\mu(\nabla\varphi_P^+\cdot\nu)\mathbf{a}_{P,k+1}^++\mathrm{i}\mu(\mathbf{a}_{P,k+1}^+\cdot\nu)\nabla\varphi_P^+\\
&\quad\quad\quad\quad\quad\quad\quad\quad\quad\quad+\lambda(\nabla\cdot\mathbf{a}^+_{P,k})\nu+\mu(\nabla\mathbf{a}^+_{P,k}+\nabla^T\mathbf{a}^+_{P,k})\nu\Big)e^{\mathrm{i}\varrho\varphi_P^+}\\
+&\Big(\mathrm{i}\lambda(\nabla\varphi_P^-\cdot\mathbf{a}_{P,k+1}^-)\nu+\mathrm{i}\mu(\nabla\varphi_P^+\cdot\nu)\mathbf{a}_{P,k+1}^-+\mathrm{i}\mu(\mathbf{a}_{P,k+1}^-\cdot\nu)\nabla\varphi_P^-\\
&\quad\quad\quad\quad\quad\quad\quad\quad\quad\quad+\lambda(\nabla\cdot\mathbf{a}^-_{P,k})\nu+\mu(\nabla\mathbf{a}^-_{P,k}+\nabla^T\mathbf{a}^-_{P,k})\nu\Big)e^{\mathrm{i}\varrho\varphi_P^-}\\
+&\Big(\mathrm{i}\lambda(\nabla\varphi_P^-\cdot\mathbf{a}_{S,k+1}^-)\nu+\mathrm{i}\mu(\nabla\varphi_S^+\cdot\nu)\mathbf{a}_{S,k+1}^-+\mathrm{i}\mu(\mathbf{a}_{S,k+1}^-\cdot\nu)\nabla\varphi_S^-\\
&\quad\quad\quad\quad\quad\quad\quad\quad\quad\quad+\lambda(\nabla\cdot\mathbf{a}^-_{S,k})\nu+\mu(\nabla\mathbf{a}^-_{S,k}+\nabla^T\mathbf{a}^-_{S,k})\nu\Big)e^{\mathrm{i}\varrho\varphi_S^-}.
\end{split}
\]
Note that for each $k$
\[
\mathcal{N}_{-k}=\mathcal{M}_{-k}e^{\mathrm{i}\varrho\varphi_P^+},\quad \text{to order }2K+1\text{ at }p.
\]
Here
\begin{align*}
\mathcal{M}_{-k}=&\Big(\mathrm{i}\lambda(\nabla\varphi_P^+\cdot\mathbf{a}_{P,k+1}^+)\nu+\mathrm{i}\mu(\nabla\varphi_P^+\cdot\nu)\mathbf{a}_{P,k+1}^++\mathrm{i}\mu(\mathbf{a}_{P,k+1}^+\cdot\nu)\nabla\varphi_P^+\\
&\quad\quad\quad\quad\quad\quad\quad\quad\quad\quad+\lambda(\nabla\cdot\mathbf{a}^+_{P,k})\nu+\mu(\nabla\mathbf{a}^+_{P,k}+\nabla^T\mathbf{a}^+_{P,k})\nu\Big)\\
+&\Big(\mathrm{i}\lambda(\nabla\varphi_P^-\cdot\mathbf{a}_{P,k+1}^-)\nu+\mathrm{i}\mu(\nabla\varphi_P^+\cdot\nu)\mathbf{a}_{P,k+1}^-+\mathrm{i}\mu(\mathbf{a}_{P,k+1}^-\cdot\nu)\nabla\varphi_P^-\\
&\quad\quad\quad\quad\quad\quad\quad\quad\quad\quad+\lambda(\nabla\cdot\mathbf{a}^-_{P,k})\nu+\mu(\nabla\mathbf{a}^-_{P,k}+\nabla^T\mathbf{a}^-_{P,k})\nu\Big)\\
+&\Big(\mathrm{i}\lambda(\nabla\varphi_P^-\cdot\mathbf{a}_{S,k+1}^-)\nu+\mathrm{i}\mu(\nabla\varphi_S^+\cdot\nu)\mathbf{a}_{S,k+1}^-+\mathrm{i}\mu(\mathbf{a}_{S,k+1}^-\cdot\nu)\nabla\varphi_S^-\\
&\quad\quad\quad\quad\quad\quad\quad\quad\quad\quad+\lambda(\nabla\cdot\mathbf{a}^-_{S,k})\nu+\mu(\nabla\mathbf{a}^-_{S,k}+\nabla^T\mathbf{a}^-_{S,k})\nu\Big),
\end{align*}
where $\mathbf{a}^+_{P,-1}=\mathbf{a}^-_{P,-1}=\mathbf{a}^-_{S,-1}=0$.

Consider the values of $\mathcal{N}_{-k}$ at $p$ and note that $e^{\mathrm{i}\varrho\varphi_P^+(p)}=e^{\mathrm{i}\varrho\varphi_P^-(p)}=e^{\mathrm{i}\varrho\varphi_S^-(p)}=1$. Collecting terms of same orders in $\varrho$, we can recover
\[
\mathcal{M}_1(p),\,\,\mathcal{M}_0(p),\,\,\mathcal{M}_{-1}(p),\cdots, \mathcal{M}_{-2K+1}(p).
\]

We use the condition that $\varphi_P^+=\varphi_P^-=\varphi_S^-$ at $p$ up to order $2K+1$.
Taking tangential derivative of \eqref{asymptoticneumann} at point $p$, collecting the leading order terms (in $\varrho$) and noticing
\[
\varrho\partial_\beta\mathcal{N}_1(p)=-\varrho^2\partial_\beta\varphi_P^+(p)\mathcal{M}_1(p)+\mathcal{O}(\varrho),
\]
we can determine $\partial_\beta\varphi_P^+(p)=\xi_\beta$ for $\beta=0,1,2$.
Then one can determine
\[
\begin{split}
\partial_3\varphi_P^+(p)&=\xi_3=\sqrt{c_P^{-2}\xi_0^2-\xi_1^2-\xi_2^2},\\
\partial_3\varphi_P^-(p)&=-\xi_3=-\sqrt{c_P^{-2}\xi_0^2-\xi_1^2-\xi_2^2},\\
\partial_3\varphi_S^-(p)&=-\xi_{S,3}=-\sqrt{c_S^{-2}\xi_0^2-\xi_1^2-\xi_2^2}.
\end{split}
\] 
Then we can determine tangential derivatives
\[
\partial_\beta\mathcal{M}_1(p),\,\,\partial_\beta\mathcal{M}_0(p),\,\,\partial_\beta\mathcal{M}_{-1}(p),\cdots, \partial_\beta\mathcal{M}_{-2K+2}(p).
\]
Successively taking higher order tangential derivatives and using the Eikonal equations for $\varphi_P^+,\varphi_P^-,\varphi_S^-$, we can determine
\[
\varphi_P^+,\,\,\varphi_P^-,\,\,\varphi_S^-
\]
at point $p$ up to order $2K+1$. Also we can determine the derivatives of $\mathcal{M}_{-k}$ at $p$ up to order $2K-k-1$.\\

Next we determine $\mathbf{a}_{P,0}^+(p)$. 
We can write
\begin{equation}\label{systemprincipalneumann}
\begin{split}
\frac{\mathcal{M}_1(p)}{\mathrm{i}}=&\left(\begin{array}{ccc}-\mu \xi_3&0 & \mu \xi_1\\
0 &-\mu \xi_3 &\mu\xi_2\\
\lambda \xi_1 &\lambda \xi_2&-(\lambda+2\mu) \xi_3
\end{array}\right)\left(\begin{array}{c}
\xi_1\\
\xi_2\\
-\xi_3
\end{array}\right)A_P^-\\
&+\left(\begin{array}{ccc}-\mu \xi_{S,3}&0 & \mu \xi_1\\
0 &-\mu \xi_{S,3} &\mu \xi_2\\
\lambda \xi_1 &\lambda \xi_2&-(\lambda+2\mu) \xi_{S,3}
\end{array}\right)\left(\begin{array}{c}
a_{S,01}^-\\
a_{S,02}^-\\
a_{S,03}^-
\end{array}\right)\\
&+\left(\begin{array}{ccc}\mu \xi_3&0 & \mu \xi_1\\
0 &\mu \xi_3 &\mu \xi_2\\
\lambda \xi_1 &\lambda \xi_2&(\lambda+2\mu) \xi_3
\end{array}\right)\left(\begin{array}{c}
\xi_1\\
\xi_2\\
\xi_3
\end{array}\right)A_{P}^+.
\end{split}
\end{equation}
To show that $\mathcal{M}_1(p)$ determines $A_P^+$ uniquely by \eqref{systemprincipalreflection} and \eqref{systemprincipalneumann}, we suppose $\mathcal{N}_1(p)=0$. Without loss of generality, we assume $\xi_1\neq 0$. Then use the identities
\begin{align*}
\xi_1A_P^-+a^-_{S,01}+\xi_1A_P^+=0,\\
-\xi_3A_P^-+a^-_{S,03}+\xi_3A_P^+=0,\\
-2\mu\xi_1\xi_3A_P^--\mu\xi_{S,3}a^-_{S,01}+\mu\xi_1a^-_{S,03}+2\mu\xi_1\xi_3A^+_P=0,\\
(\lambda(\xi_1^2+\xi_2^2+\xi_3^2)+2\mu\xi_3^2)A_P^--2\mu\xi_{S,3}a^-_{S,03}+(\lambda(\xi_1^2+\xi_2^2+\xi_3^2)+2\mu\xi_3^2)A^+_P=0,
\end{align*}
resulted from \eqref{systemprincipalreflection} and \eqref{systemprincipalneumann}. The above equations can be reduced to
\[
\begin{split}
\mu\xi_{S,3}a^-_{S,01}+\mu\xi_1a^-_{S,03}=0,\\
(\lambda c_P^{-2}\xi_0^2+2\mu\xi_3^2)a^-_{S,01}+2\mu\xi_1\xi_{S,3}a^-_{S,03}=0.
\end{split}
\]
For the above linear system we calculate the determinant of the coefficient matrix
\[
\det\left(\begin{array}{cc}
\mu\xi_{S,3}&\mu\xi_1\\
(\lambda c_P^{-2}\xi_0^2+2\mu\xi_3^2)&2\mu\xi_1\xi_{S,3}
\end{array}\right)=\rho\mu\xi_1\xi_0^2\neq 0.
\]
Then $a^-_{S,01}=a^-_{S,03}=0$, and consequently we conclude that $A_P^+=0$. Therefore the value of $\mathcal{N}_1(p)$ determines $\mathbf{a}_{P,0}^+(p)$.

By similar approach as above, we can determine $\partial_\beta A_{P,0}^+(p)$, $\beta=0,1,2$, from $\partial_\beta\mathcal{M}_1(p)$, and then $\partial_3A_{P,0}^+(p)$ by the transport equation for $A_{P,0}^+(p)$. 
So successively one can determine the derivatives of $A_{P,0}^+$ at $p$ up to order $2K$. Then one can determine $\mathbf{a}_{P,1}^+-\frac{\langle \nabla\varphi_P^+,\mathbf{a}_{P,1}^+\rangle \nabla\varphi_P^+}{|\nabla\varphi_P^+|^2}$ at point $p$ to order $2K-1$ by the construction of Gaussian beam solutions. Then similar as above, one can determine the derivatives of $A_{P,1}^+:=\frac{\langle \nabla\varphi_P^+,\mathbf{a}_{P,1}^+\rangle }{|\nabla\varphi_P^+|^2}$ from the derivatives of $\mathcal{M}_0$ at $p$ to order $2K-1$.
Continuing with this, one can determine the derivatives of $\mathbf{a}_{P,k}^+$ at $p$ up to order $2K-k$ for $k=1,2,\dots, 2K$. \\

So we can recover a truncation of the incident $P$-wave Gaussian beam $u_\varrho^+$ in \eqref{Pincidentwave}, still denoted by 
\[
u_\varrho=\sum_{k=0}^{2K}\varrho^{-k}\mathbf{a}_{P,k}^+e^{\mathrm{i}\varrho\varphi_P^+}.
\]
Here we know the derivatives of $\varphi_P^+$ at both $\vartheta(t_0)$ and $\vartheta(t_1)$ up to order $2K+1$, and the derivatives of $\mathbf{a}_{P,k}^+$ up to order $2K-k$. So we truncate off the higher order terms. We remark here that one can take $K$ to be arbitrarily large by taking $N$ sufficiently large. For any $K'>0$, we take $K$ large enough such that
\[
\|Pu_\varrho\|_{H^k}\leq C\varrho^{-K'+k}.
\]
So one can design boundary sources $f_\varrho=u_\varrho\vert_{(0,T)\times\partial\Omega}$ without knowing the values of $\lambda,\mu,\rho$ in the interior. Let $u$ to be the solution of the initial boundary value problem,
\begin{equation}
\begin{cases}
\rho\frac{\partial^2u}{\partial t^2}-\nabla\cdot \widetilde{\sigma}(u)=0,\quad &(t,x)\in (0,T)\times\Omega,\\
u(t,x)=f_\varrho(t,x),\quad &(t,x)\in (0,T)\times\partial \Omega,\\
u(0,x)=\frac{\partial}{\partial t}u(0,x)=0,\quad &x\in \Omega.
\end{cases}
\end{equation}
Then 
\[
\|u-u_\varrho\|_{H^{k+1}((0,T)\times\Omega)}\leq C\varrho^{-K'+k}.
\]

\begin{figure}[htbp]
\centering
\includegraphics[width=0.3\textwidth]{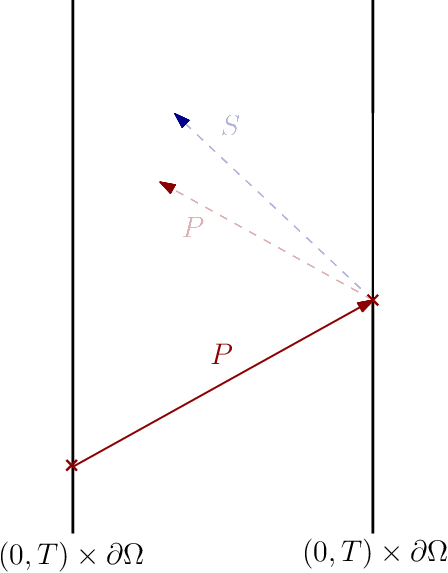}
\hspace{7em}
\includegraphics[width=0.3\textwidth]{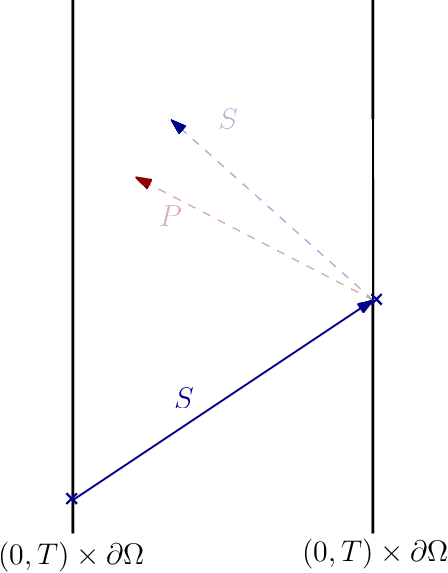}
\caption{Removing the boundary reflections of Gaussian beams}
\label{GBwithoutrefction}
\end{figure}

\begin{remark}
Generally speaking, one can first send a Gaussian beam into the domain and find the first reflection point at the boundary. According to the measurements near the first reflection point one can design boundary sources to remove the reflection up to a lower order level.
\end{remark}

\subsection{Scattering control for $S$-waves}
Now consider an incident $S$-wave, then near the first reflection point $p$, the total Gaussian beam takes the form
\[
\begin{split}
u_\varrho^++u_\varrho^-:=&\sum_{k=0}^N\varrho^{-k}\mathbf{a}_{S,k}^+e^{\mathrm{i}\varrho\varphi_S^+}+\sum_{k=0}^N\varrho^{-k}\mathbf{a}_{P,k}^-e^{\mathrm{i}\varrho\varphi_P^-}+\sum_{k=0}^N\varrho^{-k}\mathbf{a}_{S,k}^-e^{\mathrm{i}\varrho\varphi_S^-}.
\end{split}
\]
Here
\[
\begin{split}
\mathcal{M}_1=&\mathrm{i}\left(\lambda(\nabla\varphi_S^+\cdot\mathbf{a}_{S,0}^+)\nu+\mu(\nabla\varphi_S^+\cdot\nu)\mathbf{a}_{S,0}^++\mu(\mathbf{a}_{S,0}^+\cdot\nu)\nabla\varphi_S^+\right)\\
&+\mathrm{i}\left(\lambda(\nabla\varphi_P^-\cdot\mathbf{a}_{P,0}^-)\nu+\mu(\nabla\varphi_P^-\cdot\nu)\mathbf{a}_{P,0}^-+\mu(\mathbf{a}_{P,0}^-\cdot\nu)\nabla\varphi_P^-\right)\\
&+\mathrm{i}\left(\lambda(\nabla\varphi_S^-\cdot\mathbf{a}_{S,0}^-)\nu+\mu(\nabla\varphi_S^-\cdot\nu)\mathbf{a}_{S,0}^-+\mu(\mathbf{a}_{S,0}^-\cdot\nu)\nabla\varphi_S^-\right).
\end{split}
\]
The phase function $\varphi_S^+$ up to order $2K+1$ at point $p$ can be determined in a similar way as for incident $P$-waves.

Now we have
\[
\xi_1a_{S,01}^++\xi_2a_{S,02}^++\xi_3a_{S,03}^+=0,
\]
and
\begin{align*}
&\left(\begin{array}{ccc}-\mu \xi_{P,3}&0 & \mu \xi_1\\
0 &-\mu \xi_{P,3} &\mu\xi_2\\
\lambda \xi_1 &\lambda \xi_2&-(\lambda+2\mu) \xi_{P,3}
\end{array}\right)\left(\begin{array}{c}
\xi_1\\
\xi_2\\
-\xi_{P,3}
\end{array}\right)A_P^-\\
&+\left(\begin{array}{ccc}-\mu \xi_{3}&0 & \mu \xi_1\\
0 &-\mu \xi_{3} &\mu \xi_2\\
\lambda \xi_1 &\lambda \xi_2&-(\lambda+2\mu) \xi_{3}
\end{array}\right)\left(\begin{array}{c}
a_{S,01}^-\\
a_{S,02}^-\\
a_{S,03}^-
\end{array}\right)\\
&+\left(\begin{array}{ccc}\mu \xi_3&0 & \mu \xi_1\\
0 &\mu \xi_3 &\mu \xi_2\\
\lambda \xi_1 &\lambda \xi_2&(\lambda+2\mu) \xi_3
\end{array}\right)\left(\begin{array}{c}
a_{S,01}^+\\
a_{S,02}^+\\
a_{S,03}^+
\end{array}\right)=\mathcal{M}_1(p).
\end{align*}
Suppose $\mathcal{M}_1(p)=0$. Then by above identities together with \eqref{systemprincipalreflectionS} we have
\[
\begin{split}
\xi_1A_P^-+a_{S,01}^-+a_{S,01}^+=0,\\
\xi_2A_P^-+a_{S,02}^-+a_{S,02}^+=0,\\
-\xi_{P,3}A^-_P+a_{S,03}^-+a_{S,03}^+=0,\\
-2\mu\xi_1\xi_{P,3}A_P^--\mu\xi_3a_{S,01}^-+\mu\xi_1a^-_{S,03}+\mu\xi_3a_{S,01}^++\mu\xi_1a^+_{S,03}=0,\\
-2\mu\xi_2\xi_{P,3}A_P^--\mu\xi_3a_{S,02}^-+\mu\xi_2a^-_{S,03}+\mu\xi_3a_{S,02}^++\mu\xi_2a^+_{S,03}=0,\\
(\lambda c_P^{-2}+2\mu\xi_{P,3}^2)A_P^--2\mu\xi_3a_{S,03}^-+2\mu\xi_3a_{S,03}^+=0.
\end{split}
\]
We can deduce the following identities from above equations
\[
\begin{split}
(\xi_1^2+\xi_2^2)A_P^-+\xi_3(a_{S,03}^--a_{S,03}^+)=0,\\
(\lambda c_P^{-2}\xi_0^2+2\mu\xi_{P,3}^2)A_P^--2\mu\xi_3(a_{S,03}^--a_{S,03}^+)=0.
\end{split}
\]
The above linear system we calculate the determinant of the coefficient matrix
\[
\det\left(\begin{array}{cc}
(\xi_1^2+\xi_2^2)&\xi_3\\
\lambda c_P^{-2}\xi_0^2+2\mu\xi_{P,3}^2&-2\mu\xi_3
\end{array}\right)=-\xi_3\rho\xi_0^2\neq 0.
\]
Therefore $A_P^-=a_{S,03}^--a_{S,03}^+=0$. Then $a_{S,03}^-+a_{S,03}^+=0$, and followed by $a_{S,01}^-\pm a_{S,01}^+=0$ and $a_{S,02}^-\pm a_{S,02}^+=0$. Now we can conclude that
\[
A_P^-=a_{S,01}^-=a_{S,01}^+=a_{S,02}^-=a_{S,02}^+=a_{S,03}^-=a_{S,03}^+=0.
\]
Therefore $\mathcal{M}_1(p)$ determines $\mathbf{a}_{S,0}^+$ at point $p$. Similar as the $P$-waves, we can determine the derivatives of $\mathbf{a}_{S,k}^+$ at $p$ up to order $2K-k$ for $k=1,2,\dots, 2K$. The rest discussion is also similar.\\

From now on we can use Gaussian beams ignoring reflections at the boundary.


\section{The linear Dirichlet-to-Neumann map and lens relations}\label{linearequationresult}

Let us first recall the notion of \textit{lens relation} of a non-trapping Riemannian manifold $(M,g)$ with smooth boundary $\partial M$,
\[
\begin{split}
\mathcal{L}=&\{(x,\xi,y,\eta,t)\in SM\times SM\times[0,\infty)\vert (x,\xi)\in\partial_+SM,(y,\eta)\in\partial_-SM,t=\tau(x,\xi)\\
&\quad\quad\quad\quad\quad\quad\quad (y,\eta)=(\gamma_{x,\xi}(\tau(x,\xi),\dot{\gamma}_{x,\xi}(\tau(x,\xi))\}.
\end{split}
\]
Here
\[
\begin{split}
&\partial_+ SM=\{(x,v)\in SM,x\in\partial M,\langle v,\nu(x)\rangle_g> 0\},\\
&\partial_- SM=\{(x,v)\in SM,x\in\partial M,\langle v,\nu(x)\rangle_g< 0\}.
\end{split}
\]
The \textit{lens rigidity problem} asks whether one can recover $g$ (up to a diffeomorphism fixing the boundary) from the lens data. A landmark result was recently established in \cite{stefanov2021local}. For metrics in the same conformal class, which is more relevant to our problem, the uniqueness of the conformal factor has been proved under the simplicity condition \cite{michel1981rigidite} and the foliation condition \cite{stefanov2016boundary}.\\

One can show that the linear DtN map $\Lambda^\mathrm{lin}$ determines the lens relations for both $(\Omega,g_P)$ and $(\Omega,g_S)$, see \cite{hansen2003propagation}. Here we give a different proof inspired by the proof of \cite[Theorem 3.1]{yi2026leakage} and we hope it is of independent interest. Since we want to prove the uniqueness of $c_{P/S}$ under the less restrictive non-trapping condition, we will not apply the known lens rigidity results in \cite{michel1981rigidite} and \cite{stefanov2016boundary}. We start with a lemma proved in \cite{yi2026leakage}, which is essentially derived by the Blagovestchenskii identity (cf. \cite[Lemma 4.15]{kachalov2001inverse}).
\begin{lemma}{\textnormal{(}\cite[Lemma 6.2]{yi2026leakage}\textnormal{)}}
The $\rho$-weighted inner product
\[
\int_\Omega u^{(1)}(\frac{T}{2},x)\cdot u^{(2)}(\frac{T}{2},x)\rho(x)\mathrm{d}x
\]
can be determined from $\Lambda^\mathrm{lin}$. Here $u^{(j)}, j=1,2$ is the solution to \eqref{linear_eqj}.
\end{lemma}

\begin{theorem}
Assume $(\Omega,g_\bullet)$ is non-trapping and $\partial\Omega$ is convex with respect to $g_\bullet$ for either $\bullet=P,S$ and $T>2\max\{\mathrm{diam}(\Omega,g_P),\mathrm{diam}(\Omega,g_S)\}$. Then $\Lambda^\mathrm{lin}$ uniquely determines the lens relation of $(\Omega,g_P)$ and $(\Omega,g_S)$.
\end{theorem}
\begin{proof}
First take $(x_1,\xi_1),(x_2,\xi_2)\in\partial_+S^P\Omega$, $t_1,t_2\in (0,T)$ and denote
\[
\vartheta^{(1)}(t)=(t,\gamma_{x_1,\xi_1}^P(t-t_1)),\quad \vartheta^{(2)}(t)=(t,\gamma_{x_2,\xi_2}^P(t-t_2)).
\]
Then $\vartheta^{(1)}$ and $\vartheta^{(1)}$ are both null geodesics in $((0,T)\times\Omega,-\mathrm{d}t^2+g_P)$.
Here and throughout the paper, we add superscript $P/S$ to indicate it is a geometrical object with respect to $g_{P/S}$.
Construct $P$-wave Gaussian beam solutions, without reflections as constructed in Section \ref{gaussianwithoutreflection}, of the forms,
\[
\begin{split}
&u_{\varrho}^{(1)}=\sum_{k=0}^{2K}\varrho^{-k}\mathbf{a}^{(1)}_{P,k}e^{\mathrm{i}\varrho\varphi_P^{(1)}},\\
&u_{\varrho}^{(2)}=\sum_{k=0}^{2K}\varrho^{-k}\mathbf{a}^{(2)}_{P,k}e^{\mathrm{i}\varrho\varphi_P^{(2)}},
\end{split}
\]
concentrating near $\vartheta^{(1)}$ and $\vartheta^{(2)}$ respectively, with $\varrho>0$.

If $(x_1,\xi_1,x_2,-\xi_2,T-t_1-t_2)\in\mathcal{L}^P$ then $\gamma_{x_1,\xi_1}$ and $\gamma_{x_2,\xi_2}$ coincide (while $\vartheta^{(1)}$ and $\vartheta^{(2)}$ do not). We assume that $t_1,t_2\in(0,\frac{T}{2})$, and $\vartheta^{(1)}(\frac{T}{2})=\vartheta^{(2)}(\frac{T}{2})=(\frac{T}{2},x_q)$.  
Using the integral identity, we obtain
\[
\begin{split}
&\varrho^{3/2}\int_{\Omega}u^{(1)}(\frac{T}{2},x)\cdot u^{(2)}(\frac{T}{2},x)\rho(x)\mathrm{d}x\\
=&\varrho^{3/2}\int_{\Omega}\chi^{(1)}(\frac{T}{2},x)\chi^{(2)}(\frac{T}{2},x)\rho(x)A_P^{(1)}A_P^{(2)}\nabla\varphi^{(1)}(\frac{T}{2},\cdot)\cdot\nabla\varphi^{(2)}(\frac{T}{2},\cdot)e^{\mathrm{i}\varrho(\varphi^{(1)}(\frac{T}{2},\cdot)+\varphi^{(2)}(\frac{T}{2},\cdot))}\mathrm{d}x+\mathcal{O}(\varrho^{-1}).
\end{split}
\]
One can construct the Gaussian beam solutions such that 
\[
\partial_t\varphi^{(1)}(\frac{T}{2},x_q)=\partial_t\varphi^{(2)}(\frac{T}{2},x_q),\quad\nabla(\varphi^{(1)}+\varphi^{(2)})(\frac{T}{2},x_q)=0
\]
and
\[
\Im(\varphi^{(1)}+\varphi^{(2)})(\frac{T}{2},x)\geq c|x-x_q|^2.
\]
 Using the identity \eqref{identityHY}, we get
\[
\lim_{\varrho\rightarrow+\infty}\varrho^{3/2}e^{-\mathrm{i}\varrho(\varphi^{(1)}(\frac{T}{2},x_q)+\varphi^{(2)}(\frac{T}{2},x_q))}\int_{\Omega}u^{(1)}(\frac{T}{2},x)\cdot u^{(2)}(\frac{T}{2},x)\rho(x)\mathrm{d}x=c_0\neq 0,
\]
by applying the method of stationary phase. 

On the contrary, if $(x_1,\xi_1,x_2,-\xi_2,T-t_1-t_2)\notin\mathcal{L}^P$, then $\gamma^{(1)}$ and $\gamma^{(2)}$ do not intersect or intersect transversely. Then by the principle of non-stationary phase, 
\[
\lim_{\varrho\rightarrow+\infty}\varrho^{3/2}\int_{\Omega}u^{(1)}(\frac{T}{2},x)\cdot u^{(2)}(\frac{T}{2},x)\rho(x)\mathrm{d}x=0,
\]
In conclusion, the above argument shows that $\Lambda^\mathrm{lin}$ determines the lens data associated with $g_P$. Similarly one can use $S$-wave Gaussan beams to prove that $\Lambda^\mathrm{lin}$ determines the lens data associated with $g_S$.

\end{proof}

For a Riemannian manifold $(M,g)$ with boundary $\partial M$, we denote
\[
\mu_1(z)=\tau(z,\nu(z)),\quad z\in\partial M,
\]
where $\nu(z)$ is the unit interior normal vector to $\partial M$ at $z$. Since the lens relations for $(\Omega,g_P)$ and $(\Omega,g_S)$ are both known, we can determine $\mu_1^P(z)$ and $\mu_1^S(z)$ for any $z\in\partial\Omega$.

\section{Broken scattering relations}\label{brokenscattering}
In this section, let us depart from the elastic wave equation for a while and consider a geometrical inverse problem.\\

 For a compact Riemannian manifold $(M,g)$ with smooth boundary $\partial M$, we can extend it to a closed compact Riemannian manifold $(\widetilde{M},\widetilde{g})$.
The \textit{cut locus distance} along a geodesic $\gamma_{x,\xi}$ on $\widetilde{M}$ is denoted by
\[
\tau_R(x,\xi)=\sup\{s>0:\mathrm{dist}(x,\gamma_{x,\xi}(s))=s\}.
\]
For $(x,\xi)\in S\widetilde{M}$, we define the \textit{conjugate distance} $\tau_c(x,\xi)$ to be
\[
\tau_c(x,\xi)=\inf\{s>0:\mathrm{d}\exp_x\vert_{s\xi}\text{ is not one-to-one}\}.
\]

We also need to use the concepts of the \textit{boundary cut locus distance} and \textit{focal distance}. For this, we introduce the boundary exponential map
\[
\exp_{\partial M}:\partial M\times\mathbb{R}\rightarrow \widetilde{M},\quad \exp_{\partial M}(z,s)=\gamma_{z,\nu(z)}(s).
\]
 The pair $(z,s)$ defines the \textit{boundary normal coordinates} in $\widetilde{M}$ near $\partial M$. The \textit{boundary cut locus distance} is defined as
\[
\tau_b(z)=\sup\{s>0:\mathrm{dist}(\gamma_{z,\nu(z)}(s),\partial M)=s\}.
\]
The corresponding \text{boundary cut locus} is denoted by
\[
\omega_{\partial M}=\{y\in M\vert y=\gamma_{z,\nu(z)}(\tau_b(z)),z\in\partial M\}.
\]
For some $z\in \partial M$, we define its focal distance, $\tau_f(z)$ to be
\[
\tau_f(z)=\inf\{s>0:\mathrm{d}\exp_{\partial M}\vert_{(z,s)}\text{ is not one-to-one}\}.
\]
It can be shown that all the functions $\tau_R,\tau_c,\tau_b,\tau_f$ are continuous \cite{klingenberg1995riemannian}. Also we need the following relations
\begin{equation}
\tau_c(z,\nu(z))\geq \tau_f(z)+c_0,\quad z\in\partial M,
\end{equation}
\begin{equation}
\tau_R(z,\nu(z))\geq \tau_b(z)+c_0,\quad z\in\partial M,
\end{equation}
with some constant $c_0>0$, c.f. \cite{kurylev2010rigidity}.
~\\

A broken geodesic in a Riemannian manifold $(M,g)$ with smooth boundary $\partial M$ is a path $\alpha=\alpha_{x,\xi,z,\zeta}(t)$, where $z=\gamma_{x,\xi}(s)\in M$ for some $s\geq 0$, $\zeta\in S_zM$ and
\[
\alpha_{x,\xi,z,\zeta}(t)=\begin{cases}
\gamma_{x,\xi}(t),&t<s,\\
\gamma_{z,\zeta}(t-s),&t\geq s.
\end{cases}
\]
Denote 
\[
\ell(\alpha_{x,\xi,z,\zeta})=\inf\{t\vert \alpha_{x,\xi,z,\zeta}(t)\in\partial M\}.
\]
If $(M,g)$ is non-trapping, then $\ell(\alpha_{x,\xi,z,\zeta})<+\infty$ is the exit time of the broken geodesic $\alpha_{x,\xi,z,\zeta}$.
Recall that the \textit{broken scattering relation} of $(M,g)$ is then defined as
\[
\begin{split}
\mathcal{R}=\{&(x,\xi,y,\eta,t)\in SM\times SM\times [0,+\infty)\vert (x,\xi)\in\partial_+SM,(y,\eta)\in\partial_-SM,t=\ell(\alpha_{x,\xi,z,\zeta}),\\
&\quad\quad\quad\quad\quad\quad(\alpha_{x,\xi,z,\zeta}(t),\partial_t\alpha(x,\xi,z,\zeta)(t))=(y,\eta)\text{ for some }(z,\zeta)\in SM\}.
\end{split}
\]

\begin{figure}[htbp]
\centering
\includegraphics[width=0.4\textwidth]{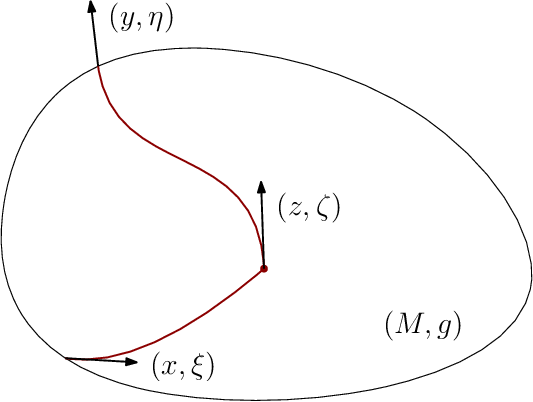}
\caption{Broken scattering relation}
\label{scatteringrelation}
\end{figure}
To simplify notation, we denote
\[
(x_0,\xi_0)\mathcal{R}_t(x_1,\xi_1)\quad\text{if and only if }(x_0,\xi_0,x_1,-\xi_1,t)\in\mathcal{R}.
\]
Clearly $\mathcal{L}\subset\mathcal{R}$. However, one cannot directly tell which element in $\mathcal{R}$ is also in $\mathcal{L}$. For the rest of this section, we assume that $\mathcal{L}$ is known.

It is proved in \cite{kurylev2010rigidity} that the broken scattering relation determines the set of \textit{boundary distance functions}
\[
\mathcal{R}(M)=\{r_x,x\in M\}\subset C(\partial M),
\]
where $r_x:\partial M\rightarrow\mathbb{R}$ is the function
\[
r_x(z)=\mathrm{dist}(x,z),\quad z\in\partial M.
\]
We emphasize here that one does not really know the location of $x$ for each boundary distance function in $\mathcal{R}(M)$.
It is well known that the set $\mathcal{R}(M)$ determines the metric $g$ up to a diffeomorphism fixing the boundary \cite{kachalov2001inverse}. However, we will see that one can only recover a subset of the broken scattering relation associated with $g_P$ from $\Lambda$.

This motivates us to define \textit{an abridged broken scattering relation} $\mathcal{R}'\subset \mathcal{R}$ such that 
\[
(x_0,\xi_0,x_1,-\xi_1,t)\in\mathcal{R}'
\]
 if one of the following two conditions holds:
\begin{enumerate}[label=(R\arabic*)]
\setlength{\itemindent}{-1em}
\item $\gamma_{x_0,\xi_0}(t_0)=\gamma_{x_1,\xi_1}(t_1)=z\in M$ with $\gamma_{x_0,\xi_0}([0,t_0]),\gamma_{x_1,\xi_1}([0,t_1])\subset M$ and  $t_0+t_1=t$,  where $t_0<\tau_R(x_0,\xi_0)$,  $t_1<\tau_R(x_1,\xi_1)$;
\item$\gamma_{x_0,\xi_0}(t_0)=\gamma_{x_1,\xi_1}(t_1)=z\in M$ with $\gamma_{x_0,\xi_0}([0,t_0]),\gamma_{x_1,\xi_1}([0,t_1])\subset M$ and  $t_0+t_1=t$,  where $t_0<\tau_c(x_0,\xi_0)$, $t_1<\tau_c(x_1,\xi_1)$, and in addition $(x_0,\xi_0)$ and $(x_1,\xi_1)$ are close enough.
\end{enumerate}

We remark here that if $(M,g)$ has no cut points, then $\mathcal{R}'=\mathcal{R}$. In general, $\mathcal{R}'$ could be a proper subset of $\mathcal{R}$. However, the set $\mathcal{R}'$ still contains enough information to determine $\mathcal{R}(M)$. This is the main result we aim to show in this section. We also denote
\[
(x_0,\xi_0)\mathcal{R}'_t(x_1,\xi_1)\quad\text{if and only if }(x_0,\xi_0,x_1,-\xi_1,t)\in\mathcal{R}'.
\]

\begin{definition}
Consider the set $S(z)$ of those $s\in (0,\mu_1(z))$ for which there are sequences $z_n,z'_n\rightarrow z$, $z_n,z'_n\in \partial M$, $z_n\neq z'_n$, $T_n\rightarrow 2s$ such that
\[
(z_n,\nu(z_n))\mathcal{R}'_{T_n}(z'_n,\nu(z'_n)).
\]
Define $\mu_2(z)=\inf S(z)$ if $S(z)\neq \emptyset$ and $\mu_2(z)=\mu_1(z)$ otherwise.
\end{definition}
\begin{lemma}
The function $\mu_2:\partial M\rightarrow \mathbb{R}$ satisfies
\begin{equation}
\min\{\mu_1,\tau_f(z),\tau_R(z,\nu)\}\leq \mu_2(z)\leq \min(\mu_1,\tau_f(z)),
\end{equation}
and $\mu_2(z)\geq\tau_b(z)$.
\end{lemma}
\begin{proof}

To prove the right inequality we note that since $\mathrm{d}\exp_{\partial M}$ is not one-to-one at $(z,\tau_f(z))$, there exist $z_n,z_n'\rightarrow z$, $z_n\neq z_n'$ and $s_n,t_n\rightarrow \tau_f(z)$ such that
\[
\gamma_{z_n,\nu(z_n)}(s_n)=\gamma_{z'_n,\nu(z'_n)}(t_n).
\]
Because $\tau_f(z)$ is strictly less than $\tau_c(z,\nu(z))$,
we can take $n$ large enough such that $s_n<\tau_c(z_n,\nu(z_n))$ and $t_n<\tau_c(z'_n,\nu(z'_n))$.
In terms of the relation $\mathcal{R}'$, these imply that, for $n$ sufficiently large,
\[
(z_n,\nu(z_n))\mathcal{R}'_{T_n}(z'_n,\nu(z'_n)),\quad T_n=s_t+t_n,
\]
with $T_n\rightarrow 2\tau_f(z)$.


To prove the left inequality, let us assume that $s<\min\{\mu_1,\tau_f(z),\tau_R(z,\nu)\}$ such that there are sequences $z_n,z'_n\rightarrow z$, $z_n,z'_n\in \partial M$, $z_n\neq z'_n$, $T_n\rightarrow 2s$ such that
\[
(z_n,\nu(z_n))\mathcal{R}'_{T_n}(z'_n,\nu(z'_n)).
\]
For $n$ sufficiently large, $T_n<2\tau_R(z,\nu(z))$, and by \cite[Lemma 2.2]{kurylev2010rigidity} we have
\begin{equation}\label{convergencetwogeodesics}
\gamma_{z_n,\nu(z_n)}(s_n)=\gamma_{z'_n,\nu(z'_n)}(t_n),\quad s_n\rightarrow s, \quad t_n\rightarrow s,\quad z_n,z'_n\rightarrow z,\quad z_n\neq z'_n.
\end{equation}
As $s<\tau_f(z)$, $\exp_{\partial M}$ is a local diffeomorphism near $(z,s)$, which contradicts \eqref{convergencetwogeodesics}.

Finally notice that $\tau_b(z)\leq \tau_f(z)$, $\tau_b(z)\leq \tau_R(z,\nu(z))$ and $2\tau_b(z)<\mu_1(z)$, so clearly $\tau_b(z)\leq \mu_2(z)$.

\end{proof}

Denote
\[
\tau_M(z)=\min\{\mu_1(z),\tau_R(z,\nu(z))\},\quad z\in \partial M.
\]
Next we will use the family of intersecting geodesics.

\begin{definition}
Let $z_0\in \partial M$ and $t_0>0$. Consider a family $\mathcal{F}(z_0,t_0)=\{U,\xi(\cdot),t(\cdot)\}$ where $U\subset\partial M$ is a neighborhood of $z_0$, $\xi:U\rightarrow SM$ is a smooth section, and $t:U\rightarrow \mathbb{R}$ is a smooth function. We say that $\mathcal{F}(z_0,t_0)$ is a family of focusing directions if $\xi(\cdot),t(\cdot)$ satisfy the following conditions
\begin{equation}
\begin{split}
&(z,\xi(z))\mathcal{R}'_{T(z)}(z_0,\nu(z_0)),\quad (z,\xi(z))\mathcal{R}'_{T(z,z')}(z',\xi(z')),\quad z,z'\in U,\\
&T(z)=t(z)+t_0,\quad T(z,z')=t(z)+t(z'),
\end{split}
\end{equation}
and
\begin{equation}
t(z_0)=t_0,\quad \mathrm{d}t(z)\vert_{z_0}=0,\quad \xi(z_0)=\nu(z_0).
\end{equation}
\end{definition}

We need the following important proposition.
\begin{proposition}
Let $z_0\in \partial M$, $t_0<\tau_M(z_0)$, and $\mathcal{F}(z_0,t_0)$ be a family of focusing directions. Then there exists a neighborhood $\widetilde{U}\subset U$ of $z_0$ such that
\[
\gamma_{z,\xi(z)}(t(z))=\gamma_{z_0,\nu(z_0)}(t_0),\quad\text{for }z\in\widetilde{U}.
\]
\end{proposition}
\begin{proof}
The proof is identical to that of \cite[Theorem 2.6]{kurylev2010rigidity} by noticing that $ \tau_M(z_0)\leq\tau_R(z,\nu(z))$ and therefore for $(z,\xi(z))$ near $(z_0,\nu(z_0))$ an abridged broken scattering relation is the same as the complete one since condition (R2) is satisfied.
\end{proof}

The next lemma is from \cite[Lemma 2.10]{kurylev2010rigidity}. For the self-containedness of this paper, we still include the proof. 
\begin{lemma}
The abridged broken scattering relation $\mathcal{R}'$ determines the boundary cut locus distance $\tau_b(z)$, $z\in\partial M$.
\end{lemma}
\begin{proof}
Fix $z_0\in\partial M$. Take $t_0>0$ and denote $x_0=\gamma_{z_0,\nu(z_0)}(t_0)$, $\eta=-\dot{\gamma}_{z_0,\nu(z_0)}(t_0)$. Clearly $z_0=\gamma_{x_0,\eta_0}(t_0)$.
If $t_0<\tau_M(z_0)$, the exponential map $\exp_{x_0}: S_{x_0}\widetilde{M}\times\mathbb{R}_+\rightarrow\widetilde{M}$ is a local diffeomorphism near $(\eta_0,t_0)$. As $\gamma_{x_0,\eta_0}$ hits $\partial M$ normally, all geodesics $\gamma_{x_0,\eta}$ hit $\partial M$ transversely for $\eta\in S_{x_0}M$ close to $\eta_0$. Then there exists a neighborhood $U\subset \partial M$ of $z_0$ a smooth section $\xi(z):U\rightarrow SU$ and a smooth function $t(z)$ such that
\[
\gamma_{z,\xi(z)}(t(z))=x_0,\quad z\in U.
\]
Then $\mathcal{F}(z_0,t_0)=\{U,\xi(\cdot),t(\cdot)\}$ is a family of focusing directions.

Recall that $\tau_b(z_0)<\tau_M(z_0)$. If $t_0<\tau_b(z_0)$, $z_0$ is the unique point on $\partial M$ closest to $x_0$. If $\tau_b(z_0)<t_0<\tau_M(z_0)$ there exists a point $w\in \partial M$, $w\neq z_0$ such that $\mathrm{dist}(x_0,w)<t_0$. Then there exists $0<s_0<t_0$ such that
\[
x_0=\gamma_{w,\nu(w)}(s_0).
\]
Therefore
\begin{equation}\label{scatteringzw}
(z,\xi(z))\mathcal{R}'_{t(z)+s_0}(w,\nu(w)),\quad z\in U.
\end{equation}

We will show that if $t_0<\tau_b(z_0)$, then there are no $w\in \partial M$ and $\mathcal{F}(z_0,t_0)$ satisfying \eqref{scatteringzw} with $s_0<t_0$. We assume the opposite and argue by contradiction. Then there exists a function $r(z)$ in $U$ with
\[
\gamma_{z,\xi(z)}(r(z))=\gamma_{w,\nu(w)}(t(z)-r(z)+s_0), \quad z\in U.
\]
We will prove that
\[
r_0:=\limsup_{z\rightarrow z_0} r(z)\leq t_0.
\]
If $r_0>t_0$, then there exists a sequence $z_n\rightarrow z_0$ with $r(z_n)\rightarrow r_0$. By the continuity of the exponential map, it follows that 
\[
y_0:=\gamma_{z_0,\nu(z_0)}(r_0)=\gamma_{w,\nu(w)}(t_0-r_0+s_0).
\]
By the triangle inequality, we have
\[
\mathrm{dist}(w,x_0)\leq\mathrm{dist}(x_0,y_0)+\mathrm{dist}(y_0,w)=(r_0-t_0)+(t_0-r_0+s_0)=s_0<t_0,
\]
which is a contradiction with $t_0<\tau_b(z_0)$.

Therefore, by making $U$ smaller if necessary, we have
\[
r(z)<\tau_M(z_0),\quad z\in U.
\]
Assume first that the geodesics $\gamma_{z_0,\nu(z_0)}$ and $\gamma_{w,\nu(w)}$ do not coincide. Applying \cite[Lemm 2.9]{kurylev2010rigidity} with $\gamma(\tau)=\gamma_{w,\nu(w)}(t_0+s_0-r_0+\tau)$ and $L=2t_0$, we obtain that $\gamma_{z_0,\nu(z_0)}(t_0)=\gamma_{w,\nu(w)}(s_0)$.
Since $s_0<t_0$ this contradicts with the definition of $\tau_b$. If  $\gamma_{z_0,\nu(z_0)}$ and $\gamma_{w,\nu(w)}$ do coincide, and since $w\neq z_0$, this implies that $\gamma_{z_0,\nu(z_0)}=w$. Then we would have $\mathrm{dist}(x_0,\partial M)\leq \mathrm{dist}(x_0,w)\leq s_0<t_0$, which is also a contradiction.

Now use the fact that the relation $\mathcal{R}'$ determines the function $\mu_2(z)$ satisfying $\tau_b(z)\leq \mu_2(z)$. Let $J(z_0)$ be the set of those $t_0\in [0,\mu_2(z_0)]$ for which there are $w\in\partial M$, $s_0<t_0$, and $\mathcal{F}(z_0,t_0)$ satisfying \eqref{scatteringzw}. If $\tau_b(z_0)<\mu_2(z_0)$, then $(\tau_b(z_0),\min\{\mu_2(z_0),\tau_M(z_0)\})\subset J(z_0)$. Remember also that $(0,\tau_b(z_0))\cap J(z_0)=\emptyset$. So if $J(z_0)=\emptyset$, we know that $\tau_b(z_0)=\mu_2(z_0)$. If $J(z_0)\neq \emptyset$, we can determine $\tau_b(z_0)$ as $\tau_b(z_0)=\inf J(z_0)$.
\end{proof}

The main result of this section is:
\begin{theorem}\label{brokentobdfs}
An abridged broken scattering relation $\mathcal{R}'$ determines the set $\mathcal{R}(M)\subset C(\partial M)$.
\end{theorem}
\begin{proof}
Because our boundary is convex, we will give a simplified proof than that of \cite[Theorem 2.13]{kurylev2010rigidity}. 
Since for any $x_0\in M\setminus\omega_{\partial M}$ there exists a unique point $z_0\in\partial M$ closest to $x_0$ and $x_0=\gamma_{z_0,\nu(z_0)}(t_0)$. Therefore there is a one-to-one correspondence between the set
\[
\{\gamma_{z_0,\nu(z_0)}(t_0)\vert z_0\in\partial M,t_0<\tau_b(z_0)\},
\]
which could be indexed by $(z_0,t_0)$, and the set $M\setminus\omega_{\partial M}$. 
Recall that $\tau_b(z_0)$ can be determined by $\mathcal{R}'$. As $\omega_{\partial M}=\{\gamma_{z_0,\nu(z_0)}(t_0)\vert z_0\in\partial M,t_0=\tau_b(z_0)\}$ has measure zero, one only needs consider the case $t_0<\tau_b(z_0)$.
Now fix $z_0\in\partial M$ and let $x_0=\gamma_{z_0,\nu(z_0)}(t_0)$ with $t_0<\tau_b(z_0)$. For any point $w_0\in\partial M$ we will determine $D(z_0,t_0,w_0):=\mathrm{dist}(x_0,w_0)$. Keep in mind that we do not know the precise location of $x_0$, but only that it is on the geodesic issuing from $z_0$ in normal direction and has arclength $t_0$ from $z_0$.

We claim that for $t_0<\tau_b(z_0)$,
\begin{equation}\label{Dztwexpression}
\begin{split}
D(z_0,t_0,w_0)=&\inf\{s\vert s\geq t_0,\text{there exist a focusing sequence }\mathcal{F}(z_0,t_0)\text{ and }\eta\in S_{w_0}M\\
&\quad\quad\quad\text{ such that } (z,\xi(z))\mathcal{R}'_{s+t(z)}(w_0,\eta)\text{ for any }(z,\xi(\cdot),t(\cdot))\in\mathcal{F}(z_0,t_0)\}.
\end{split}
\end{equation}
To prove the claim, we first take $x(s)$ to be a shortest path from $x_0$ to $w_0$ parametrized by the arclength such that $w_0=x(s_0)$. By the convexity of $\partial M$ the path $x(s)$ is a geodesic in $M$. Denote $\eta_0=-x'(s_0)$. As $t_0\leq\tau_b(z_0)<\tau_M(z_0)$, there is a family of focusing directions $\mathcal{F}(z_0,t_0)=\{U,\xi(\cdot),t(\cdot)\}$ such that
\[
\gamma_{z,\xi(z)}(t(z))=x_0,\quad\text{for all }z\in U.
\]
Then, since $\gamma_{z_0,\nu(z_0)}([0,t_0])$ is length minimizing,
\[
(z,\xi(z))\mathcal{R}'_{s_0+t(z)}(w_0,\eta_0),\quad \text{for all }z\in U.
\]
Since $z_0$ is the closest point on boundary to $x_0$, we know that $s_0\geq t_0$. Therefore the right hand side of \eqref{Dztwexpression} is less or equal than $D(z_0,t_0,w_0)$.

To prove that the right hand side of \eqref{Dztwexpression} is greater or equal than $D(z_0,t_0,w_0)$, assume for some $\mathcal{F}(z_0,t_0)=\{U,\xi(\cdot),t(\cdot)\}$ and $\eta$,
\[
 (z,\xi(z))\mathcal{R}'_{s+t(z)}(w_0,\eta),\quad\text{for all }z\in U,
 \]
 for some $s\geq t_0$.
 Then for some $r(z),\tau(z)$ with $r(z)+\tau(z)=s+t(z)$ we have $\gamma_{z,\xi(z)}(r(z))=\gamma(\tau(z))$. For the trivial case when the geodesics $\gamma_{z_0,\nu(z_0)}$ and $\gamma_{w_0,\eta}$ coincide, we have $t_0+s\geq\mathrm{dist}(z_0,x_0)+\mathrm{dist}(x_0,w_0)$ and then $s\geq \mathrm{dist}(x_0,w_0)$.  Otherwise we consider first the case when $\limsup r(z)=:r>t_0$ as $z\rightarrow z_0$. Denoting $\gamma_{z_0,\nu(z_0)}(r)=x_1$, we have
 \[
 \mathrm{dist}(w_0,x_0)\leq\mathrm{dist}(w_0,x_1)+\mathrm{dist}(w_1,x_0)\leq (s+t_0-r)+(r-t_0)\leq s.
 \]
 For the case $r\leq t_0$, \cite[Lemma 2.9]{kurylev2010rigidity} implies that
 \[
 \gamma_{z_0,\nu(z_0)}=\gamma_{w_0,\eta}(s).
 \]
 This also yields $s\geq\mathrm{dist}(w_0,x_0)$. This proves the claim and completes the proof of the theorem.
 
\end{proof}

\section{Determination of the $P$-wavespeed}\label{determinationPwavespeed}
In this section, we will show that the nonlinear displacement-to-traction map $\Lambda$ determines an abridged broken scattering relations with respect to $(\Omega, g_P)$. We will use the nonlinear interactions of one $P$-wave and one $S$-wave. The incoming $P$-wave and the newly generated $P$-wave would give a broken scattering relation.

When needed, we add superscripts $P/S$ to indicate the certain geometric objects are with respect to $g_{P/S}$.
Let $(x_1,\xi_1),(x_0,\xi_0)\in\partial_+S^P\Omega$, $(x_2,\xi_2)\in\partial_+ S^S\Omega$, $t_1,t_2,t_0\in (0,T)$ and $z_1=(t_1,x_1)$, $z_2=(t_2,x_2)$, $z_0=(t_0,x_0)$. 
Let
\[
\vartheta^{(1)}(t)=(t,\gamma^P_{x_1,\xi_1}(t-t_1)),\quad \vartheta^{(0)}(t)=(t,\gamma^P_{x_0,\xi_0}(t_0-t))
\]
be one forward and one backward null geodesics in $((0,T)\times \Omega,-\mathrm{d}t^2+g_P)$, and
\[
\vartheta^{(2)}(t)=(t,\gamma^S_{x_2,\xi_2}(t-t_2))
\]
be a forward null geodesic in $((0,T)\times \Omega,-\mathrm{d}t^2+g_S)$. 
\begin{figure}[htbp]
\centering
\includegraphics[width=0.27\textwidth]{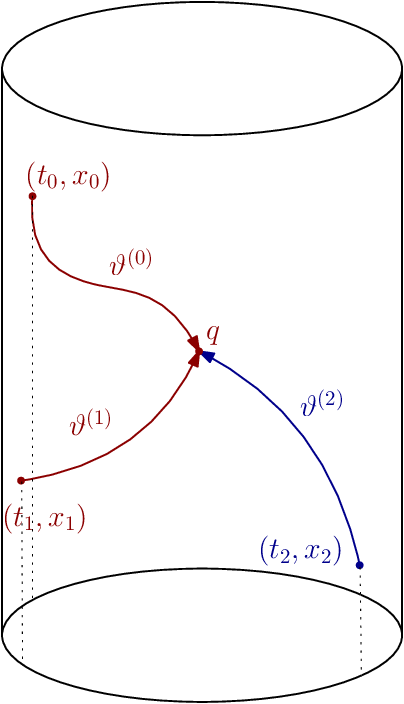}
\caption{The null geodesics $\vartheta^{(1)}$, $\vartheta^{(2)}$ and $\vartheta^{(0)}$ intersect at $q$.}
\label{PSPspacetime}
\end{figure}

We construct Gaussian beam solutions $u_{\kappa_j\varrho}^{(j)}$ concentrating near $\vartheta^{(j)}$, $j=1,2,0$, without reflections, as constructed in Section \ref{gaussianwithoutreflection}, of the following forms,
\[
\begin{split}
&u_{\kappa_1\varrho}^{(1)}=\sum_{k=0}^{2K}(\kappa_1\varrho)^{-k}\mathbf{a}^{(1)}_{P,k}e^{\mathrm{i}\kappa_1\varrho\varphi_P^{(1)}},\\
&u_{\kappa_2\varrho}^{(2)}=\sum_{k=0}^{2K}(\kappa_2\varrho)^{-k}\mathbf{a}^{(2)}_{S,k}e^{\mathrm{i}\kappa_2\varrho\varphi_S^{(2)}},\\
&u_{\kappa_0\varrho}^{(0)}=\sum_{k=0}^{2K}(\kappa_0\varrho)^{-k}\mathbf{a}^{(0)}_{P,k}e^{\mathrm{i}\kappa_0\varrho\varphi_P^{(0)}},
\end{split}
\]
where $\kappa_1,\kappa_2,\kappa_0$ are constants.
In the following, we drop the subscripts $P/S$ for $\varphi^{(j)}$ and $\mathbf{a}^{(j)}$ as there is no ambiguity.
The boundary values of $u^{(j)}_{\kappa_j\varrho}$ can be determined according to discussions in Section \ref{gaussianwithoutreflection}. Let $u^{(j)}$ be solutions to \eqref{linear_eqj} with $f^{(j)}=u^{(j)}_{\kappa_j\varrho}\vert_{(0,T)\times\partial\Omega}$, $j=1,2$ and $u^{(0)}$ be the solution to \eqref{backward_eq} with $f^{(0)}=u^{(0)}_{\kappa_0\varrho}\vert_{(0,T)\times\partial\Omega}$.\\

For the boundary values, once $(x_j,\xi_j,t_j)$ for $j=1,2,0$ have been decided, there are still some parameters to be chosen:
(1) values of $\kappa_1,\kappa_2,\kappa_0$;
(2) the values of $\varphi^{(j)}_2$ and $\mathbf{a}^{(j)}_{0,0}$ at the initial point $(t_j,x_j)$ (we introduce the notation $\sigma$ to include all these parameters and fixing the choices of $\varphi^{(j)}_k$ for $k\geq 3$ and $\mathbf{a}^{(j)}_{k,\ell}$ with $(k,\ell)\neq (0,0)$ at $(t_j,x_j)$ to be zero); (3) the parameter $\delta$ in the cut-off function.

 We now define
\begin{equation}\label{definitionD}
\mathfrak{D}_{\delta,\sigma}(x_1,\xi_1,t_1,x_2,\xi_2,t_2,x_0,\xi_0,t_0)=\lim_{\varrho\rightarrow+\infty}\frac{\mathrm{i}}{\varrho\kappa_1\kappa_2\kappa_0}\int_0^T\int_{\partial\Omega}\Lambda''(f_{\kappa_1\varrho}^{(1)},f_{\kappa_2\varrho}^{(2)})f_{\kappa_0\varrho}^{(0)}\mathrm{d}S\mathrm{d}t.
\end{equation}
Take $K'>0$ sufficiently large. Using the integral identity \eqref{integralform} and the construction of Gaussian beam solutions, one obtains
\[
\mathfrak{D}_{\delta,\sigma}(x_1,\xi_1,t_1,x_2,\xi_2,t_2,x_0,\xi_0,t_0)=\lim_{\varrho\rightarrow+\infty}\varrho^2\int_0^T\int_\Omega e^{\mathrm{i}\varrho S}\mathcal{A}\mathrm{d}x\mathrm{d}t,
\]
where
\[
S=\kappa_1\varphi^{(1)}+\kappa_2\varphi^{(2)}+\kappa_0\varphi^{(0)},
\]
and
\begin{align*}
\mathcal{A}=&\mathcal{G}(\mathbf{a}_0^{(1)}\otimes\nabla\varphi^{(1)},\mathbf{a}_0^{(2)}\otimes\nabla\varphi^{(2)},\mathbf{a}_0^{(0)}\otimes\nabla\varphi^{(0)})\\
=&\mathscr{B}[(\mathbf{a}_0^{(1)}\cdot\nabla\varphi^{(1)})(\mathbf{a}_0^{(2)}\cdot\nabla\varphi^{(0)})(\mathbf{a}_0^{(0)}\cdot\nabla\varphi^{(2)})+(\mathbf{a}_0^{(2)}\cdot\nabla\varphi^{(2)})(\mathbf{a}_0^{(1)}\cdot\nabla\varphi^{(0)})(\mathbf{a}_0^{(0)}\cdot\nabla\varphi^{(1)})\\
&\quad\quad\quad\quad+ (\mathbf{a}_0^{(2)}\cdot\nabla\varphi^{(1)})(\mathbf{a}_0^{(1)}\cdot\nabla\varphi^{(2)})(\mathbf{a}_0^{(0)}\cdot\nabla\varphi^{(0)})]\\
&+\frac{\mathscr{A}}{4}\left((\mathbf{a}_0^{(2)}\cdot\nabla\varphi^{(1)})(\mathbf{a}_0^{(1)}\cdot\nabla\varphi^{(0)})(\mathbf{a}_0^{(0)}\cdot\nabla\varphi^{(2)})+(\mathbf{a}_0^{(1)}\cdot\nabla\varphi^{(2)})(\mathbf{a}_0^{(2)}\cdot\nabla\varphi^{(0)})(\mathbf{a}_0^{(0)}\cdot\nabla\varphi^{(1)})\right)\\
&+(\lambda+\mathscr{B})[(\mathbf{a}_0^{(1)}\cdot\nabla\varphi^{(1)})(\mathbf{a}_0^{(2)}\cdot\mathbf{a}_0^{(0)})(\nabla\varphi^{(2)}\cdot\nabla\varphi^{(0)})+(\mathbf{a}_0^{(2)}\cdot\nabla\varphi^{(2)})(\mathbf{a}_0^{(1)}\cdot\mathbf{a}_0^{(0)})(\nabla\varphi^{(1)}\cdot\nabla\varphi^{(0)})\\
&\quad\quad+(\mathbf{a}_0^{(1)}\cdot\mathbf{a}_0^{(2)})(\nabla\varphi^{(1)}\cdot\nabla\varphi^{(2)})(\mathbf{a}_0^{(0)}\cdot\nabla\varphi^{(0)})]+2\mathscr{C}(\mathbf{a}_0^{(1)}\cdot\nabla\varphi^{(1)})(\mathbf{a}_0^{(2)}\cdot\nabla\varphi^{(2)})(\mathbf{a}_0^{(0)}\cdot\nabla\varphi^{(0)})\\
&+(\mu+\frac{\mathscr{A}}{4})\Big((\mathbf{a}_0^{(1)}\cdot\mathbf{a}_0^{(2)})(\nabla\varphi^{(1)}\cdot\mathbf{a}_0^{(0)})(\nabla\varphi^{(2)}\cdot\nabla\varphi^{(0)})+(\nabla\varphi^{(1)}\cdot\nabla\varphi^{(2)})(\mathbf{a}_0^{(1)}\cdot\mathbf{a}_0^{(0)})(\mathbf{a}_0^{(2)}\cdot\nabla\varphi^{(0)})\\
&\quad \quad+(\mathbf{a}_0^{(2)}\cdot\mathbf{a}_0^{(1)})(\nabla\varphi^{(2)}\cdot\mathbf{a}_0^{(0)})(\nabla\varphi^{(1)}\cdot\nabla\varphi^{(0)})+(\mathbf{a}_0^{(2)}\cdot\mathbf{a}_0^{(0)})(\mathbf{a}_0^{(1)}\cdot\nabla\varphi^{(0)})(\nabla\varphi^{(1)}\cdot\nabla\varphi^{(2)})\\
&\quad\quad+(\nabla\varphi^{(1)}\cdot\mathbf{a}_0^{(2)})(\nabla\varphi^{(2)}\cdot\nabla\varphi^{(0)})(\mathbf{a}_0^{(1)}\cdot\mathbf{a}_0^{(0)})+(\nabla\varphi^{(2)}\cdot\mathbf{a}_0^{(1)})(\nabla\varphi^{(1)}\cdot\nabla\varphi^{(0)})(\mathbf{a}_0^{(2)}\cdot\mathbf{a}_0^{(0)})\Big),
\end{align*}
In this section we focus on the recovery of an abridged broken scattering relation $\mathcal{R}'$ associated with $g_P$. The lens relation associated with $g_P$ is denoted by $\mathcal{L}$. We first introduce the notation for the set of $P/S$-wave light-like covectors at some point $q\in (0,T)\times\Omega$,
\[
L_q^{P/S,\pm}((0,T)\times\Omega):=\{(\tau,\xi)\in T^*_q((0,T)\times\Omega),\,\tau^2=c_{P/S}^2|\xi|^2,\,\pm \tau>0\}.
\]

\begin{theorem}
Assume $\lambda+3\mu+\mathscr{A}+2\mathscr{B}\neq 0$ in $\overline{\Omega}$.
Let $(x_1,\xi_1,x_0,-\xi_0,t_0-t_1)\notin\mathcal{L}$. Then the following statements hold:
\begin{enumerate}[label=\textnormal{(\arabic*)}]
\setlength{\itemindent}{-0.5em}
\item If $\mathfrak{D}_{\delta_j,\sigma}(x_1,\xi_1,t_1,x_2,\xi_2,t_2,x_0,\xi_0,t_0)\neq 0$ for some $(x_2,\xi_2,t_2)$, $\sigma$, and a sequence $\{\delta_j\}_{j=1}^\infty$ converging to $0$, then $(x_1,\xi_1,x_0,-\xi_0,t_0-t_1)\in\mathcal{R}$.
\item For $(x_1,\xi_1,x_0,-\xi_0,t)$ in a generic subset of $\mathcal{R}$ satisfying either of the two conditions \textnormal{(R1)(R2)}, then there exist $(x_2,\xi_2)\in\partial_+S^S\Omega$, $t_1,t_2,t_0\in (0,T)$ with $t_0-t_1=t$,  and $\sigma$ such that $\mathfrak{D}_{\delta,\sigma}(x_1,\xi_1,t_1,x_2,\xi_2,t_2,x_0,\xi_0,t_0)\neq 0$ for all $\delta$ sufficiently small.
\end{enumerate}
\end{theorem}
\begin{proof}
\noindent For (1), assume in opposite $(x_1,\xi_1,x_0,-\xi_0,t_0-t_1)\notin\mathcal{R}$, then $\vartheta^{(1)}\cap\vartheta^{(0)}=\emptyset$. So if $\delta$ is taken sufficiently small, the supports of $\mathbf{a}^{(1)},\mathbf{a}^{(2)}, \mathbf{a}^{(0)}$ are disjoint, and then
\[
\lim_{\varrho\rightarrow+\infty}\varrho^2\int_0^T\int_\Omega e^{\mathrm{i}\varrho S}\mathcal{A}\mathrm{d}x\mathrm{d}t=0.
\]

Next we prove (2).
If $(x_1,\xi_1,x_0,-\xi_0,t_0-t_1)\in\mathcal{R}$ and it satisfies either (R1) or (R2), then $\vartheta^{(1)}$ and $\vartheta^{(0)}$ intersect at a single point $q=(t_q,x_q)\in (0,T)\times\Omega$. Denote
\[
\zeta^{(1)}=\dot{\vartheta}^{(1)}(q)^{\flat,P}\in L_q^{P,+}((0,T)\times\Omega),\quad\zeta^{(0)}=\dot{\vartheta}^{(0)}(q)^{\flat,P}\in L_q^{P,+}((0,T)\times\Omega).
\]
Then we can choose $\zeta^{(2)}:=(\tau^{(2)},\eta^{(2)})\in L_q^{S,+}((0,T)\times\Omega)$, such that there are nonzero constants $\kappa_1,\kappa_2,\kappa_0$ with
\[
\kappa_1\zeta^{(1)}+\kappa_2\zeta^{(2)}+\kappa_0\zeta^{(0)}=0.
\]
For the choice of $\zeta^{(2)}$ and $\kappa_j$, $j=1,2,0$, we also denote $\zeta^{(1)}=(\tau^{(1)},\eta^{(1)})$, $\zeta^{(0)}=(\tau^{(0)},\eta^{(0)})$. Without loss of generality, we assume that $|\eta^{(1)}|=|\eta^{(0)}|=c_P^{-1}$ and then $\tau^{(1)}=\tau^{(0)}=1$. We seek $\zeta^{(2)}$ of the form $\zeta^{(2)}=\frac{1}{1+a}(a\zeta^{(1)}+\zeta^{(0)})=(1,\eta^{(2)})$ with $|\eta^{(2)}|=c_S^{-1}$. Then we have the relation
\[
(a\tau^{(1)}+\tau^{(0)})^2=c_S^2|a\eta^{(1)}+\eta^{(0)}|^2.
\]
By simple calculation, the above equality reduces to
\begin{equation}\label{equationPPSa}
(c_P^2-c_S^2)a^2+2(c_P^2-c_S^2\eta^{(0)}\cdot\eta^{(1)})a+c_P^2-c_S^2=0.
\end{equation}
Since $c_P>c_S$, the determinant of this quadratic equation is positive, then there exist two distinct values of $a$ satisfying the equation.\\

By the assumption that $(\Omega,g_S)$ is non-trapping and $\partial\Omega$ is convex with respect to $g_S$, the geodesic $\gamma^S_{x_q,-\eta^{(2)}}$ hits the boundary transversely at some point $x_2$ with exit direction $-\xi_2$. Take $t_2=t_q-\tau^S(x_q,-\eta^{(2)})$. Since $T>2\mathrm{diam}(\Omega,g_S)$, one can always choose $t_0,t_1\in (0,T)$ properly such that $t_2>0$. Denote $\vartheta^{(2)}(t)=(t,\gamma^S_{x_2,\xi_2}(t-t_2))$. Then one can see that $\vartheta^{(1)}\cap \vartheta^{(2)}\cap \vartheta^{(0)}=\{q\}$ and $\zeta^{(2)}=\dot{\vartheta}^{(2)}(q)^{\flat,S}$.
\begin{figure}[htbp]
\centering
\includegraphics[width=0.4\textwidth]{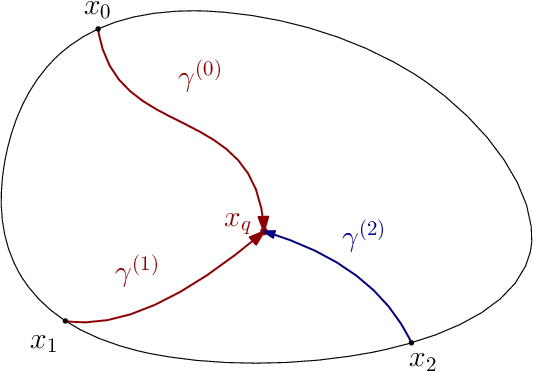}
\caption{The nonlinear interaction of a $P$-wave and an $S$-wave generates a new $P$-wave.}
\label{PSPnonlinear}
\end{figure}

Without loss of generality we can assume that
\[
\nabla\varphi^{(1)}(q)=\eta^{(1)},\quad \nabla\varphi^{(0)}(q)=\eta^{(0)},\quad \nabla\varphi^{(2)}(q)=\eta^{(2)},
\]
and 
\[
\begin{split}
\mathbf{a}^{(1)}_0(q)=\eta^{(1)},\quad\mathbf{a}^{(0)}_0(q)=\eta^{(0)},\quad \mathbf{a}^{(2)}_0(q)=\alpha,
\end{split}
\]
with $\alpha\perp\eta^{(2)}$. Take $\alpha\in\mathrm{span}\{\eta^{(1)},\eta^{(0)}\}$. By the choices of the phase functions we have
 \[
 S(q)=0,
 \] 
\[
\partial_tS(q)=0,\quad \nabla S(q)=0,
\]
and
\[
\Im S(t,x)\geq cd((t,x),q)^2 \text{ for }(t,x) \text{ in a neighborhood of }q,\text{ for some } c>0.
\]
 Then using the method of stationary phase, we obtain that for some constant $c_0$,
\[
\begin{split}
&c_0\lim_{\varrho\rightarrow+\infty}\varrho^2\int_0^T\int_\Omega e^{\mathrm{i}\varrho S}\mathcal{A}\mathrm{d}x\mathrm{d}t\\
=&\mathcal{A}(q)\\
=&(\lambda+2\mathscr{B})(q)\left((\alpha\cdot\eta^{(0)})(\eta^{(2)}\cdot\eta^{(0)})+(\alpha\cdot\eta^{(1)})(\eta^{(1)}\cdot\eta^{(2)})\right)\\
&+(3\mu+\mathscr{A})(q)\left((\eta^{(0)}\cdot \eta^{(1)})(\eta^{(0)}\cdot \eta^{(2)})(\alpha\cdot \eta^{(1)})+(\eta^{(0)}\cdot \eta^{(1)})(\eta^{(1)}\cdot \eta^{(2)})(\alpha\cdot \eta^{(0)})\right)
\end{split}
\]
If we write $\alpha=b\eta^{(1)}+\eta^{(0)}$ and $\cos\beta=\eta^{(0)}\cdot\eta^{(1)}$, the orthogonality condition $\alpha\cdot\eta^{(2)}=0$ then becomes $ab+(a+b)\cos\beta+1=0$. Then we obtain
\[
\mathcal{A}(q)=(\lambda+3\mu+\mathscr{A}+2\mathscr{B})(q)\cos\beta(ab\cos\beta+\cos\beta+a+b),
\]
which is generically non-zero if $(\lambda+3\mu+\mathscr{A}+2\mathscr{B})(q)\neq 0$.
\end{proof}

The above theorem allows us to recover an abridged broken scattering relation $\mathcal{R}'$ associated with $(\Omega,g_P)$. Then by Theorem \ref{brokentobdfs}, we can recover $g_P=c_P^{-2}\mathrm{d}s^2$ up to a diffeomorphism fixing the boundary $\partial \Omega$.
Since $g_P$ is conformal to the Euclidean metric, one can uniquely determine $c_P$ in $\overline{\Omega}$ (cf., for example, \cite[Proposition 3.3]{lionheart1997conformal}).

\begin{remark}
One obstacle for recovering a sufficiently large subset of the broken scattering relation for $S$-waves is that the plane spanned two $S$-wave light-like covectors does not necessarily have intersection with the set of light-like covectors of $P$-wave. On technical level, if one switches the roles of $c_P$ and $c_S$ in equation \eqref{equationPPSa}, then the quadratic polynomial might not have real roots.
\end{remark}
\section{Determination of the $S$-wavespeed}\label{determinationSwavespeed}
To recover the $S$-wavespeed, we use the already determined $P$-wavespeed $c_P$. We still choose $u^{(1)}$, $u^{(0)}$ to be $P$-waves and $u^{(2)}$ to be $S$-waves as the previous section.

Fix $q=(\frac{T}{2},x_q)\in (0,T)\times\Omega$, and take $x_1=x_0$ to the closest point on $\partial\Omega$ to $x_q$ with respect to $g_P$. Then $\gamma_{x_1,\nu_P(x_1)}(\mathrm{dist}_P(x_q,x_1))=x_q$. Denote $t_1=\frac{T}{2}-\mathrm{dist}_P(x_q,x_1)$, $t_0=\frac{T}{2}+\mathrm{dist}_P(x_q,x_1)$  and $\xi_1=\xi_0=\nu_P(x_1)$ and denote
\[
\vartheta^{(1)}(t)=(t,\gamma^P_{x_1,\xi_1}(t-t_1)),\quad \vartheta^{(0)}(t)=(t,\gamma^P_{x_1,\xi_1}(t_0-t)),
\]
which are $P$ null geodesics. Since $\gamma^P_{x_1,\xi_1}=\gamma^P_{x_0,\xi_0}$ is length minimizing between the point $x_q$ and the boundary $\partial\Omega$, we have
\[
 \vartheta^{(1)}\cap  \vartheta^{(0)}=\{q\}.
 \]
  Denote
\[
\zeta^{(1)}=\dot{\vartheta}^{(1)}(\frac{T}{2})^{\flat,P},\quad \zeta^{(0)}=\dot{\vartheta}^{(0)}(\frac{T}{2})^{\flat,P}.
\]
We can write $\zeta^{(1)}=(1,\eta^{(1)})$, $\zeta^{(0)}=(1,\eta^{(0)})$ with $\eta^{(0)}=-\eta^{(1)}$. We first choose some $(x_2,\xi_2,t_2)$ such that $\mathfrak{D}_{\delta_j,\sigma}(x_1,\xi_1,t_1,x_2,\xi_2,t_2,x_0,\xi_0,t_0)\neq 0$ for some $\sigma$ and a sequence $\{\delta_j\}$ converging to $0$. We claim such choice can be made.

To prove the claim, we notice that there exists $\zeta^{(2)}\in L_q^{S,+}((0,T)\times\Omega)$ such that
\[
\kappa_1\zeta^{(1)}+\kappa_2\zeta^{(2)}+\kappa_0\zeta^{(0)}=0,
\]
for some nonzero $\kappa_1,\kappa_2,\kappa_0$. Without loss of generality, one can assume $\zeta^{(1)}=(1,c_P^{-1},0,0)$, $\zeta^{(0)}=(1,-c_P^{-1},0,0)$. Then we can take $\zeta^{(2)}=(1,c_S^{-1},0,0)$ and $\kappa_1=c_P+c_S$, $\kappa_2=-2c_S$, $\kappa_0=-(c_P-c_S)$. Then one can determine $x_2,\xi_2,t_2$ as in previous section.

Once one such $(x_2,\xi_2,t_2)$ has been chosen, we fix the $S$-wave null geodesic
\[
\vartheta^{(2)}(t)=(t,\gamma^S_{x_2,\xi_2}(t-t_2)),
\]
and let $(x_1,\xi_1,t_1)$ and $(x_0,\xi_0,t_0)$ vary. We consider the following subset of $(0,T)\times\Omega$,
\[
\begin{split}
\Theta'=\big\{q'\in (0,T)\times\Omega\vert &\text{there exist }(x_1,\xi_1,t_1,x_0,\xi_0,t_0) \text{ such that }\vartheta^{(1)}\cap\vartheta^{(0)}=\{q'\}\\
&\text{and }\mathfrak{D}_{\delta_j,\sigma}(x_1,\xi_1,t_1,x_2,\xi_2,t_2,x_0,\xi_0,t_0)\neq 0\\
&\text{for some }\sigma\text{ and a sequence }\delta_j\text{ converging to }0\big\}.
\end{split}
\]
By previous discussions it is obvious that $q\in \Theta'$. We claim that $\Theta'$ is a subset of $\vartheta^{(2)}$ containing a neighborhood of $q$. To prove the claim first assume that $q'\notin\vartheta^{(2)}$, then $\vartheta^{(1)}\cap\vartheta^{(2)}\cap \vartheta^{(0)}=\emptyset$. By taking $\delta$ small enough, the supports of $\mathbf{a}^{(1)},\mathbf{a}^{(2)},\mathbf{a}^{(0)}$ would have disjoint union, and then 
\[
\mathfrak{D}_{\delta,\sigma}(x_1,\xi_1,t_1,x_2,\xi_2,t_2,x_0,\xi_0,t_0)=0.
\]
Now for any $q'=(t',x_q')\in \vartheta^{(2)}$ sufficiently close to $q$, we choose $\zeta'_1=(1,\eta_1')\in L_{q'}^{P,+}$ sufficiently close to $\zeta^{(1)}\in L_{q}^{P,+}$. Let $x_1'=\gamma_{q',\eta_1'}(\tau(q',\eta_1'))$, $\xi_1'= \dot{\gamma}_{q',\eta_1'}(\tau(q',\eta_1'))^{\flat,P}$ and $t_1'=t'-\tau(q',\eta_1')$.
Then if we let $\vartheta_1'(t)=(t,\gamma_{x_1',\xi_1'}(t-t_1'))$, then $\vartheta_1'(t')=q'$. Then there exists $\zeta_0'=(1,\eta_0')\in L_{q'}^{P,+}$ sufficiently close to $\zeta^{(0)}\in L_{q}^{P,+}$ such that
\[
\kappa_1\zeta_1'+\kappa_2\zeta^{(2)}+\kappa_0\zeta_0'=0,
\]
with some nonzero constants $\kappa_1,\kappa_2,\kappa_0$. Then let $x'_0=\gamma_{q',\zeta_0'}(\tau(q',\eta_0'))$, $\xi_0'= \dot{\gamma}_{q',\eta_0'}(\tau(q',\eta_0'))^{\flat,P}$ and $t_0'=t'+\tau(q',\eta_1')$. Let $\vartheta_0'(t)=(t,\gamma_{x_0',\xi_0'}(t_0'-t))$.
Since $\tau^S_R$ is continuous, $\tau^S(q',\eta_1')<\tau^S_R(q',\eta_1')$, $\tau^S(q',\eta_0')<\tau^S_R(q',\eta_0')$ for $(q',\zeta_1')$ sufficiently close to $(q,\zeta^{(1)})$. Then $\vartheta_1'\cap \vartheta_0'=\{q'\}$. Similar to previous discussions, for $\delta$ arbitrarily small, we have $\mathfrak{D}_{\delta,\sigma}(x'_1,\xi'_1,t'_1,x_2,\xi_2,t_2,x'_0,\xi'_0,t'_0)\neq 0$. This proves that $\Theta'$ contains an open neighborhood of $q$ on $\vartheta^{(2)}$.

This enables us to recover the tangent vector $\dot{\vartheta}^{(2)}(\frac{T}{2}):=(1,\eta_2)$, where $|\eta_2|_{g_S}=1$. Then one can simply recover $c_S^2(x_q)=|\eta_2|^2$. Since $x_q$ could be any point in $\Omega$, we can recover $c_S$ in $\overline{\Omega}$.

\section{Determination of the density and two nonlinear parameters}\label{uniquenessnonlinearparameters}
With $c_P$ and $c_S$ already known, we prove the unique determination of the density $\rho$ and two nonlinear parameters $\mathscr{A}$, $\mathscr{B}$. We adopt the procedure already taken in \cite{uhlmann2024determination}, where one can find more details.

Fix $q=(\frac{T}{2},x_q)$. Assume $x_1$ is the closest point on $\partial\Omega$ to $x_q$ with respect to $g_P$, and $x_0$ is the closest point on $\partial\Omega$ with respect to $g_S$. Then take $\xi_1=\nu(x_1)$, $\xi_0=\nu(x_0)$ and $t_1=\frac{T}{2}-\mathrm{dist}_P(x_q,x_1)$, $t_0=\frac{T}{2}+\mathrm{dist}_S(x_q,x_0)$. Denote
\[
\vartheta^{(1)}(t)=(t,\gamma^P_{x_1,\xi_1}(t-t_1)),\quad \vartheta^{(0)}(t)=(t,\gamma^S_{x_0,\xi_0}(t_0-t)).
\]
Clearly,
\[
q\in \vartheta^{(1)}\cap\vartheta^{(0)}.
\]
Let
\[
\zeta^{(1)}=\dot{\vartheta}^{(1)}(\frac{T}{2})^{\flat,P},\quad \zeta^{(0)}=\dot{\vartheta}^{(0)}(\frac{T}{2})^{\flat,S}.
\]
We claim that there exists $\zeta^{(2)}\in L_q^{S,+}((0,T)\times\Omega)$ such that
\[
\kappa_1\zeta^{(1)}+\kappa_2\zeta^{(2)}+\kappa_0\zeta^{(0)}=0,
\]
for some nonzero $\kappa_1,\kappa_2,\kappa_0$. Without loss of generality, denote $\zeta^{(1)}=(1,\eta^{(1)})$, $\zeta^{(0)}=(1,\eta^{(0)})$ with $|\eta^{(1)}|=c_P^{-1}$, $|\eta^{(2)}|=c_S^{-1}$, and seek $\zeta^{(2)}$ of the form $\zeta^{(2)}=\frac{1}{1+b}(\zeta^{(1)}+b\zeta^{(0)})=:(1,\eta^{(2)})$. Then we have the equation
\[
(1+b)^2=c_S^2(c_P^{-2}+2b\eta^{(1)}\cdot\eta^{(0)}+b^2c_S^{-2}),
\]
which could be simplified into
\[
2b(c_S^{-2}-\eta^{(1)}\cdot\eta^{(0)})=c_P^{-2}-c_S^{-2}.
\]
Since $|\eta^{(1)}\cdot\eta^{(0)}|\leq c_P^{-1}c_S^{-1}<c_S^{-2}$, the above equation has a solution. Therefore the claim is proved. Take $t_2=\frac{T}{2}-\tau^S(x_q,-\eta^{(2)})$, $\xi_2=\dot{\gamma}^S_{x_q,-\eta^{(2)}}(\tau^S(x_q,-\eta^{(2)}))$
and
\[
\vartheta^{(2)}(t)=(t,\gamma^S_{x_2,\xi_2}(t-t_2)).
\]
Since $\gamma^S_{x_0,\xi_0}([0,t_0-\frac{T}{2}])$ is length-minimizing in $(\Omega,g_S)$, $\vartheta^{(2)}$ and $\vartheta^{(0)}$ intersects only at $q$ in $[\frac{T}{2},T)\times\Omega$. Also because $\gamma^P_{x_1,\xi_1}([0,\frac{T}{2}-t_1])$ is length-minimizing in $(\Omega,g_P)$ and $c_P>c_S$, $\vartheta^{(2)}$ and $\vartheta^{(1)}$ intersects only at $q$ in $(0,\frac{T}{2}]\times\Omega$. Therefore
\[
\vartheta^{(1)}\cap \vartheta^{(2)}\cap \vartheta^{(0)}=\{q\}.
\]

We then use Gaussian beam solutions $u_{\kappa_j\varrho}^{(j)}$ concentrating near $\vartheta^{(j)}$, $j=1,2,0$, without reflections as constructed in Section \ref{gaussianwithoutreflection}, of the following forms,
\[
\begin{split}
&u_{\kappa_1\varrho}^{(1)}=\sum_{k=0}^{2K}(\kappa_1\varrho)^{-k}\mathbf{a}^{(1)}_{P,k}e^{\mathrm{i}\kappa_1\varrho\varphi_P^{(1)}},\\
&u_{\kappa_2\varrho}^{(2)}=\sum_{k=0}^{2K}(\kappa_2\varrho)^{-k}\mathbf{a}^{(2)}_{S,k}e^{\mathrm{i}\kappa_2\varrho\varphi_S^{(2)}},\\
&u_{\kappa_0\varrho}^{(0)}=\sum_{k=0}^{2K}(\kappa_0\varrho)^{-k}\mathbf{a}^{(0)}_{S,k}e^{\mathrm{i}\kappa_0\varrho\varphi_S^{(0)}}.
\end{split}
\]
Since $c_P$ and $c_S$ are already known, we can construct above Gaussian beams such that
\[
\begin{split}
\mathbf{a}^{(1)}_0(q)&=\rho^{-1/2}(x_q)c_P\eta^{(1)},\\
\mathbf{a}^{(2)}_0(q)&=\rho^{-1/2}(x_q)\alpha_2,\\
\mathbf{a}^{(0)}_0(q)&=\rho^{-1/2}(x_q)\alpha_0,
\end{split}
\]
where $\alpha_2\cdot\eta^{(2)}=0$, $\alpha_0\cdot\eta^{(0)}=0$, $|\alpha_2|=|\alpha_0|=1$. Then one can calculate
\[
\begin{split}
\mathcal{A}(q)=&\rho^{-3/2}\mathscr{B}(x_q)(\alpha_2\cdot\eta^{(0)})(\alpha_0\cdot\eta^{(2)})+(\lambda+\mathscr{B})(x_q)(\alpha_2\cdot\alpha_0)(\eta^{(2)}\cdot\eta^{(0)})\\
&+\rho^{-3/2}(\mu+\frac{\mathscr{A}}{2})(x_q)\left[(\alpha_2\cdot\eta^{(1)})(\eta^{(1)}\cdot\eta^{(0)})(\alpha_0\cdot\eta^{(2)})+(\eta^{(1)}\cdot\eta^{(2)})(\alpha_2\cdot\eta^{(0)})(\alpha_0\cdot\eta^{(1)})\right]\\
&+\rho^{-3/2}(2\mu+\frac{\mathscr{A}}{2})(x_q)\left[(\eta^{(1)}\cdot\alpha_2)(\eta^{(1)}\cdot\alpha_0)(\eta^{(2)}\cdot\eta^{(0)})+(\alpha_2\cdot\alpha_0)(\eta^{(1)}\cdot\eta^{(0)})(\eta^{(1)}\cdot\eta^{(2)})\right].
\end{split}
\]

Notice that $\eta^{(2)}\in\mathrm{span}\{\eta^{(1)},\eta^{(0)}\}$, we can first choose $\alpha_0=\alpha_2=\alpha\perp \mathrm{span}\{\eta^{(1)},\eta^{(0)}\}$. Then
\[
\mathcal{A}(q)=\rho^{-3/2}(\lambda+\mathscr{B})(x_q)(\eta^{(2)}\cdot\eta^{(0)})+\rho^{-3/2}(2\mu+\frac{\mathscr{A}}{2})(x_q)(\eta^{(1)}\cdot\eta^{(0)})(\eta^{(1)}\cdot\eta^{(2)}).
\]
By perturbing $\eta^{(1)}$, we can recover
\[
\rho^{-3/2}(\lambda+\mathscr{B})(x_q)\quad\text{and}\quad \rho^{-3/2}(2\mu+\frac{\mathscr{A}}{2})(x_q)
\]
from $\mathcal{A}(q)$.

Next we take $\alpha_0,\alpha_2\in\mathrm{span}\{\eta^{(1)},\eta^{(0)}\}$. Then $\eta^{(1)},\eta^{(0)},\eta^{(2)},\alpha_0,\alpha_2$ are all in the same plane. Assume
\[
\eta^{(1)}\cdot\eta^{(0)}=\cos\alpha,\quad \eta^{(1)}\cdot\eta^{(2)}=\cos\psi.
\]
Then
\[
\begin{split}
\alpha_2\cdot\eta^{(0)}&=\cos(\frac{\pi}{2}-\alpha+\psi)=\sin(\alpha-\psi),\\
\alpha_0\cdot\eta^{(2)}&=\cos(\frac{\pi}{2}+\alpha-\psi)=-\sin(\alpha-\psi),\\
\alpha_2\cdot\eta^{(1)}&=\cos(\frac{\pi}{2}+\psi)=-\sin\psi,\\
\alpha_0\cdot\eta^{(1)}&=\cos(\frac{\pi}{2}+\alpha)=-\sin\alpha.
\end{split}
\]
Then we can recover
\[
\begin{split}
&-\rho^{-3/2}\mathscr{B}(x_q)\sin^2(\alpha-\psi)+\rho^{-3/2}(\mu+\frac{\mathscr{A}}{2})(x_q)(\sin\psi\cos\alpha\sin(\alpha-\psi)-\cos\psi\sin(\alpha-\psi)\sin\alpha)\\
=&-\rho^{-3/2}(\mu+\frac{\mathscr{A}}{2}+\mathscr{B})(x_q)\sin^2(\alpha-\psi)
\end{split}
\]
from $\mathcal{A}(q)$, since $\rho^{-3/2}(\lambda+\mathscr{B})$ and $\rho^{-3/2}(2\mu+\frac{\mathscr{A}}{2})$ are already recovered. Note that $\sin(\alpha-\psi)\neq 0$, so we can recover $\rho^{-3/2}(\mu+\frac{\mathscr{A}}{2}+\mathscr{B})(x_q)$.

Since $x_q$ could be any point in $\Omega$, one can determine $\mathscr{A},\mathscr{B},\rho$ in $\overline{\Omega}$ with $\frac{\lambda}{\rho}$, $\frac{\mu}{\rho}$ already known.
\section{Determination of the last nonlinear parameter}\label{recoveryC}
Now there is only one parameter $\mathscr{C}$ remained to be determined. Under the more restrictive simplicity condition or foliation condition, the uniqueness of $\mathscr{C}$ has already been proved in \cite{uhlmann2021inverse}. If only the non-trapping condition is supposed, we prove uniqueness under one additional condition $\nabla(\mu^2/\rho)\neq 0$ in this section.

\subsection{Subprincipal amplitude of $S$-waves}
We calculate the explicit form of $\langle\nabla\varphi,\mathbf{a}_1\rangle$ for $S$-wave Gaussian beams. Similar calculations can also be found in \cite{yi2026leakage}. 
We can write the equation
\[
\langle\mathrm{i}T(x,\varphi)\mathbf{a}_0-E(x,\varphi)\mathbf{a}_1,\nabla\varphi\rangle=0,\quad \text{on }\vartheta
\]
as
\[
\begin{split}
&2\rho\partial_t\varphi\langle\partial_t\mathbf{a}_0,\nabla\varphi\rangle-2\mu\langle\nabla\mathbf{a}_0,\nabla\varphi\otimes\nabla\varphi\rangle-(\lambda+\mu)(\nabla\cdot\mathbf{a}_0)|\nabla\varphi|^2-(\nabla\mu\cdot\mathbf{a}_0)|\nabla\varphi|^2\\
&-\mathrm{i}(\lambda+\mu)\langle\nabla\varphi,\mathbf{a}_1\rangle|\nabla\varphi|^2=0\quad \text{on }\vartheta.
\end{split}
\]
Taking inner product of \eqref{Stransportinvariant} and $\nabla\varphi$, one obtains
\[
2\rho\partial_t\varphi\langle\partial_t\mathbf{a}_0,\nabla\varphi\rangle-2\mu\langle\nabla\mathbf{a}_0,\nabla\varphi\otimes\nabla\varphi\rangle+\rho\langle\mathbf{a}_0,\nabla(\partial_t\varphi)^2\rangle-\mu\langle\mathbf{a}_0,\nabla|\nabla\varphi|^2\rangle=0\quad \text{on }\vartheta.
\]
Noticing also that
\[
\begin{split}
0=&\nabla(\rho(\partial_t\varphi)^2-\mu|\nabla\varphi|^2)\\
=&\nabla\rho(\partial_t\varphi)^2+\rho\nabla(\partial_t\varphi)^2-\nabla\mu|\nabla\varphi|^2-\mu\nabla|\nabla\varphi|^2\\
=&\rho\nabla(\partial_t\varphi)^2-\mu\nabla|\nabla\varphi|^2+(\mu\rho^{-1}\nabla\rho-\nabla\mu)|\nabla\varphi|^2\quad \text{on }\vartheta.
\end{split}
\]
Summarizing above identities, we get
\[
\langle\nabla\varphi,\mathbf{a}_1\rangle=\mathrm{i}\nabla\cdot\mathbf{a}_0+\frac{\mathrm{i}}{\lambda+\mu}\langle2\nabla\mu-\mu\rho^{-1}\nabla\rho,\mathbf{a}_0\rangle\quad \text{on }\vartheta.
\]
 Then
\[
\mathrm{i}\nabla\varphi\cdot\mathbf{a}_1+\nabla\cdot\mathbf{a}_0=-\frac{1}{\lambda+\mu}\langle2\nabla\mu-\mu\rho^{-1}\nabla\rho,\mathbf{a}_0\rangle\quad \text{on }\vartheta.
\]
We will use the asymptotic behavior
\[
\begin{split}
e^{-\mathrm{i}\varrho\varphi}\nabla\cdot u_\varrho=&\mathrm{i}\varrho\nabla\varphi\cdot\mathbf{a}_0+\mathrm{i}\nabla\varphi\cdot\mathbf{a}_1+\nabla\cdot\mathbf{a}_0+\mathcal{O}(\varrho^{-1})\\
=&-\frac{1}{\lambda+\mu}\langle2\nabla\mu-\mu\rho^{-1}\nabla\rho,\mathbf{a}_0\rangle+\mathcal{O}(\varrho^{-1})
\end{split}
\]
near $\vartheta$.
\subsection{Determination of $\mathscr{C}$}
To recover $\mathscr{C}$, we take $u_\varrho^{(1)}$, $u_\varrho^{(0)}$ to be $P$-wave Gaussian beams, and $u_\varrho^{(2)}$ an $S$-wave Gaussian beam as in Section \ref{determinationPwavespeed}. We can take $\vartheta^{(1)}$, $\vartheta^{(2)}$ and $\vartheta^{(0)}$ so that they intersect only at $q$. One can recover
\[
\int_0^T\int_\Omega\mathscr{C}(\nabla\cdot u^{(1)})(\nabla\cdot u^{(2)})(\nabla\cdot u^{(0)})\mathrm{d}x\mathrm{d}t
\]
from the integral identity \eqref{integralform}, since all other parameters are already determined.
Now assume $\nabla(\mu^2\rho^{-1})(x_q)\neq 0$.
Using the asymptotic behaviors of Gaussian beams and the method of stationary phase, one can recover
\[
\mathscr{C}(\nabla\varphi^{(1)}\cdot\mathbf{a}_0^{(1)})\langle2\nabla\mu-\mu\rho^{-1}\nabla\rho,\mathbf{a}^{(2)}_0\rangle(\nabla\varphi^{(0)}\cdot\mathbf{a}_0^{(0)})
\]
at $q$. By perturbing $\vartheta^{(1)},\vartheta^{(0)}$ if needed, one can guarantee $\langle2\nabla\mu-\mu\rho^{-1}\nabla\rho,\mathbf{a}^{(2)}_0\rangle\neq 0$ at $q$. Then we can determine $\mathscr{C}(x_q)$.

\section*{Acknowledgement}
JZ thanks the enlightening discussions with Yuchao Yi.  \\
\bibliographystyle{abbrv}
\bibliography{biblio}

\begin{thebibliography}{10}

\bibitem{agemi2000global}
R.~Agemi.
\newblock Global existence of nonlinear elastic waves.
\newblock {\em Inventiones Mathematicae}, 142(2):225--250, 2000.

\bibitem{alexakis2024inverse}
S.~Alexakis, H.~Isozaki, M.~Lassas, and T.~Tyni.
\newblock Inverse scattering problems for non-linear wave equations on
  lorentzian manifolds.
\newblock {\em arXiv preprint arXiv:2411.09354}, 2024.

\bibitem{balehowsky2022inverse}
T.~Balehowsky, A.~Kujanp{\"a}{\"a}, M.~Lassas, and T.~Liimatainen.
\newblock An inverse problem for the relativistic {B}oltzmann equation.
\newblock {\em Communications in Mathematical Physics}, 396(3):983--1049, 2022.

\bibitem{bao2014sensitivity}
G.~Bao and H.~Zhang.
\newblock Sensitivity analysis of an inverse problem for the wave equation with
  caustics.
\newblock {\em Journal of the American Mathematical Society}, 27(4):953--981,
  2014.

\bibitem{belishev1996boundary}
M.~Belishev and A.~Katchalov.
\newblock Boundary control and quasiphotons in the problem of reconstruction of
  a {R}iemannian manifold via dynamic data.
\newblock {\em Journal of Mathematical Sciences}, 79(4):1172--1190, 1996.

\bibitem{bhattacharyya2018local}
S.~Bhattacharyya.
\newblock Local uniqueness of the density from partial boundary data for
  isotropic elastodynamics.
\newblock {\em Inverse Problems}, 34(12):125001, 2018.

\bibitem{chen2021detection}
X.~Chen, M.~Lassas, L.~Oksanen, and G.~P. Paternain.
\newblock Detection of {H}ermitian connections in wave equations with cubic
  non-linearity.
\newblock {\em Journal of the European Mathematical Society}, 24(7):2191--2232,
  2021.

\bibitem{chen2025inverse}
X.~Chen, M.~Lassas, L.~Oksanen, and G.~P. Paternain.
\newblock An inverse problem for the standard model of particle physics.
\newblock {\em arXiv preprint arXiv:2505.24454}, 2025.

\bibitem{chen2025stable}
X.~Chen, S.~Lu, and R.~Zhang.
\newblock Stable inversion of potential in nonlinear wave equations with cubic
  nonlinearity.
\newblock {\em Mathematische Annalen}, 392(3):4283--4314, 2025.

\bibitem{ciarlet2021mathematical}
P.~G. Ciarlet.
\newblock {\em Mathematical elasticity: {T}hree-dimensional elasticity}.
\newblock SIAM, 2021.

\bibitem{de2018nonlinear}
M.~de~Hoop, G.~Uhlmann, and Y.~Wang.
\newblock Nonlinear interaction of waves in elastodynamics and an inverse
  problem.
\newblock {\em Mathematische Annalen}, 376(1-2):765--795, 2020.

\bibitem{de2003finite}
W.~De~Lima and M.~Hamilton.
\newblock Finite-amplitude waves in isotropic elastic plates.
\newblock {\em Journal of sound and vibration}, 265(4):819--839, 2003.

\bibitem{dos2016calderon}
D.~Dos Santos~Ferreira, Y.~Kurylev, M.~Lassas, and M.~Salo.
\newblock The {C}alder{\'o}n problem in transversally anisotropic geometries.
\newblock {\em Journal of the European Mathematical Society},
  18(11):2579--2626, 2016.

\bibitem{feizmohammadi2019timedependent}
A.~Feizmohammadi, J.~Ilmavirta, Y.~Kian, and L.~Oksanen.
\newblock Recovery of time dependent coefficients from boundary data for
  hyperbolic equations.
\newblock {\em J. Spectr. Theory}, 11(3):1107--1143, 2021.

\bibitem{feizmohammadi2021inverse}
A.~Feizmohammadi, M.~Lassas, and L.~Oksanen.
\newblock Inverse problems for nonlinear hyperbolic equations with disjoint
  sources and receivers.
\newblock In {\em Forum of Mathematics, Pi}, volume~9, page e10. Cambridge
  University Press, 2021.

\bibitem{feizmohammadi2019inverse}
A.~Feizmohammadi and L.~Oksanen.
\newblock An inverse problem for a semi-linear elliptic equation in riemannian
  geometries.
\newblock {\em Journal of Differential Equations}, 269(6):4683--4719, 2020.

\bibitem{feizmohammadi2019recovery}
A.~Feizmohammadi and L.~Oksanen.
\newblock Recovery of zeroth order coefficients in non-linear wave equations.
\newblock {\em Journal of the Institute of Mathematics of Jussieu},
  21(2):367–393, 2022.

\bibitem{gol1961interaction}
Z.~Gol’dberg.
\newblock Interaction of plane longitudinal and transverse elastic waves.
\newblock {\em Sov. Phys. Acoust}, 6(3):306--310, 1961.

\bibitem{gurtin1981topics}
M.~E. Gurtin.
\newblock {\em Topics in finite elasticity}.
\newblock SIAM, 1981.

\bibitem{hansen2003propagation}
S.~Hansen and G.~Uhlmann.
\newblock Propagation of polarization in elastodynamics with residual stress
  and travel times.
\newblock {\em Mathematische Annalen}, 326(3):563--587, 2003.

\bibitem{hintz2024inverse}
P.~Hintz, A.~S. Barreto, G.~Uhlmann, and Y.~Zhang.
\newblock Inverse nonlinear scattering by a metric.
\newblock {\em arXiv preprint arXiv:2411.09671}, 2024.

\bibitem{hintz2022dirichlet}
P.~Hintz, G.~Uhlmann, and J.~Zhai.
\newblock The {D}irichlet-to-{N}eumann map for a semilinear wave equation on
  {L}orentzian manifolds.
\newblock {\em Communications in Partial Differential Equations},
  47(12):2363--2400, 2022.

\bibitem{hintz2022inverse}
P.~Hintz, G.~Uhlmann, and J.~Zhai.
\newblock An inverse boundary value problem for a semilinear wave equation on
  lorentzian manifolds.
\newblock {\em International Mathematics Research Notices},
  2022(17):13181--13211, 2022.

\bibitem{kachalov2001inverse}
A.~Kachalov, Y.~Kurylev, and M.~Lassas.
\newblock {\em Inverse boundary spectral problems}.
\newblock Chapman and Hall/CRC, 2001.

\bibitem{katchalov1998multidimensional}
A.~Katchalov and Y.~Kurylev.
\newblock Multidimensional inverse problem with incomplete boundary spectral
  data.
\newblock {\em Communications in Partial Differential Equations},
  23(1-2):27--59, 1998.

\bibitem{klingenberg1995riemannian}
W.~Klingenberg.
\newblock {\em Riemannian geometry}, volume~1.
\newblock Walter de Gruyter, 1995.

\bibitem{kurylev2022inverse}
Y.~Kurylev, M.~Lassas, L.~Oksanen, and G.~Uhlmann.
\newblock Inverse problem for {E}instein-scalar field equations.
\newblock {\em Duke Mathematical Journal}, 171(16):3215--3282, 2022.

\bibitem{kurylev2010rigidity}
Y.~Kurylev, M.~Lassas, and G.~Uhlmann.
\newblock Rigidity of broken geodesic flow and inverse problems.
\newblock {\em American Journal of Mathematics}, 132(2):529--562, 2010.

\bibitem{kurylev2018inverse}
Y.~Kurylev, M.~Lassas, and G.~Uhlmann.
\newblock Inverse problems for {L}orentzian manifolds and non-linear hyperbolic
  equations.
\newblock {\em Inventiones Mathematicae}, 212(3):781--857, 2018.

\bibitem{landau1960theory}
L.~Landau, E.~Lifshitz, J.~Sykes, W.~Reid, and E.~H. Dill.
\newblock Theory of elasticity: {V}ol. 7 of course of theoretical physics.
\newblock {\em Phys. Today}, 13:44, 1960.

\bibitem{lassas2021inverse}
M.~Lassas, T.~Liimatainen, Y.-H. Lin, and M.~Salo.
\newblock Inverse problems for elliptic equations with power type
  nonlinearities.
\newblock {\em Journal de Math{\'e}matiques Pures et Appliqu{\'e}es},
  145:44--82, 2021.

\bibitem{lassas2018inverse}
M.~Lassas, G.~Uhlmann, and Y.~Wang.
\newblock Inverse problems for semilinear wave equations on {L}orentzian
  manifolds.
\newblock {\em Communications in Mathematical Physics}, 360(2):555--609, 2018.

\bibitem{lionheart1997conformal}
W.~Lionheart.
\newblock Conformal uniqueness results in anisotropic electrical impedance
  imaging.
\newblock {\em Inverse problems}, 13(1):125--134, 1997.

\bibitem{michel1981rigidite}
R.~Michel.
\newblock Sur la rigidit{\'e} impos{\'e}e par la longueur des
  g{\'e}od{\'e}siques.
\newblock {\em Inventiones Mathematicae}, 65(1):71--83, 1981.

\bibitem{nursultanov2025determining}
M.~Nursultanov, L.~Oksanen, and L.~Tzou.
\newblock Determining lorentzian manifold from non-linear wave observation at a
  single point.
\newblock {\em Journal of Differential Equations}, 444:113563, 2025.

\bibitem{ogden1997non}
R.~W. Ogden.
\newblock {\em Non-linear elastic deformations}.
\newblock Courier Corporation, 1997.

\bibitem{paternain2015invariant}
G.~P. Paternain, M.~Salo, and G.~Uhlmann.
\newblock Invariant distributions, {B}eurling transforms and tensor tomography
  in higher dimensions.
\newblock {\em Mathematische Annalen}, 363(1-2):305--362, 2015.

\bibitem{rachele2000boundary}
L.~Rachele.
\newblock Boundary determination for an inverse problem in elastodynamics.
\newblock {\em Communications in Partial Differential Equations},
  25(11-12):1951--1996, 2000.

\bibitem{rachele2000inverse}
L.~Rachele.
\newblock An inverse problem in elastodynamics: uniqueness of the wave speeds
  in the interior.
\newblock {\em Journal of Differential Equations}, 162(2):300--325, 2000.

\bibitem{rachele2003uniqueness}
L.~Rachele.
\newblock Uniqueness of the density in an inverse problem for isotropic
  elastodynamics.
\newblock {\em Transactions of the American Mathematical Society},
  355(12):4781--4806, 2003.

\bibitem{sa2024recovery}
A.~S{\'a}~Barreto and P.~Stefanov.
\newblock Recovery of a general nonlinearity in the semilinear wave equation.
\newblock {\em Asymptotic Analysis}, 138(1-2):27--68, 2024.

\bibitem{sideris1996null}
T.~C. Sideris.
\newblock The null condition and global existence of nonlinear elastic waves.
\newblock {\em Inventiones Mathematicae}, 123(2):323--342, 1996.

\bibitem{sideris2000nonresonance}
T.~C. Sideris.
\newblock Nonresonance and global existence of prestressed nonlinear elastic
  waves.
\newblock {\em Annals of Mathematics}, 151(2):849--874, 2000.

\bibitem{stefanov2016boundary}
P.~Stefanov, G.~Uhlmann, and A.~Vasy.
\newblock Boundary rigidity with partial data.
\newblock {\em Journal of the American Mathematical Society}, 29(2):299--332,
  2016.

\bibitem{stefanov2017local}
P.~Stefanov, G.~Uhlmann, and A.~Vasy.
\newblock Local recovery of the compressional and shear speeds from the
  hyperbolic {DN} map.
\newblock {\em Inverse Problems}, 34(1):014003, 2017.

\bibitem{stefanov2018inverting}
P.~Stefanov, G.~Uhlmann, and A.~Vasy.
\newblock Inverting the local geodesic {X}-ray transform on tensors.
\newblock {\em Journal d'Analyse Mathematique}, 136(1):151--208, 2018.

\bibitem{stefanov2021local}
P.~Stefanov, G.~Uhlmann, and A.~Vasy.
\newblock Local and global boundary rigidity and the geodesic {X}-ray transform
  in the normal gauge.
\newblock {\em Annals of Mathematics}, 194(1):1--95, 2021.

\bibitem{uhlmann2020determination}
G.~Uhlmann and Y.~Wang.
\newblock Determination of space-time structures from gravitational
  perturbations.
\newblock {\em Communications on Pure and Applied Mathematics},
  73(6):1315--1367, 2020.

\bibitem{uhlmann2026inverse}
G.~Uhlmann, Y.~Yi, and J.~Zhai.
\newblock An inverse problem for compressible {E}uler's equations.
\newblock {\em arXiv preprint arXiv:2604.14636}, 2026.

\bibitem{uhlmann2021inverse}
G.~Uhlmann and J.~Zhai.
\newblock On an inverse boundary value problem for a nonlinear elastic wave
  equation.
\newblock {\em Journal de Math{\'e}matiques Pures et Appliqu{\'e}es},
  153:114--136, 2021.

\bibitem{uhlmann2024determination}
G.~Uhlmann and J.~Zhai.
\newblock Determination of the density in a nonlinear elastic wave equation.
\newblock {\em Mathematische Annalen}, 390(2):2825--2858, 2024.

\bibitem{uhlmann2024invertibility2}
G.~Uhlmann and J.~Zhai.
\newblock Invertibility of local geodesic transverse and mixed ray transforms
  {II}: higher order tensors.
\newblock {\em arXiv preprint arXiv:2402.12640}, 2024.

\bibitem{yi2026leakage}
Y.~Yi.
\newblock Leakage detection, collision relation, and self-adjoint cancellation
  in multi-velocity systems.
\newblock {\em arXiv preprint arXiv:2606.25218}, 2026.

\bibitem{zhai2026determination}
J.~Zhai.
\newblock Determination of the density in the linear elastic wave equation.
\newblock {\em Journal of Differential Equations}, 459:114064, 2026.

\end{thebibliography}

\end{document}